\documentclass[a4paper,10pt]{amsart}
\usepackage{amsmath,amssymb,amsthm,graphicx,amsfonts}
\usepackage{mathrsfs}
\usepackage{relsize}
\usepackage{subcaption}
\usepackage{float}
\usepackage{setspace}
\usepackage[all]{xy}

\usepackage[top=25mm,bottom=25mm,left=30mm,right=30mm]{geometry}
\usepackage{tikz}
\usetikzlibrary{patterns}
\usepackage{pxpgfmark}
\usepackage{multirow}
\usetikzlibrary{intersections,arrows,calc,decorations.markings}
\usepackage[colorlinks, linkcolor=blue, anchorcolor=blue, citecolor=green]{hyperref}
\tikzset{->-/.style={decoration={  markings,  mark=at position #1 with
			{\arrow{>}}},postaction={decorate}}}
\tikzset{-<-/.style={decoration={  markings,  mark=at position #1 with
			{\arrow{<}}},postaction={decorate}}}

\theoremstyle{plain}          
\newtheorem{theorem}{Theorem}[section]
\newtheorem{proposition}[theorem]{Proposition}
\newtheorem{lemma}[theorem]{Lemma}

\theoremstyle{definition}
\newtheorem{definition}[theorem]{Definition}

\newtheorem{construction}[theorem]{Construction}

\newtheorem{rk}[theorem]{Remark}
\newtheorem{ex}[theorem]{Example}
\newtheorem{lem}[theorem]{Lemma}

\newtheorem{thm}[theorem]{Theorem}
\newtheorem*{convention}{Convention}

\def\Tcal{\mathcal{T}}
\def\Bcal{\mathcal{B}}

\def\SS{\mathbf{S}}
\def\RR{\mathbf{R}}
\def\af{\mathbf{a}}

\def\Pf{\mathbf{P}}
\def\MM{\mathbf{M}}

\def\Tcal{\mathcal{T}}
\def\PS{\mathbb{P}(\mathbf{S})}
\def\<{\langle}
\def\>{\rangle}
\def\ba{\mathbf{a}}

\def\PP{\mathbf{P}}
\def\MM{\mathbf{M}}

\newcommand{\opname}[1]{\operatorname{#1}}

\newcommand{\id}{\mathbf{1}}
\newcommand{\Int}{\mathbf{Int}}
\newcommand{\Aut}{\mathrm{Aut}}
\newcommand{\Ext}{\mathrm{Ext}}

\newcommand{\mmod}{\mathrm{mod}}
\newcommand{\Intv}{\underline{\mathbf{Int}}}
\newcommand{\dimv}{\underline{\mathbf{dim}}}

\newcommand{\Hom}{\mathrm{Hom}}

\newcommand{\apfree}{\angle^-_{\mathbf{P-free}}}
\newcommand{\appfree}{\angle^-_{p,\mathbf{P-free}}}

\newcommand{\ha}{\angle^-}
\newcommand{\FD}{\mathbf{FD}}
\newcommand{\FPD}{\mathbf{FPD}}
\newcommand{\aad}{\mathbf{acwad}}
\newcommand{\cad}{\mathbf{cwad}}
\newcommand{\PD}{\mathbf{PD}}
\newcommand{\D}{\mathbf{D}}
\newcommand{\End}{\operatorname{End}\nolimits}

\newcommand{\proj}{\opname{proj}}

\newcommand{\nn}{node[black]{$\bullet$}}

\newcommand{\ie}{{\em i.e.}\ }
\newcommand{\confer}{{\em cf.}\ }

\makeatletter
\newcommand*\bigcdot{\mathpalette\bigcdot@{.5}}
\newcommand*\bigcdot@[2]{\mathbin{\vcenter{\hbox{\scalebox{#2}{$\m@th#1\bullet$}}}}}

 \title[Endomorphism algebras of tilting modules over gentle algebras]{Geometric models for endomorphism algebras of tilting modules over gentle algebras}
\thanks{Partially supported by the National Natural Science Foundation of China (Grant No. 12171397, 12471037)}
\author[Deng]{Difan Deng}
\address{Difan Deng\\Department of Mathematics\\
	Southwest Jiaotong University\\
	610031 Chengdu \\
	P.R.China}
\email{difandeng@my.swjtu.edu.cn}

\author[Geng]{Shengfei Geng}
\address{Shengfei Geng\\Department of Mathematics\\
Sichuan University\\
610064 Chengdu\\
P.R.China}
\email{genshengfei@scu.edu.cn}

\author[Liu]{Pin Liu}
\address{Pin Liu\\Department of Mathematics\\
	Southwest Jiaotong University\\
	610031 Chengdu \\
	P.R.China}
\email{pinliu@swjtu.edu.cn}

\begin{document}
\dedicatory{Dedicated to Professor Yanan Lin on the occasion of his 70th birthday}
\begin{abstract}
This paper investigates tilting modules over gentle algebras and their endomorphism algebras within the framework of marked surfaces and tilings introduced by Baur and Sim\~{o}es. Faithful dissections of a tiling are shown to correspond to tilting modules. For a faithful dissection, we define an auxiliary algebra and prove that it is isomorphic to the endomorphism algebra of the corresponding tilting module. We also construct a new tiling realizing this endomorphism algebra and introduce a  flip preserving tilting modules.
\end{abstract}

\maketitle	
\tableofcontents
\section{Introduction}\label{s:intro}

Tilting theory plays a fundamental role in the representation theory of finite-dimensional algebras. Its origins lie in the reflection functors introduced by Bernstein, Gelfand and Ponomarev \cite{BGP}, which were subsequently generalized by Auslander, Platzeck, and Reiten \cite{APR}. The axiomatic study of tilting modules was initiated by Brenner and Butler \cite{BB}, and the standard definition was later formulated by Happel and Ringel \cite{HR}. A fundamental result of Happel states that a tilting module over a finite-dimensional algebra induces a derived equivalence between the algebra and the endomorphism algebra of the tilting module. Thus, the structure of tilting modules and their endomorphism algebras is closely related to the study of derived categories and derived equivalences.


 Gentle algebras form an important class of finite-dimensional algebras arising naturally in surface theory, cluster theory, and homological mirror symmetry. Geometric models have proved particularly useful in the study of their representation theory; see, for instance, \cite{ABCJP,APS,BZ,CCS,CS,HKK,OPS,PPP}. 
 Baur and Sim\~{o}es \cite{BS} realized gentle algebras as tiling algebras and established a correspondence between indecomposable modules and permissible curves on marked surfaces equipped with tilings. This construction was extended to skew-gentle algebras by He, Zhou, and Zhu \cite{HZZ}, who introduced skew-tiling algebras and obtained a correspondence between tagged permissible curves and certain indecomposable modules. They also gave a geometric realization of $\tau$-tilting theory for skew-gentle algebras. The tilting theory of gentle algebras has also been studied by means of surface models. Chang \cite{Chang} proved that every almost complete $n$-tilting module over a gentle algebra is partial $n$-tilting and hence admits a complement to a $n$-tilting module. He also established a modified form of Happel's conjecture concerning the number of complements in the gentle setting. These results further demonstrate the usefulness of surface models in the study of tilting modules over gentle algebras.
 

The purpose of this paper is to give a geometric description of the endomorphism algebras of classical tilting modules over gentle algebras, using the geometric models introduced by Baur and Sim\~{o}es \cite{BS}. This approach differs from the models employed in \cite{Chang}. 

Fix a tiling $(\SS,\MM,\PP)$, and let $A_{\PP}$ be the associated gentle algebra. 
 A set of  permissible arcs in $(\SS,\MM,\PP)$  that do not intersect each other or themselves in the interior of $\SS$ will be called a {\it partial dissection} on $(\SS,\MM,\PP)$. A partial dissection is a {\it dissection} if it is maximal. By \cite{HZZ},  there is  a bijection between partial dissections and basic $\tau$-rigid $A$-modules which restricts to a bijection between dissections and basic support $\tau$-tilting $A$-modules. Furthermore, by \cite{Nie}, basic tilting $A_{\PP}$-modules correspond precisely to faithful dissections of $(\SS,\MM,\PP)$; see Theorem~\ref{t: tilting modules and faithful dissection}. 

Let $\RR$ be a partial dissection, and
for any two arcs $\gamma_1,\gamma_2\in \RR$, let $M(\gamma_1)$ and $M(\gamma_2)$ the associated indecomposable $\tau$-rigid $A_\PP$-modules. Using \cite[Proposition 3.19]{BS}, we find that the radical morphisms from  $M(\gamma_1)$ to $M(\gamma_2)$
are determined by the $\PP$-free negatively oriented angles from $\gamma_1$ to $\gamma_2$,  namely, those negatively oriented angles from $\gamma_1$ to $\gamma_2$ that are not divided by any arc in $\PP$ (see Subsection~\ref{ss:p-free negatively oriented angles}), which we simply denote by $\apfree(\gamma_1,\gamma_2)$. More precisely, we obtain the following result.
\begin{thm}\label{t:main resuls 1}(Theorem~\ref{t:basis of rad fromed by angles})
Let $\gamma_1,\gamma_2$ be two  permissible arcs on $(\SS,\MM,\PP)$ that do not intersect each other or themselves in the interior of $\SS$. Then 
$$\{f_{\angle\alpha}|\angle\alpha\in \apfree(\gamma_1,\gamma_2)\}$$
forms a basis of $\mathbf{rad}_{A_\PP}(M(\gamma_1),M(\gamma_2))$.
\end{thm}



Let $\RR$ be a faithful dissection. We associate to $\RR$ an algebra $B_{\RR}$, whose quiver has vertices indexed by the arcs in $\RR$ and arrows indexed by the indecomposable negatively oriented angles with sides in $\RR$---that is, negatively oriented angles between arcs of $\RR$ not subdivided by any arc in $\PP\cup\RR$---which we refer to as \emph{$\PP\RR$-$\mathbf{free}$ negatively oriented angles}.
The relations of $B_{\RR}$ are determined by the local configurations of these angles. Denote the corresponding tilting module associated with $\RR$ by $M(\RR)=\bigoplus_{\gamma\in\RR}M(\gamma).$
  Our main result identifies $B_\RR$ with the endomorphism algebra $\End_{A_{\PP}}M(\RR)$ of the tilting module $M(\RR)$.
\begin{thm}\label{t:main result}(Theorem~\ref{t:iso for end})
    Let $(\SS,\MM,\PP)$ be a tiling and let $\mathbf{R}$ be a faithful dissection of $(\SS,\MM,\PP)$. Then there is an isomorphism $$B_{\RR}\cong \End_{A_{\PP}}M(\mathbf{R}).$$
\end{thm}

Since $M(\RR)$ is a tilting $A_{\PP}$-module, it follows that $B_{\RR}$ is derived equivalent to $A_{\PP}$. We also study a combinatorial operation on faithful dissections, called the \emph{tilting flip}. Given a faithful dissection $\RR$ and an arc $\gamma\in\RR$, the flip $\mu_\gamma(\RR)$ is obtained by replacing $\gamma$ with a uniquely determined permissible arc. A result of \cite{Nie} gives a criterion under which this operation preserves tilting modules; see Theorem~\ref{t:tilting flip}. We provide a simple alternative proof of this criterion. Finally, for each faithful dissection $\RR$, we construct a tiling $(\SS_{\RR},\MM_{\RR},\PP_{\RR})$ such that its associated tiling algebra $A_{\PP_{\RR}}$ satisfies
\[
A_{\PP_{\RR}}\cong \End_{A_{\PP}}M(\RR).
\]
Thus, we obtain two methods to construct algebras derived equivalent to $A_\PP$ from a faithful dissection $\RR$. Furthermore, these constructions can be iterated indefinitely, yielding an infinite family of algebras all derived equivalent to $A_\PP$.

The paper is organized as follows. In Section~\ref{s: gentle algebra}, we recall the necessary background on gentle algebras, tiling algebras, permissible arcs, and tilting theory. In Section~\ref{ss:tilting modules and faithful dissection}, we study the correspondence between tilting modules and faithful dissections. Section~\ref{s: end algebra} is devoted to the endomorphism algebras associated with faithful dissections. We introduce $\PP$-$\mathbf{free}$ and $\PP\RR$-$\mathbf{free}$ angles, describe bases of radical morphism spaces in terms of such angles, and prove the isomorphism of Theorem~\ref{t:main result}. In Section~\ref{s:tilting flip}, We study tilting flips and give a simple alternative proof of the criterion for preserving tilting modules. Finally, in Section~\ref{s:geo model of endo algebra}, we construct a tiling model for the endomorphism algebra and illustrate the construction with examples.

\section{Geometric models of gentle algebras}\label{s: gentle algebra}
Throughout, $K$ denotes an algebraically closed field. Let $Q$ be a quiver with vertex set $Q_0=\{1,\dots,n\}$ and arrow set $Q_1$. For $\alpha\in Q_1$, let $s(\alpha)$ and $t(\alpha)$ denote its source and target, respectively. Paths are read from left to right. For a finite-dimensional $K$-algebra $A$, let $\mmod A$ denote the category of finite-dimensional right $A$-modules.
\subsection{Gentle algebras}
Following \cite{AS87}, a finite-dimensional $K$-algebra $A=KQ/\langle I\rangle$ is called gentle if it satisfies the following conditions:
    \begin{itemize}
        \item [(G1)] For any vertex $a\in Q_0$, there are at most two incoming arrows and two outgoing arrows at $a$.
        \item [(G2)]For any arrow $\alpha\in Q_1$, there is at most one arrow $\beta$ such that $\alpha\beta\notin I$, and at most one arrow $\gamma$ such that $\gamma\alpha\notin I$.
        \item[(G3)] For any arrow $\alpha\in Q_1$, there is at most one arrow $\delta$ in $Q_1$ such that $\alpha\delta\in I$, and there is at most one arrow $\theta$ such that $\theta\alpha\in I$.
 	    \item [(G4)] $I$ is generated exactly by a finite set of paths of length two.
  \end{itemize}
    For each arrow $\alpha \in Q_1$, we associate a formal inverse $\alpha^{-1}$, with $s(\alpha^{-1})=t(\alpha)$ and $t(\alpha^{-1}) = s(\alpha)$. The set of formal inverses of arrows in $Q_1$ is denoted by $Q_1^{-1}$. Elements of $Q_1\cup Q_1^{-1}$ are called \emph{letters}.
A word $s=\alpha_1\dots \alpha_l$ of letters is called a \textit{string} of length $l$ if $s(\alpha_{i+1})=t(\alpha_{i})$, no subword or its inverse belongs to $I$, and $\alpha_i\neq \alpha_{i+1}^{-1}$ for all $1\leq i\leq l-1$.
A \emph{band} $b=b_1\dots b_n$ is a string with $s(b)=t(b)$ such that every power $b^m$ is a string, but $b$ itself is not a proper power of any other string.
 


 Each string $\mu$ defines a unique string module $M(\mu)$. 
Precisely, for any string $\mu=\mu_{1}\cdots \mu_{l},$ there is a function $\phi:\{0,1,\dots,l\}\to Q_0$ defined by
		\begin{equation}
			\phi(i)=\begin{cases}
				t(\mu_{i}),&\text{if } i\neq 0,\\
				s(\mu)      ,&\text{if } i=0.
			\end{cases}
		\end{equation}
		For $a\in Q_0$, let $\mathcal{I}_a=\{i\mid \phi(i)=a\}\subseteq\{0,1,\dots,l\}$. The string module $M(\mu)$ is the representation of $Q$ defined as follows:
		\begin{itemize}
			\item [$\bullet$]  For a vertex $a$,  $M_a=\bigoplus
   _{i\in \mathcal{I}_a} K_i$, where $K_i\cong K$ for each $i\in \mathcal{I}_a$.
			\item [$\bullet$]  For an arrow $\alpha$, with $s(\alpha)=a$ and $t(\alpha)=b$, the matrix $$M_\alpha:\bigoplus_{i\in \mathcal{I}_a}K_i\to\bigoplus_{j\in \mathcal{I}_b} K_j$$ is given by
   \begin{equation}
				(M_{\alpha})_{i,j}=\begin{cases}
					\id_K,&\text{if } j=i+1\text{ and } \mu_j=\alpha,\\
					\id_K,&\text{if } i=j+1\text{ and } \mu_i=\alpha^{-1},\\
					0,&\text{otherwise.}
				\end{cases}
    \end{equation}
		\end{itemize}
Thus $\dim_K M_a=|\mathcal{I}_a|$ for all $a\in Q_0$, and $M(\mu)\cong M(\mu^{-1})$ for every string $\mu$.

Each band $b$ gives rise to a family of band modules $M(b,m,\varphi)$ with $m\geq 1$ and $\varphi\in\Aut (K^m)$. All string and band modules are indecomposable, and every
indecomposable $A$-module is either a string module or a band module \cite{BR97}.

Following \cite{PPP}, for an arbitrary string $\mu=\mu_1\dots \mu_l$, 
a substring $\rho = \mu_i \cdots \mu_j$ of $\mu$ is called:
\begin{itemize}
    \item a \textit{top substring} of $\mu$ if $i = 1$ or $\mu_{i-1} \in Q_1^{-1}$, and $j = l$ or $\mu_{j+1} \in Q_1$. 
    The set of top substrings of $\mu$ is denoted by $\mathcal{T}(\mu)$.
    
    \item a \textit{bottom substring} of $\mu$ if $i = 1$ or $\mu_{i-1} \in Q_1$, and $j = l$ or $\mu_{j+1} \in Q_1^{-1}$. 
    The set of bottom substrings of $\mu$ is denoted by $\mathcal{B}(\mu)$.
\end{itemize}

    
        
Note that any string $\mu$ is both its own top string and its own bottom string, and $\mu^{\pm1}\in\Tcal(\mu)$ and $\mu^{\pm1}\in\Bcal(\mu)$.
 
 \begin{proposition}\cite{CB}\label{p:dim pro two strings}
		Let $A$ be a gentle algebra, $\mu$, $\omega$ be two strings. Then  $$\dim_{K}\Hom_{A}(M(\mu),M(\omega))=|\{(\rho_1,\rho_2)\in \mathcal{T}(\mu)\times \mathcal{B}(\omega)\mid {\rho_1}=\rho_2^{\pm 1}\}|.$$
        Furthermore, $\{f_\rho\mid (\rho,\rho^{\pm1})\in\Tcal(\mu)\times\Bcal(\omega)\}$ is a basis of $\Hom (M(\mu),M(\omega))$.
	\end{proposition}

\subsection{Marked surfaces and tilings}\label{ss:tiling}
A \emph{marked surface} is a pair $(\SS,\MM)$, where $\SS$ is a compact oriented surface with nonempty boundary $\partial \SS$,  and $\MM\subset \partial \SS$ is a finite set of marked points on the boundary.  We assume that if $\SS$ is a disk, then $\MM$ has at least four marked points.

A connected component of $\partial \SS$ is called a \emph{boundary component} of $\SS$. A boundary component $B$ of $\SS$ is called \emph{unmarked} if $\MM\cap B=\emptyset$. A \emph{boundary segment} is the closure of a component of $\partial \SS\setminus \MM$. 

An arc in $(\SS,\MM)$ is a continuous map $\gamma:[0,1]\to \SS$ such that
\begin{itemize}
 \item $\gamma(0),\gamma(1)\in \MM$; 
 \item $\gamma(t)\in \SS\backslash \MM$ for $0<t<1$;
    \item $\gamma$ is neither null-homotopic nor homotopic to a boundary segment.
\end{itemize}
Arcs on $\SS$ are considered up to homotopy relative to their endpoints and up to inverse, where the inverse of an arc $\gamma$ on $\SS$ is defined as $\gamma^{-1}(t)= \gamma(1-t)$ for $t\in[0,1]$. An arc $\gamma$ is called a loop if both endpoints coincide. Denote by $\mathbb{C}(\SS)$ the set of arcs on $\SS$. 
We always consider arcs up to isotopy relative to their endpoints. For any arcs $\gamma_1, \gamma_2\in \mathbb{C}(\SS)$, whenever we consider their intersections, we always assume that they are representatives in their homotopy classes such that the number of their intersections is minimal.

Following \cite{BS}, a \emph{partial triangulation} $\Pf$ is a collection of arcs that do not intersect each other or themselves in the interior of $\SS$. 

Let $\Pf$ be a partial triangulation of a marked surface $(\SS,\MM)$. The partial triangulation $\Pf$ decomposes $\SS$ into regions, called \emph{tiles}. We consider the following types of tiles (see Figure \ref{f:t1-3} and Figure \ref{f:t4-5}):
\begin{enumerate}
	\item[(I)] monogons with exactly one unmarked boundary component and no punctures in their interiors;
	\item[(II)] digons with exactly one unmarked boundary component and no punctures in their interiors;
    \item [(III)]three-gons bounded by two boundaries and one arc in $\Pf$, whose interiors contain no unmarked boundary component of $\SS$;
    \item [(IV)] $m$-gons whose edges are arcs in $\Pf$ and one boundary segment, with no unmarked boundary components or punctures in their interiors; 
      \item[(V)] $m$-gons ($m\geq 3$) whose edges are arcs in $\Pf$, with no unmarked boundary components or punctures in their interiors;
\end{enumerate}
\begin{figure}[h]
\begin{minipage}[t]{0.15\linewidth} 
\begin{tikzpicture}[xscale=0.8,yscale=0.8]

\draw[thick,fill=black!20] (2,0.17)arc (250:290:3);

\fill(3,0) circle(2pt);

\draw[thick,color=blue] (3,0) .. controls (2,-1) and (2,-2) .. (3,-2);
\draw[thick,color=blue] (3,0) .. controls (4,-1) and (4,-2) .. (3,-2);

\draw[thick,fill=black!20] (3,-1.3) circle(8pt);
\node at (3,-3){Type (I)};

\end{tikzpicture}
\end{minipage}
\begin{minipage}[t]{0.16\linewidth} 
\centering
		\begin{tikzpicture}[xscale=0.8,yscale=0.8]
\draw[thick,fill=black!20] (2,0.17)arc (250:290:3);

\fill(3,0) circle(2pt);

\draw[thick,fill=black!20] (2,-2.17)arc (110:70:3);
\fill(3,-2) circle(2pt);
\draw[thick,fill=black!20] (3,-1) circle(8pt);
\draw[thick,color=blue] (3,0) .. controls (2,-1) and (2.5,-1.5) .. (3,-2);
\draw[thick,color=blue] (3,0) .. controls (4.5,-1) and (4,-1.5) .. (3,-2);
\node at (3,-3){Type (II)};

\end{tikzpicture}
\end{minipage}
\begin{minipage}[t]{0.15\linewidth} 
\centering
		\begin{tikzpicture}[xscale=0.8,yscale=0.8]
\draw[thick,fill=black!20] (2,0.17)arc (200:240:3);
\draw[thick,fill=black!20] (2,0.17)arc (340:300:3);
\draw[blue,thick] (0.7,-1.4) to (3.3,-1.4);	
\fill(2,0.17) circle(2pt);
\fill(0.7,-1.4) circle(2pt);
\fill(3.3,-1.4) circle(2pt);
\node at (2,-3){Type (III)};

\end{tikzpicture}
\end{minipage}
\caption{Basic tiles of type (I)-(III)}\label{f:t1-3}
\end{figure}
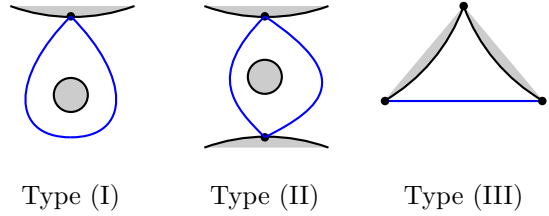

\begin{figure}[h]
\begin{minipage}[t]{0.25\linewidth} 
  \centering
		\begin{tikzpicture}[xscale=0.5,yscale=0.5]
					
	 \draw[thick,fill=black!40] (-1.3,-2.15)arc (120:60:1.5);
			
	\draw[blue,thick] (0,2) to (-1,2);
	\draw[blue,thick] (-1,2)\nn to (-2,1.5)\nn;
	\draw[blue,thick] (-2,1.5)\nn to (-2.7,0.5)\nn;
			
	\draw[blue,thick] (-1,-2)\nn to (-2,-1.5)\nn;
	\draw[blue,thick] (-2,-1.5)\nn to (-2.7,-0.5)\nn;
		
	\draw[blue,thick] (0,2)\nn to (1,1.5)\nn;
	\draw[blue,thick] (1,1.5)\nn to (1.7,0.5)\nn;
			
	\draw[blue,thick] (0,-2)\nn to (1,-1.5)\nn;
	\draw[blue,thick] (1,-1.5)\nn to (1.7,-0.5)\nn;
        \draw[blue](-2.7,-0.4)node[above]{$\vdots$}(1.7,-0.4)node[above]{$\vdots$};

	\node at (0,-3){ Type (IV)};
			\end{tikzpicture}
		
		\end{minipage}
\begin{minipage}[t]{0.2\linewidth} 
\centering
		\begin{tikzpicture}[xscale=0.5,yscale=0.5]

		\draw[blue,thick] (0,2) to (-1,2);
		\draw[blue,thick] (-1,2)\nn to (-2,1.5)\nn;
		\draw[blue,thick] (-2,1.5)\nn to (-2.7,0.5)\nn;
	        \draw[blue,thick] (0,-2)\nn to (-1,-2)\nn;
		\draw[blue,thick] (-1,-2)\nn to (-2,-1.5)\nn;
		\draw[blue,thick] (-2,-1.5)\nn to (-2.7,-0.5)\nn;
		
		\draw[blue,thick] (0,2)\nn to (1,1.5)\nn;
		\draw[blue,thick] (1,1.5)\nn to (1.7,0.5)\nn;
			
		\draw[blue,thick] (0,-2)\nn to (1,-1.5)\nn;
		\draw[blue,thick] (1,-1.5)\nn to (1.7,-0.5)\nn;

		\draw[blue](-2.7,-0.4)node[above]{$\vdots$}(1.7,-0.4)node[above]{$\vdots$};

	\node at (0,-3){ Type (V)};
			\end{tikzpicture}
			
	\end{minipage}	
		
    \caption{Basic tiles of type (IV)-(V)} \label{f:t4-5}
\end{figure}
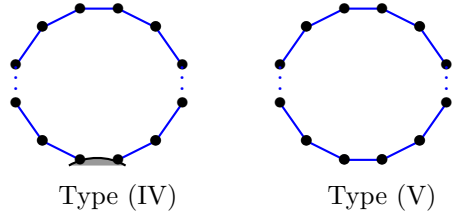

\begin{definition}[Tiling]\label{def:tiling}
A \emph{tiling} is a triple $(\SS, \MM, \Pf)$, where $(\SS, \MM)$ is a marked surface and $\Pf$ is a partial triangulation on $(\SS, \MM)$ whose tiles are all of type $\mathrm{(I)}$--$\mathrm{(V)}$.
\end{definition}


\subsection{Fans and angles on a marked surface} \label{ss: faithful dissection}
Let $(\SS,\MM)$ be a marked surface, 
$\mathcal{A}$ be a collection of arcs in $(\SS, \MM)$ (for instance, a partial triangulation $\PP$). 
\begin{definition}[Fans with respect to a collection of arcs]\label{def:fans}
For a marked point $p \in \MM$, let $m'$ and $m''$ be two points on the same boundary component as $p$ such that $m', m'' \notin \MM$ and $p$ is the unique marked point on the boundary segment $\delta$ between them. Let $c$ be a curve homotopic to $\delta$ that lies in a sufficiently small neighborhood of $p$ in which the arcs in $\mathcal{A}$ incident to $p$ are pairwise non-intersecting, except for its endpoints $m'$ and $m''$. The {\it complete fan} at $p$ with respect to $\mathcal{A}$ is the sequence of arcs $\gamma_1, \dots, \gamma_k \in \mathcal{A}$ intersected by $c$ in anticlockwise order. Any subsequence $\gamma_i, \dots, \gamma_j$ is called a {\it fan} (or {\it $\mathcal{A}$-fan}) at $p$. The {\it length} of the fan $\gamma_i, \dots, \gamma_j$ is defined as $j - i$. The arcs $\gamma_i$ and $\gamma_j$ are called the {\it sides} of the fan; moreover, $\gamma_i$ (resp.\ $\gamma_j$) is said to be the {\it start} (resp.\ {\it end}) of this fan. For $i < l < j$, the arc $\gamma_l$ is called an {\it interior arc} of the fan.
\end{definition}

\begin{definition}[(Oriented) angles]
Let $\alpha, \beta \in \mathcal{A}$ share a common endpoint $p$.
\begin{enumerate}
    \item An {\it (oriented) angle} (or {\it (oriented) $\mathcal{A}$-angle}) between $\alpha$ and $\beta$ at $p$ is a fan at $p$ of length at least $1$ with two sides $\alpha$ and $\beta$. The arcs $\alpha$ and $\beta$ are called the {\it sides} of the angle.
    
    \item A \textit{positive oriented angle} from $\alpha$ to $\beta$ at $p$ is an angle at $p$ with two sides $\alpha$ and $\beta$ such that $\beta$ is on the anti-clockwise side of $\alpha$.
    
    \item A \textit{negative oriented angle} from $\alpha$ to $\beta$ at $p$ is an angle at $p$ with two sides $\alpha$ and $\beta$ such that $\beta$ is on the clockwise side of $\alpha$.
    
    \item An angle is called \emph{indecomposable} if the length of its fan is $1$.
\end{enumerate}
\end{definition}

    

\begin{definition}[Addition of oriented angles]
Let $\angle\alpha$ and $\angle\beta$ be two oriented $\mathcal{A}$-angles. 
Suppose the fan of $\angle\alpha$ is the sequence $\gamma_1, \gamma_2, \dots, \gamma_s$ at a marked point $p_1$, 
and the fan of $\angle\beta$ is the sequence $\omega_1, \omega_2, \dots, \omega_t$ at a marked point $p_2$. 

We say $\angle\alpha$ and $\angle\beta$ are \emph{adjacent} if $p_1 = p_2$ and one of the following conditions holds:
\begin{itemize}
    \item[(C1)] $\omega_1 = \gamma_s$, and the concatenation $\gamma_1, \gamma_2, \dots, \gamma_s, \omega_2, \dots, \omega_t$ forms a fan at $p_1$;
    \item[(C2)] $\gamma_1 = \omega_t$, and the concatenation $\omega_1, \omega_2, \dots, \omega_t, \gamma_2, \dots, \gamma_s$ forms a fan at $p_1$.
\end{itemize}

If $\angle\alpha$ and $\angle\beta$ are adjacent, their \emph{sum}, denoted by $\angle\alpha + \angle\beta$, 
is the oriented angle whose fan is the concatenation given in (C1) or (C2), respectively.
Geometrically, the fan of $\angle\alpha + \angle\beta$ is the union of the fans of $\angle\alpha$ and $\angle\beta$.
\end{definition}

\subsection{Tiling algebras and gentle algebras}\label{s:def of tiling}
 Baur and Sim\~{o}es \cite{BS} introduced the finite-dimensional  \emph{tiling algebra} $A_{\Pf}$ associated with a tiling $(\SS,\MM,\Pf)$.
It is the quotient algebra $A_{\PP}=KQ_{\Pf}/\langle I_{\PP}\rangle$, where the quiver with relations $(Q_{\Pf}, I_{\Pf})$ is defined as follows:
\begin{itemize}
\item The vertex set $(Q_{\Pf})_0$ corresponds bijectively to the arcs in  $\Pf$;
\item An arrow $\alpha: i \to j$ exists in $(Q_{\Pf})_1$ whenever the arcs $i$ and $j$ share an endpoint $p_\alpha \in \MM$ and $j$ follows $i$ immediately in the anticlockwise direction around $p_\alpha$;
    \item The relation set $I_{\PP}$ consists of
    \begin{itemize}
        \item $\epsilon^2$, where $\epsilon$ is a loop;
        \item $\alpha\beta$ if $p_\alpha\neq p_\beta$, or the endpoints of the arc corresponding to $t(\alpha)=s(\beta)$ coincide ({\em cf.} Figure \ref{f:tiling-algebra}).
      
    \end{itemize}
      In other words, $\alpha\beta\in I_{\PP}$ if and only if either $p_{\alpha}\neq p_{\beta}$  or $p_{\alpha}=p_{\beta}$ and $t(\alpha)=s(\beta)$
corresponds to a loop arc.
\end{itemize}
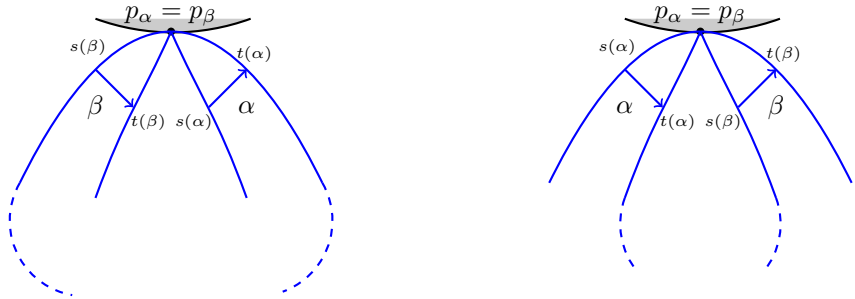
\begin{figure}[h]
\begin{tikzpicture}

\draw[thick,fill=black!20] (2,0.17)arc (250:290:3);
\node at (3,0.2){$p_\alpha=p_\beta$};
\fill(3,0) circle(1.5pt);
\draw[thick,color=blue] (3,0) parabola (1,-2);
\draw[thick,color=blue,dashed] (1,-2) arc (150:260:1);
\draw[thick,color=blue,->] (2,-0.5)--(2.5,-1);
\node at (2.7,-1.2){\tiny{$t(\beta)$}};
\node at (1.9,-0.2){\tiny{$s(\beta)$}};
\node at (2,-1) {$\beta$};
\draw[thick,color=blue] (3,0) .. controls (2.6,-0.9) and (2.4,-1.1) .. (2,-2.2);
\draw[thick,color=blue] (3,0) .. controls (3.4,-0.9) and (3.6,-1.1) .. (4,-2.2);
\draw[thick,color=blue] (3,0) parabola (5,-2);
\draw[thick,color=blue,dashed] (5,-2) arc (30:-70:1);
\draw[thick,color=blue,->] (3.5,-1)--(4,-0.5);
\node at (4,-1) {$\alpha$};
\node at (3.3,-1.2){\tiny{$s(\alpha)$}};
\node at (4.1,-0.3){\tiny{$t(\alpha)$}};

\draw[thick,fill=black!20] (9,0.17)arc (250:290:3);
\node at (10,0.2){$p_\alpha=p_\beta$};
\fill(10,0) circle(1.5pt);
\draw[thick,color=blue] (10,0) parabola (8,-2);
\draw[thick,color=blue,dashed] (9,-2.2) arc (160:220:1);
\draw[thick,color=blue,->] (9,-0.5)--(9.5,-1);
\node at (9,-1) {$\alpha$};
\node at (9.7,-1.2){\tiny{$t(\alpha)$}};
\node at (8.9,-0.2){\tiny{$s(\alpha)$}};
\draw[thick,color=blue] (10,0) .. controls (9.6,-0.9) and (9.4,-1.1) .. (9,-2.2);
\draw[thick,color=blue] (10,0) .. controls (10.4,-0.9) and (10.6,-1.1) .. (11,-2.2);
\draw[thick,color=blue] (10,0) parabola (12,-2);
\draw[thick,color=blue,dashed] (11,-2.2) arc (20:-40:1);
\draw[thick,color=blue,->] (10.5,-1)--(11,-0.5);
\node at (11,-1) {$\beta$};
\node at (10.3,-1.2){\tiny{$s(\beta)$}};
\node at (11.1,-0.3){\tiny{$t(\beta)$}};
\end{tikzpicture}
\caption{ Case $\alpha\beta=0$ when $p_\alpha=p_\beta$}\label{f:tiling-algebra}
\end{figure}
\begin{rk}
   Our convention for the orientation of $(Q_{\Pf})_1$ is opposite to that of \cite{BS}: for arcs $i,j$ sharing an endpoint $p_\alpha\in\MM$, we have $\alpha\colon i\to j$ when $j$ follows $i$ immediately in the anticlockwise direction around $p_\alpha$, whereas \cite{BS} takes the clockwise direction.
\end{rk}

\begin{rk}\label{r: bijection between fans and paths}
By the definition of $A_{\PP}$, one can get that there  is a natural bijection between the  $\PP$-fans of $(\SS,\MM,\PP)$ and paths of $A_\PP$. In particular, there is a bijection between indecomposable $\PP$-angles and arrows of $Q_\PP$. 
\end{rk}

\begin{theorem}[{\cite[Theorem 2.10]{BS}}]\label{t:tiling algebra} 
An algebra is gentle if and only if it is a tiling algebra.
\end{theorem}

\subsubsection{Permissible arcs} Let $(\SS,\MM,\PP)$ be a tiling and $\Delta$ a tile.
An {\it irreducible arc segment} in $\Delta$ is a curve $\eta: [0,1]\to \Delta$ such that  $\eta(0)$ and $\eta(1)$ lie on the edges of $\Delta$
and $\eta(t)(\text{for } 0<t<1)$ lies in the interior of $\Delta$. 
 An irreducible arc segment $\eta$ in $\Delta$ is called {\it permissible} with respect to $\PP$ if it satisfies one of the following conditions ({\em cf.} \cite[Definition 2.1]{HZZ} and compare \cite[Definition~3.1]{BS}).
	\begin{enumerate}
		\item[(P1)] One endpoint $P$ of $\eta$ is in $\MM$ and the other endpoint $Q$ lies in the interior of a non-boundary edge, say $\gamma$, of $\Delta$, such that 
        \begin{itemize}
            \item $\eta$ is not isotopic to a segment of an edge of $\Delta$ relative to its endpoints, and
            \item after moving $P$ along the edges of $\Delta$ in anticlockwise order to the next marked point, say $P'$, the new arc segment obtained from $\eta$ is isotopic to a segment of $\gamma$ relative to its endpoints.
        \end{itemize} See Figure~\ref{f:P1} for all possible cases of permissible irreducible arc segments satisfying (P1).
		
		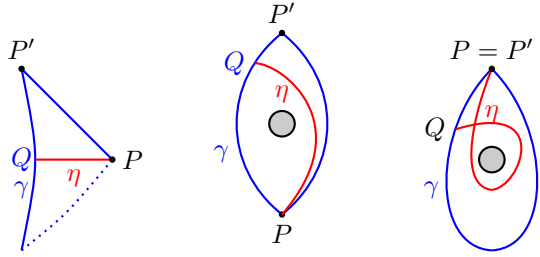
\begin{figure}[h]
			\begin{tikzpicture}[xscale=0.8,yscale=0.8]
				\draw[thick,color=blue] (1,0) .. controls (1.3,1.5) and (1.3,1.5) .. (1,3);
				\draw[thick,color=blue] (1,3)node[black,above]{$P'$} .. controls (2,2) and (2, 2) .. (2.5,1.5);
				\draw[thick,color=blue,dotted] (1,0) .. controls (1.75, 0.6) and (1.75,0.6) .. (2.5,1.5);
				\draw[thick,color=red] (2.5,1.5)node[black,right]{$P$} tonode[below]{$\eta$} (1.23,1.5);
				\fill(2.5,1.5) circle(1.5pt); \fill(1,3) circle(1.5pt);
				\draw[blue] (1,1)node{$\gamma$} (1,1.5)node[blue]{$Q$};	
			\end{tikzpicture}\qquad
			\begin{tikzpicture}[xscale=0.8,yscale=0.8]
				\draw[thick,color=blue] (7,0) .. controls (8,0.8) and (8,2.2) .. (7,3)node[black,above]{$P'$};	
				\draw[thick,color=blue] (7,0) .. controls (6,0.8) and (6,2.2) .. (7,3);	
				\fill(7,0) circle(1.5pt);		\fill(7,3) circle(1.5pt);	
				\draw[thick,fill=black!20] (7,1.5) circle(6pt);
				\draw[thick,color=red] (7,0)node[black,below]{$P$} .. controls (8.3,1.5) and (7,2.5) .. (6.55,2.5)node[blue,left]{$Q$};
				\draw[red] (7,2)node{$\eta$};
				\draw[blue] (6,1)node{$\gamma$};
			\end{tikzpicture}\qquad
		\begin{tikzpicture}[xscale=0.8,yscale=0.8]
				\draw[thick,color=blue] (13,0) .. controls (14,0.05) and (14,2) .. (13,3)node[black,above]{$P=P'$};
				\draw[thick,color=blue] (13,0) .. controls (12,0.05) and (12,2) .. (13,3);
				\draw[thick,color=red] (13,1) .. controls (12.5,1.05) and (12.6,2) .. (13,3);	
				\draw[thick,color=red] (13,1) .. controls (13.5,1.05) and (14,2.5) .. (12.4,2)node[black,left]{$Q$};	
				\draw[thick,fill=black!20] (13,1.5) circle(6pt);
				\fill(13,3) circle(1.5pt);
				\draw[red] (13,2.25)node{$\eta$};
				\draw[blue] (12,1)node{$\gamma$};
		\end{tikzpicture}
			\caption{Condition (P1)}\label{f:P1}
		\end{figure}
		
		\item[(P2)] The endpoints of $\eta$ are in the interiors of non-boundary edges $x,y$ (which are possibly not distinct) of $\Delta$ such that 
        \begin{itemize}
            \item $\eta$ has no self-intersections; 
            \item $x$ and $y$ have a common endpoint $p_\eta\in \MM$;
            \item $\eta$ cuts out an $\PP$-angle from $\Delta$ as shown in Figure \ref{f:PAS}. 
        \end{itemize}
        We denote by $\triangle(\eta)$ the local triangle cut out by $\eta$ and $\angle \eta$ the angle opposite to $\eta$ in the local triangle $\triangle(\eta)$.
		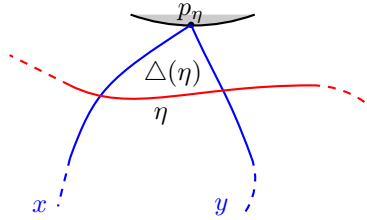
\begin{figure}[htpb]
			\begin{tikzpicture}[xscale=0.8,yscale=0.8]
				\draw[thick,fill=black!20] (9,0.17)arc (250:290:3);
				\node at (10,0.2){$p_{\eta}$};
				\fill(10,0) circle(1.5pt);
				\draw[thick,color=blue] (10,0) .. controls (8.6,-0.9) and (8.4,-1.1) .. (8,-2.2);
				\draw[thick,color=blue,dashed] (8,-2.2) -- (7.8,-3);
				\draw[thick,color=blue] (10,0) .. controls (10.4,-0.9) and (10.6,-1.1) .. (11,-2.2);
				\draw[thick,color=blue,dashed] (11,-2.2) arc (20:-40:1);
				\draw[thick,color=red] (8,-1) .. controls (9,-1.5) and (10,-1) .. (12,-1);
				\draw[thick,color=red,dashed] (12,-1) arc (90:50:1.5);
				\draw[thick,color=red,dashed] (8,-1) -- (7,-0.5);
				\node at (9.5, -1.5){$\eta$};
				\node at (7.5,-3){{\color{blue}$x$}};
				\node at (10.5, -3){{\color{blue}$y$}};
				\draw (9.7,-.8)node{$\triangle(\eta)$};
			\end{tikzpicture}
			\caption{Condition (P2)}\label{f:PAS}
		\end{figure}
	\end{enumerate}
	Every arc is assumed to be in minimal position with respect to $\PP$. The arc $\gamma$ is divided by $\PP$ into irreducible arc segments.
    \begin{definition}[{\cite[Definition 3.1]{BS} and \cite[Definition 2.2]{HZZ}}]
An arc $\gamma$ on $\SS$ is called permissible (with respect to $\PP$) if each irreducible arc segment of $\gamma$ obtained by cutting along $\PP$ is permissible.
 \end{definition}

For a permissible arc $\gamma$, choose an orientation of $\gamma$.
Suppose that $\gamma$ successively crosses
\[
\ba_1,\ba_2,\ldots,\ba_m\in\PP .
\]
Denote by
\[
\gamma^{(1)},\gamma^{(2)},\ldots,\gamma^{(m+1)}
\]
the corresponding arc segments of $\gamma$ obtained by cutting along $\PP$.
The first and last segments, $\gamma^{(1)}$ and $\gamma^{(m+1)}$, are called the
\emph{end arc segments} of $\gamma$.
Two end arc segments $\delta_1,\delta_2$ of permissible arcs are called
\emph{homotopic} if one endpoint of $\delta_1$ and $\delta_2$ is the same boundary marked point and their other endpoints lie on the same arc of $\PP$.

 We denote by $\PS\subset\mathbb{C}(\SS)$ the set of permissible arcs on $\SS$. 
 The following result provides a geometric model for gentle algebras.
\begin{theorem}[{\cite[Theorem 3.8]{BS}}]\label{t:permissible arcs and strings} 
 Let $(\SS,\MM,\Pf)$ be a tiling and $A_{\Pf}$ be its associated tiling algebra. Then there is a bijection between the  non-trivial permissible arcs in $(\SS,\MM,\Pf)$ and the non-zero strings of $A_{\Pf}$.
\end{theorem}
By this result, in the following, we identify the  nontrivial permissible arcs in $(\SS,\MM,\Pf)$ and the nonzero strings of $A_{\Pf}$.

\subsection{Support $\tau$-tilting modules and tilting modules}
Let $A$ be a finite-dimensional algebra, and let $\tau$ be the Auslander-Reiten translation of $\text{mod}A$. For any $M\in\text{mod}A$, denote by $|M|$ the number of pairwise non-isomorphic indecomposable summands of $M$.
\begin{definition} Let $M\in\mmod A$.
 \begin{enumerate}
    \item $M$ is called {\it rigid} if $\Ext^1_A(M,M)=0$;
    \item $M$ is called {\it partial tilting} if $M$ is rigid and $\proj.\dim M\le 1$;
    \item $M$ is called {\it tilting} if $M$ is rigid, $\proj.\dim M\le 1$ and $|M|=|A|$;
    \item $M$ is called {\it almost tilting} if $M$ is rigid, $\proj.\dim M\le 1$ and $|M|=|A|-1$;
    \item $M$ is called {\it $\tau$-rigid}, if $\Hom_{A}(M,\tau M)=0$;
    \item $M$ is called {\it $\tau$-tilting} if it is $\tau$-rigid and $|M|=|A|$;
    \item A {\it $\tau$-rigid} pair is a pair $(M,P)$ with $M \in\text{mod}A$ and $P$ a finitely generated projective $A$-module, such that $M$ is $\tau$-rigid and $\Hom_A(P,M) = 0$;
    \item A $\tau$-rigid pair $(M,P)$ is called a {\it $\tau$-tilting pair} if $|M| + |P| = |A|$. In this case, $M$ is called a {\it support $\tau$-tilting module}.
   
\end{enumerate}     
\end{definition}


Every $\tau$-rigid $A$-module is rigid. In the following, we denote by
\begin{itemize}
    \item $\opname{ind~\tau-rigid } A$: the set of isomorphism classes of indecomposable $\tau$-rigid $A$-modules;
    \item $\opname{\tau-rigid } A$: the set of isomorphism classes of basic $\tau$-rigid $A$-modules;
    \item $\opname{\tau-tilting}  A$: the set of isomorphism classes of basic $\tau$-tilting $A$-modules;
    \item $\opname{s\tau-tilt }  A$: the set of isomorphism classes of basic support $\tau$-tilting modules;
\end{itemize}
We recall the relation between faithfulness and tilting in $\tau$-tilting theory. This result will motivate the geometric notion of faithfulness introduced below.
\begin{proposition}[{\cite[Proposition 1.4]{AIR} and \cite[VIII.5.1]{ASS}}]\label{p:realtion between tau tilting module and tilting module}
    \begin{itemize}
        \item [(a)] Any faithful $\tau$-rigid $A$-module is a partial tilting $A$-module.
        \item [(b)] Any faithful $\tau$-tilting $A$-module is a tilting $A$-module.
    \end{itemize}
\end{proposition}
The following result motivates the use of surface models in the study of the endomorphism algebras of tilting modules over gentle algebras.
\begin{lemma}[{\cite[Theorem 1.1]{S}}]\label{l:endomorphism algebra is gentle}
	Let $A$ be a finite-dimensional gentle algebra over $K$ and ${M}$ a rigid $A$-module. Then the endomorphism algebra $\End_A (M)$ is a gentle algebra.
\end{lemma}

\subsection{Support $\tau$-tilting modules and dissections} Let $(\SS,\MM,\PP)$ be a tiling.
By \cite[Definition 3.3]{QZ17} that for any two arcs $\gamma, \delta$, their \emph{intersection number} $\Int(\gamma,\delta)$ is defined as $$\Int(\gamma,\delta)=|\{(t_1,t_2)\mid0<t_1,t_2<1,\gamma(t_1)=\delta(t_2)\}|.$$ 
Given an arc $\gamma$ in $(\SS,\MM,\PP)$, the {\it intersection vector} $\Intv_\PP(\gamma)$ of $\gamma$ with respect to $\PP$ is the vector $(\Intv(\gamma,\af_i))_{\af_i\in\PP}$. The {\it intersection number} $|\Int_{\PP}(\gamma)|$ of $\gamma$ with respect to $\PP$ is given by $\sum_{\af_i\in\PP} \Int(\gamma,\af_i)$

\begin{definition}
  Let $\mathbf{R}$ be a set of  permissible arcs, $\mathbf{R}$ is called a {\it partial dissection} if $$\Int(\gamma_1,\gamma_2)=0$$  for any $\gamma_1,\gamma_2\in\mathbf{R}$. A partial dissection $\mathbf{R}$ is called a {\it dissection} of $\SS$ if it is maximal.
\end{definition}
We denote by $\PD(\SS)$ the set of partial  dissections of $\SS$,
the set of  dissections of $\SS$ is denoted by $\D(\SS)$. The following result describes of support $\tau$-tilting modules of  gentle algebras.
\begin{lemma}[{\cite[Theorems 2.29, 3.12]{HZZ} and \cite[Lemma 2.4]{FGLZ}}]\label{l:bijection for tau rigid}
Let $A = KQ/\langle I\rangle $ be a finite-dimensional gentle algebra. Then there exists a tiling  $(\SS,\MM,\PP)$  such that:
\begin{itemize}
    \item There is a complete set of primitive orthogonal idempotents $\{e_{\af_i} \mid \af_i \in \PP\}$ of $A$ indexed by $\PP$.
    \item There is a bijection
\[
			M: \{\gamma\in\PS\mid \Int(\gamma,\gamma)=0\}\longrightarrow\opname{ind~} \tau \text{-}\operatorname{rigid} A,\quad 
			\gamma\mapsto M(\gamma)
\]
    satisfying $\Intv_{\PP}(\gamma) = \dimv M(\gamma)$.
    \item The map $M$ induces a bijection
    \[
			 M: \PD(\SS) \longrightarrow \tau\text{-}\operatorname{rigid} A,\quad
	\RR \mapsto \bigoplus_{\gamma \in\mathbf{R}} M(\gamma).
\]
Moreover, $M$ restricts to a bijection between $\D(\SS)$ and $\opname{s\tau-tilt }  A$.
\end{itemize}
\end{lemma}

\subsection{Tilting modules and faithful dissections} \label{ss:tilting modules and faithful dissection}
Let $(\SS,\MM,\PP)$ be a tiling, $\mathbf{R}$ be a partial dissection of $(\SS,\MM,\PP)$

\begin{definition}
A partial dissection $\mathbf{R}$  of $(\SS,\MM,\PP)$ is called  {\it faithful }  if  for every indecomposable  $\PP$-angle, there exists an arc in $\RR$ that cuts the angle such that the cut-off angle is contractible ({\em cf.} Figure \ref{f:PAS}). 
\end{definition}

We denote the set of all  faithful partial dissections by $\FPD(\SS)$. The set of \emph{faithful dissections} is $\FD(\SS) := \D(\SS) \cap \FPD(\SS)$.  
 The following result was also obtained in \cite{Nie} (in Chinese). We provide here a more rigorous and self-contained proof for the reader's convenience.
\begin{lemma}\label{l:PFD is faithful}
    If $\mathbf{R}$ is a faithful partial dissection on $(\SS,\MM,\PP)$, \ie $\mathbf{R}\in\FPD(\SS)$, then $M(\mathbf{R})$ is a faithful $A_{\PP}$-module. 
\end{lemma}

\begin{proof}

	It suffices to show that for every nonzero element $a\in A_{\PP}$, we have $M(\mathbf{R})\cdot a\ne 0$.
	
	First let $a=p_1=\alpha_1\cdots\alpha_s$ be a nonzero path in $A_\PP$.
	Write $s(\alpha_i)=\af_i$ for $1\leq i\leq s$  and $t(\alpha_s)=\af_{s+1}$. Then $\af_1,\cdots,\af_{s+1}$ is a $\PP$-fan at some marked point $p$.  Since $\mathbf{R}\in\FPD(\SS)$,  for each $1\leq i\leq s$, there exists an arc $\gamma_i\in\mathbf{R}$ that cuts $\angle\alpha_i$ and cuts off a contractible angle ({\em cf.} Figure \ref{f:non intersecting arc segment}). Because $\af_1,\cdots,\af_{s+1}$ is a $\RR$-fan and the arcs in $\mathbf{R}$ have no self-intersections and no intersections with each other, there exists an arc $\gamma\in\mathbf{R}$ crossing the fan $\af_1,\ldots,\af_{s+1}$ and cutting off a contractible angle.  Hence $\alpha_1\cdots\alpha_s$ is a substring of $\gamma$, so
	$M(\gamma)\cdot p_1\neq 0$. Hence $M(\mathbf{R})\cdot a\neq0$.
	
	Now let $a=\lambda_1p_1+\cdots+\lambda_tp_t\in A_{\PP}$, where $0\neq \lambda_i\in K$ and $p_1,\cdots,p_t$ are distinct nonzero paths in $A_{\PP}$ from a vertex $\af_c$ to a vertex $\af_d$ in $Q_0$.
	Let $\mathcal{B}_c$ (resp. $\mathcal{B}_d$) denote the basis of the vector space $M(\mathbf{R})_c$ (resp. $M(\mathbf{R})_d$) at the vertex $\af_c$ (resp. $\af_d$). Suppose $p_1=\alpha_1\cdots\alpha_s$. By the string module construction, $\mathcal{B}_c$ is determined by the intersections of $\RR$ with the arc $\af_c$.  By the preceding argument, there is an arc $\gamma\in\mathbf{R}$ such that $\alpha_1\cdots\alpha_s$ is a substring of $M(\gamma)$. Thus there is an element $x_c\in \mathcal{B}_c$ such that $x_cp_1\in \mathcal{B}_d$. Since the paths $p_1,\cdots,p_t$ are distinct and nonzero from  $\af_c$ to $\af_d$, for each $i\neq 1$, either $x_cp_i=0$ or $x_cp_i\in \mathcal{B}_d$. In the latter case, the string module structure ensures that $x_cp_i\neq x_cp_1$ when $i\neq 1$. Consequently, $x_ca\neq 0$. Therefore, $a$ is not an annihilator of $M(\mathbf{R})$.
\end{proof}

     \begin{figure}[htpb]
			\begin{tikzpicture}[scale=1]
				\draw[thick,fill=black!20] (9,0.17)arc (250:290:3);
				\node at (10,0.2){$p$};
				\fill(10,0) circle(1.5pt);
                \draw[thick,color=blue] (10,0) .. controls (6.5,-0.5) and (6.5,-2) .. (6,-3);
                \draw[thick,color=blue] (10,0) .. controls (12,-0.5) and (13,-2) .. (13.5,-3);
				\draw[thick,color=blue] (10,0) .. controls (8.6,-0.9) and (8.4,-1.1) .. (8,-2.2);
				\draw[thick,color=blue] (8,-2.2) -- (7.8,-3);
                 \draw[thick,color=blue] (10,0) .. controls (9.5,-0.9) and (9.7,-1.1) .. (9.5,-2.2);
                 \draw[thick,color=blue] (9.5,-2.2) --  (9.3,-3);
				\draw[thick,color=blue] (10,0) .. controls (10.4,-0.9) and (10.6,-1.1) .. (11,-2.2);
				\draw[thick,color=blue] (11,-2.2) arc (20:-40:1);
				\draw[thick,color=red,dashed] (8,-0.9) .. controls (9,-1.4)  .. (9.5,-1.5);
				  \draw[thick, color=red] (9.5, -1.5).. controls (10.5, -1.3) .. (11, -1.1);
                \draw[thick,color=blue,dashed] (11,-2.2) arc (20:-40:1);
				\draw[thick,color=red,] (7.5,-1.6) .. controls (8.5,-1.9)  .. (9.5,-2);
				  \draw[thick, color=red,dashed] (9.5,-2).. controls (10.4, -2.3) .. (11.2, -2);
                \draw[thick,color=blue,->,bend right=30] (9.7,-0.7) to(10.3,-0.7) ;
                \draw[thick,color=blue,->,bend right=30] (9,-0.7)to (9.7,-0.7) ;
				\draw (10.1,-1)node{\tiny{$\alpha_{i}$}};
                \draw [red](10.1,-1.3)node[below]{\tiny{$\gamma_{i}$}};
                \draw (9.2,-1)node{\tiny{$\alpha_{i-1}$}};
                 \draw [red](8.8,-2)node[below]{\tiny{$\gamma_{i-1}$}};
                 \draw [blue] (6.3,-3.2) node {\tiny{$\af_{1}$}};
                 \draw [blue] (13.3,-3.2) node {\tiny{$\af_{s+1}$}};
                \draw [blue] (9.3,-3.2) node {\tiny{$\af_{i}$}};
               \draw [blue] (7.8,-3.2) node {\tiny{$\af_{i-1}$}}; 
               \draw [blue] (11.3,-3.2)node {\tiny{$\af_{i+1}$}};
                  \draw [blue] (7.3,-2) node {$\cdots$};
                     \draw [blue] (12,-2) node {$\cdots$};
			\end{tikzpicture}
			\caption{}\label{f:non intersecting arc segment}
		\end{figure}
        
Combining Lemma \ref{l:PFD is faithful} with Lemma \ref{l:bijection for tau rigid}, we obtain the following result; see also \cite{Nie} (in Chinese). 
\begin{theorem}\label{t: tilting modules and faithful dissection}
	Let $(\SS,\MM,\PP)$ be a tiling and let $A_\Pf$ be its associated gentle algebra. Then $M(\mathbf{R})$ is a tilting $A_{\PP}$-module if and only if  $\mathbf{R}$ is a faithful dissection of $(\SS,\MM,\PP)$.
\end{theorem}
  
\section{Geometric models for endomorphism algebras of tilting modules}\label{s: end algebra}

\subsection{Morphisms and certain pairs of admissible segments}
We first recall the notion of admissible segments from \cite[Definition 3.18]{BS}.

Let $\gamma$ be a permissible arc. Write $\gamma$ as the concatenation of segments $\gamma = \gamma^- \widetilde{\gamma} \gamma^+$, and let $x$ (resp. $y$) be the connecting point of $\gamma^-$ (resp. $\gamma^+$) with $\widetilde{\gamma}$. Assume that $x$ lies in the interior of $\SS$ and on no arc of $\PP$. Let $\af_1, \af_2, \ldots,\af_k$ be the arcs in $\PP$ crossed by $\gamma$, in this order. Assume $|\Int_\PP (\widetilde{\gamma})| \neq 0$ and  let $\af_i, \ldots,\af_j$ be the arcs crossed by $\widetilde{\gamma}$.
\begin{enumerate}
\item The segment $\tilde{\gamma}$ is called an {\it anticlockwise admissible segment} if it satisfies the following two conditions:  
\begin{itemize}
\item Either $\af_i =\af_1$ or there is an arrow $\alpha\colon \af_{i} \rightarrow \af_{i-1}$ in $Q_1^{\PP}$;

\item Either $\af_j =\af_k$ or there is an arrow $\beta\colon \af_{j} \rightarrow \af_{j+1}$ in $Q_1^{\PP}$.
\end{itemize}
\item The notion of {\it clockwise admissible segment} is obtained from the above by reversing the orientations of the arrows $\alpha$ and $\beta$. 
\end{enumerate}
For a permissible arc $\gamma$, let $\aad(\gamma)$ and $\cad(\gamma)$ denote the sets of anticlockwise and clockwise admissible segments of $\gamma$, respectively. Figure~\ref{fig:admissiblesegment} illustrates this definition.

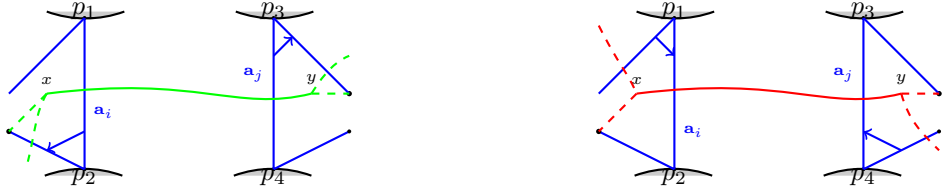
\begin{figure}[htpb]
   \centering
   \begin{minipage}[t]{0.48\linewidth}
    \centering
	\begin{tikzpicture}[scale=0.5]
	\draw[thick,fill=black!20] (1,0.17)arc (250:290:3);
    \node at (2,0.2){$p_1$};
    \fill(2,0) circle(1.5pt);
    \draw[thick,fill=black!20] (3,-4.17)arc (70:110:3);
    \node at (2,-4.2){$p_2$};
    \fill(2,-4) circle(1.5pt);
   \draw[blue,thick,->](2,-3)to  (1,-3.5);
   \draw[blue,thick,->] (7,-1) to (7.5,-.5);
    \draw[blue,thick](2,-4)--(2,0);
    \node [blue]at (2,-2.5)[right]{\tiny{$\af_i$}};
    \node [blue] at (7,-1.5)[left]{\tiny{$\af_j$}};
    \draw[blue,thick](2,-4)--(0,-3);
    \draw[blue,thick](2,0)--(0,-2);
    \fill(0,-3) circle(2pt);
  	\draw[thick,fill=black!20] (6,0.17)arc (250:290:3);
    \node at (7,0.2){$p_3$};
    \fill(7,0) circle(1.5pt);
    \fill(9,-2) circle(2pt);
    \draw[thick,fill=black!20] (8,-4.17)arc (70:110:3);
    \node at (7,-4.2){$p_4$};
    \fill(7,-4) circle(1.5pt);
    \draw[blue,thick](7,-4)--(7,0);
    \draw[blue,thick](7,-4)--(9,-3);
    \draw[blue,thick](7,0)--(9,-2);
    \fill(9,-3) circle(1.5pt);
    \draw[green,thick] (1,-2) .. controls (5,-1.5) and (6,-2.5) .. (8,-2);
    \draw[green,thick,dashed](1,-2)..controls(0.7,-2.5) and (0.7,-3)..(0.5,-3.8);
    \draw[green,thick,dashed](8,-2)..controls(8.2,-1.7) and (8.5,-1.2)..(9,-1);
    \draw[green,thick,dashed](1,-2)--(0,-3);
    \draw[green,thick,dashed](8,-2)--(9,-2);
    \node at(1,-2)[above]{\tiny{$x$}};
    \node at(8,-2)[above]{\tiny{$y$}};
	\end{tikzpicture}
    \end{minipage}
    \hfill
    \begin{minipage}[t]{0.48\linewidth}
    \centering
	\begin{tikzpicture}[scale=0.5]
	\draw[thick,fill=black!20] (1,0.17)arc (250:290:3);
    \node at (2,0.2){$p_1$};
    \fill(2,0) circle(1.5pt);
    \draw[thick,fill=black!20] (3,-4.17)arc (70:110:3);
    \node at (2,-4.2){$p_2$};
    \fill(2,-4) circle(1.5pt);
    \draw[blue,thick](2,-4)--(2,0);
    \node [blue]at (2,-3)[right]{\tiny{$\af_i$}};
    \node [blue] at (7,-1.5)[left]{\tiny{$\af_j$}};
    \draw[blue,thick](2,-4)--(0,-3);
    \draw[blue,thick](2,0)--(0,-2);
    \fill(0,-3) circle(2pt);
  	\draw[thick,fill=black!20] (6,0.17)arc (250:290:3);
    \node at (7,0.2){$p_3$};
    \fill(7,0) circle(1.5pt);
    \draw[thick,fill=black!20] (8,-4.17)arc (70:110:3);
    \node at (7,-4.2){$p_4$};
    \fill(7,-4) circle(1.5pt);
    \draw[blue,thick](7,-4)--(7,0);
    \fill(9,-2) circle(2pt);
     \draw[blue,thick](7,-4)--(9,-3);
     \draw[blue,thick](7,0)--(9,-2);
    \fill(9,-3) circle(1.5pt);
    \draw[red,thick] (1,-2) .. controls (5,-1.5) and (6,-2.5) .. (8,-2);
    \draw[red,thick,dashed](8,-2)..controls(8.2,-2.7) and (8.5,-3)..(9,-3.5);
    \draw[red,thick,dashed](1,-2)--(0,-3);
    \draw[red,thick,dashed](8,-2)--(9,-2);
     \node at(1,-2)[above]{\tiny{$x$}};
    \node at(8,-2)[above]{\tiny{$y$}};
     \draw[red,thick,dashed] (0,-.2) .. controls (0.5,-1.2) and (0.8,-1.5) .. (1,-2);
     \draw[blue,thick,->](1.5,-0.5)to (2,-1);
    \draw[blue,thick,->](8,-3.5) to (7,-3);
	\end{tikzpicture}
    \end{minipage}
\caption{{\small Left: $\widetilde{\gamma}$ is anticlockwise admissible. 
    Right: $\widetilde{\gamma}$ is clockwise admissible}
 }
    
\label{fig:admissiblesegment}
\end{figure}
\begin{rk}
 	Since our orientation convention for $(Q_{\Pf})_1$ is opposite to that of \cite{BS}, the two notions above are interchanged accordingly.
\end{rk}

Together with Proposition~\ref{p:dim pro two strings}, this gives the following geometric description of morphism spaces.
\begin{proposition}[{\cite[Proposition 3.19]{BS}}] \label{p:dim of hom}
The dimension of the $\Hom$-space $\Hom_{A_{\Pf}} (M(\gamma), M(\omega))$ equals the number of pairs $(\widetilde{\gamma}, ~\widetilde{\omega})$ of homotopic segments such that $$\Intv(\widetilde{\gamma}) = \Intv(\widetilde{\omega}),\quad \widetilde{\gamma}\in \aad(\gamma),\quad \widetilde{\omega}\in \cad(\omega).$$ Moreover, the set \[\{f_{\widetilde{\gamma}}\mid~(\widetilde{\gamma},\widetilde{\omega})\in\aad(\gamma)\times\cad(\omega)\text{~ with~} \Intv(\widetilde{\gamma})=\Intv(\widetilde{\omega})\text{~and~}\widetilde{\gamma}\sim\widetilde{\omega} \}\] is a basis of $\Hom_{A_{\Pf}}(M(\gamma),M(\omega)).$
\end{proposition}

\subsection{$\mathbf{P}$-free negatively oriented angles}\label{ss:p-free negatively oriented angles}

Let $(\SS,\MM,\PP)$ be a tiling and let $\mathbf{R}$ be a partial dissection on $(\SS,\MM,\PP)$.
\begin{convention}\label{conv:arcs}
Throughout the rest of this paper, unless otherwise specified, all fans and angles are taken with respect to $\mathcal{A} = \PP \cup \mathbf{R}$ as defined in Definition \ref{def:fans}.
\end{convention}

Let $\gamma_1, \gamma_2\in\mathbf{R}$  share a common endpoint $p$, 
and let $\angle\alpha$ be an angle between them at $p$.

\begin{definition}[$\PP$-$\mathbf{free}$ oriented angle]
An oriented angle $\angle\alpha$ is called \emph{$\PP$-$\mathbf{free}$} if its fan contains no interior arcs from $\PP$, \ie no arc in $\PP$ splits  $\angle\alpha$  into two (non-empty) angles.
\end{definition}

    

We introduce the following notation:
\begin{itemize}
    \item $\angle^{-}_p(\gamma_1,\gamma_2)$: the set of all negatively oriented angles from $\gamma_1$ to $\gamma_2$ at $p$;
    \item $\angle^{-}(\gamma_1,\gamma_2)$: the union of all $\angle^{-}_p(\gamma_1,\gamma_2)$ over all common endpoints $p$ of $\gamma_1$ and $\gamma_2$;
    \item $\appfree(\gamma_1,\gamma_2)$: the subset of $\angle^{-}_p(\gamma_1,\gamma_2)$ consisting of $\mathbf{P}$-$\mathbf{free}$ negatively oriented angles from $\gamma_1$ to $\gamma_2$ at $p$; 
    \item $\apfree(\gamma_1,\gamma_2)$: the union of all $\appfree(\gamma_1,\gamma_2)$ over all common endpoints $p$ of $\gamma_1$ and $\gamma_2$.
\end{itemize}
If $\gamma_1$ and $\gamma_2$ have no common endpoint, we set $$\angle^{-}(\gamma_1,\gamma_2)=\apfree(\gamma_1,\gamma_2)=\emptyset.$$

\begin{ex}\label{e:example for tiling and faithful dissection}
   Let $(\SS,\MM,\PP)$ be a tiling and $\mathbf{R}$ a faithful dissection as shown in the Figure \ref{f:example for A_R}. In this example, $\angle\alpha$ and $\angle \beta$ are $\mathbf{P}$-$\mathbf{free}$, while $\angle \delta$ is not $\mathbf{P}$-$\mathbf{free}$. Moreover, $\angle\alpha\in \angle_{p_{\alpha},\mathbf{P}-\mathbf{free}}^-(\gamma_1,\gamma_2)$, $\angle\beta\in \angle_{p_{\beta},\mathbf{P-free}}^-(\gamma_3,\gamma_1)$.
   \begin{figure}[h]
       \centering
      \begin{tikzpicture}[scale=0.5]
          	\draw[ thick] (0,0) circle (3);
            \fill(0,-3) circle(2pt);
            \fill(0,3) circle(2pt);
            \fill(-2.2,-2) circle(2pt);
             \fill(-2.2,2) circle(2pt);
            \fill(2.2,-2) circle(2pt);
             \fill(2.2,2) circle(2pt);
            \draw[blue,thick](0,3)to(0,-3);
             \draw[blue,thick](0,3)to(-2.2,-2);
             \draw[blue,thick](0,3)to(2.2,-2);
             \node [blue]at (-0.9,0.5){\tiny{$\af_1$}};
             \node at (-2.2,2) [left] {\tiny{$p_\alpha$}};
             \node at (2.2,2) [right] {\tiny{$p_\beta$}};
             \node [blue]at (-0.2,0){\tiny{$\af_2$}};
             \draw[red,thick, bend left=30,->](-1,2) to(-1.9,1.3);
             \node[red] at(-1,1.5){\tiny{$\angle\alpha$}};
             \node[red] at(0.8,1.5){\tiny{$\angle\beta$}};
              \draw[red,thick, bend left=30,->] (1.9,1.3) to(1,2);
            \draw[red,thick, bend right=-30,->] (-0.6,-1.8)to (0.6,-1.8) ;
            \node[red] at(0,-1.8)[above]{\tiny{$\angle\delta$}};
             \node [blue]at (0.9,0.5){\tiny{$\af_3$}};
             \draw[red,thick](-2.2,2)to(2.2,2);
             \draw[red,thick](-2.2,2)to(0,-3);
             \draw[red,thick](0,-3)to(2.2,2);
             \node[red]at (-1,2.2){\tiny{$\gamma_1$}};
             \node[red]at (-2.2,1){\tiny{$\gamma_2$}};
             \node[red]at (2,0.6){\tiny{$\gamma_3$}};
      \end{tikzpicture}
       \caption{$\mathbf{P}$-$\mathbf{free}$ oriented angles}
       \label{f:p-free angle}
   \end{figure}
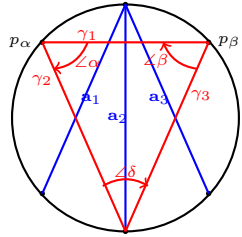
\end{ex}



\subsection{Radical morphisms and $\mathbf{P}$-free negatively oriented angles}
Let $(\SS,\MM,\PP)$ be a tiling and let $\mathbf{R}$ be a partial dissection on $(\SS,\MM,\PP)$. Let $\gamma_1, \gamma_2\in\mathbf{R}$  share a common endpoint $p$, and 
let $\angle\alpha$ be a $\PP$-$\mathbf{free}$ negatively oriented angle from $\gamma_1$ to $\gamma_2$ at $p$. One can choose orientations of $\gamma_1$ and $\gamma_2$ such that the first arc segments $\gamma_1^{(1)}$ and $\gamma_2^{(1)}$ are homotopic (If $\gamma_2 = \gamma_1$, we orient $\gamma_2$ such that $\gamma_2 = \gamma_1^{-1}$). Let $k$ be the minimal integer such that the $k$-th arc segments $\gamma_1^{(k)}$ and $\gamma_2^{(k)}$ are not homotopic. It is clear that $k\ge 2$ since  $\gamma_1^{(1)}$ and $\gamma_2^{(1)}$ are homotopic. Thus the concatenations $$\gamma_1^{(1)} \circ \cdots \circ \gamma_1^{(k-1)}\text{ and }\gamma_2^{(1)} \circ \cdots \circ \gamma_2^{(k-1)}$$ are homotopic.
Since the intersection is minimal and $\Int(\gamma_1,\gamma_2)=0$,
the possible configuration of $\gamma_1$ and $\gamma_2$ is shown in Figure \ref{f:the location for 1,2}. 

From this configuration, one obtains an anticlockwise admissible arc segment $\widetilde{\gamma_1}$ of $\gamma_1$ and a clockwise admissible arc segment $\widetilde{\gamma_2}$ of $\gamma_2$ such that $$\Intv_{\PP}(\widetilde{\gamma_1}) = \Intv_{\PP}(\widetilde{\gamma_2}).$$ These homotopic segments satisfy the conditions of Proposition \ref{p:dim of hom} and therefore induce a non-zero homomorphism from $M(\gamma_1)$ to $M(\gamma_2)$. Note that this morphism is uniquely  determined by $\angle\alpha$.
      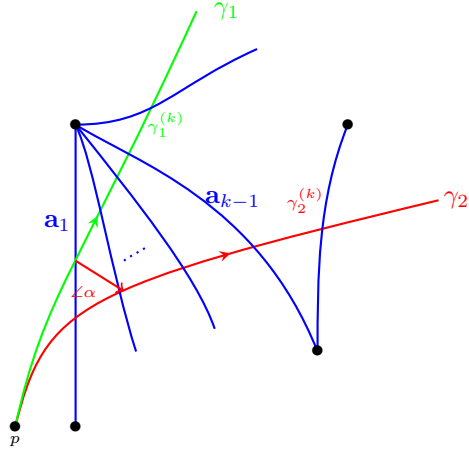
\begin{figure}[h]
			\begin{tikzpicture}[xscale=0.8]
				\draw[thick,blue] (5,0) .. controls (5,1) and (5,2) .. (5,4);
				\draw[thick,red,->-=.58,>=stealth] (4,0) .. controls (4.5,1.6) and (5,1.85) .. (11,3);
                \draw[green,thick,->-=.5,>=stealth] (4,0)..controls (4.5,1.8) and (5,2)..(7,5.5);
				\draw[green] (7.5,5.5)node{$\gamma_1$};
				\draw[thick,blue] (5,4) .. controls (5.3,3.5) and (5.6,2) .. (6,1);
				\draw[thick,blue] (5,4) .. controls (6,3) and (7,2) .. (7.3,1.3);
				\draw[thick,blue] (5,4) .. controls (6,3.5) and (8,3) .. (9,1);
                \draw[blue,thick] (5,4)..controls (6.3,4) and (6.6,4.5).. (8,5);
				\draw[thick,blue] (9,1) .. controls (9,2) and (9,3) .. (9.5,4)\nn;
				\draw[thick,blue,dotted] (5.8,2.2) .. controls (5.9,2.27) and (6.1,2.33) .. (6.2,2.4);
				\draw (9,1)\nn (4,0)\nn (5,0)\nn (5,4)\nn;
				\node at (4,0)[below]{\tiny{$p$}};
                \node[green] at (6.5,4){\tiny{$\gamma_1^{(k)}$}};
                \node[red] at (8.8,3){\tiny{$\gamma_2^{(k)}$}};
				\node[blue] at (4.7,2.7){$\af_1$};
				\node[blue] at (7.6,3){$\af_{k-1}$};
				\node[red] at (11.3,3){$\gamma_2$};
                \draw[red,thick,->] (5,2.2) to (5.8,1.8);
                \node[red] at (5.5,1.8)[left] {\tiny $\angle\alpha$};
			\end{tikzpicture}
			\caption{The location of $\gamma_1,\gamma_2$}\label{f:the location for 1,2}
		\end{figure}
        
For a $\PP$-$\mathbf{free}$ negatively oriented angle $\angle\alpha$ from $\gamma_1$ to $\gamma_2$ at $p$, we denote by $f_{\angle\alpha}$ the  morphism from $M(\gamma_1)$ to $M(\gamma_2)$ induced by $\angle\alpha$. More precisely, there exists a unique pair of arc segments $(\widetilde{\gamma_{1}},~\widetilde{\gamma_{2}})$ and orientations of  $\widetilde{\gamma_{1}},~\widetilde{\gamma_{2}}$  such that 
\[(\widetilde{\gamma_{1}},~\widetilde{\gamma_{2}})\in\aad(\gamma_1)\times\cad(\gamma_2), ~\Intv(\widetilde{\gamma_{1}})=\Intv(\widetilde{\gamma_{2}}), ~\widetilde{\gamma_{1}}\sim\widetilde{\gamma_{2}}, ~\widetilde{\gamma_{1}}(0)=\widetilde{\gamma_{2}}(0)=p.\]
This pair induces $f_{\angle\alpha}$.
\begin{rk}
   \begin{enumerate}
   \item  Since $\angle\alpha$ is a fan of length at least one, we must have $(\widetilde{\gamma_{1}}, \widetilde{\gamma_{2}})\neq  (\gamma_1,\gamma_2)$.
       \item If $\gamma_2\neq \gamma_1$, it is clear that $f_{\angle\alpha}\in\mathbf{rad} (M(\gamma_1),M(\gamma_2))$, while in the case $\gamma_2=\gamma_1$, since $\angle\alpha$ is a fan of length at least one, then one can get $(\widetilde{\gamma_{1}}, \widetilde{\gamma_{2}})\neq  (\gamma_1,\gamma_1)$, so we also have 
 $f_{\angle\alpha}\in\mathbf{rad} (M(\gamma_1),M(\gamma_1))$. 
\item Let $\widetilde{\gamma_{11}},~\widetilde{\gamma_{21}}$ be the sub-segments of $\widetilde{\gamma_{1}},~\widetilde{\gamma_{2}}$ respectively such that $~\Intv(\widetilde{\gamma_{11}})=\Intv(\widetilde{\gamma_{21}})<~\Intv(\widetilde{\gamma_{1}}) $, $\widetilde{\gamma_{11}}\sim~\widetilde{\gamma_{21}}$,  $\widetilde{\gamma_{11}}(0)=\widetilde{\gamma_{21}}(0)=p$, then it is easy to get that  $$\widetilde{\gamma_{11}}\in \aad(\gamma_1) \text{ iff  }\widetilde{\gamma_{21}}\in \aad(\gamma_2),\ \   \widetilde{\gamma_{11}}\in \cad(\gamma_1) \text{ iff }\widetilde{\gamma_{21}}\in \cad(\gamma_2).$$ 
   \end{enumerate} 
\end{rk}

\begin{lemma}\label{l:mor determined by angle}
Let $\angle\alpha\in \angle^{-}(\gamma_1,\gamma_2)$  and $\angle \beta\in \angle^{-}(\gamma_1,\gamma_2)$ with $\angle\alpha\neq\angle \beta$. Then $f_{\angle\alpha}\neq f_{\angle \beta}$. 
\end{lemma}
\begin{proof}
Denote the endpoint of $\angle\alpha$ (resp. $\angle\beta$) by $p_\alpha$ (resp.  $p_\beta$).
   Suppose that  $f_{\angle\alpha}$ is induced by the pair $(\widetilde{\gamma_{1\alpha}},~\widetilde{\gamma_{2\alpha}})$, where $\widetilde{\gamma_{1\alpha}}$ is an arc segment of $\gamma_1$,  $\widetilde{\gamma_{2\alpha}}$ is an arc segment of $\gamma_2$, and there exist orientations of  $\widetilde{\gamma_{1\alpha}},~\widetilde{\gamma_{2\alpha}}$  such that 
\[(\widetilde{\gamma_{1\alpha}},~\widetilde{\gamma_{2\alpha}})\in\aad(\gamma_1)\times\cad(\gamma_2), ~\Intv(\widetilde{\gamma_{1\alpha}})=\Intv(\widetilde{\gamma_{2\alpha}}), ~\widetilde{\gamma_{1\alpha}}\sim\widetilde{\gamma_{2\alpha}}, ~\widetilde{\gamma_{1\alpha}}(0)=\widetilde{\gamma_{2\alpha}}(0)=p_{\alpha}.\]

Similarly, suppose $f_{\angle\beta}$ is induced by the  pair $(\widetilde{\gamma_{1\beta}},~\widetilde{\gamma_{2\beta}})$ with $\widetilde{\gamma_{1\beta}}(0)=\widetilde{\gamma_{2\beta}}(0)=p_{\beta}.$

If $f_{\angle\alpha}=f_{\angle\beta}$, then we must have
$(\widetilde{\gamma_{1\alpha}},~\widetilde{\gamma_{2\alpha}})=(\widetilde{\gamma_{1\beta}},~\widetilde{\gamma_{2\beta}})$.
Hence $p_{\alpha}=p_{\beta}$ and $\angle\alpha=\angle\beta$, a contradiction.
\end{proof}

\begin{thm}\label{t:basis of rad fromed by angles}
Let $\gamma_1,\gamma_2\in\mathbf{R}$. Then the set
$$\{f_{\angle\alpha}\mid\angle\alpha\in \apfree(\gamma_1,\gamma_2)\}$$
forms a basis of $\mathbf{rad}_{A_\PP}(M(\gamma_1),M(\gamma_2))$.
\end{thm}

\begin{proof}
Denote by
 $$\mathcal{B}=\{f_{\widetilde{\gamma_1}}\mid~(\widetilde{\gamma_1},\widetilde{\gamma_2})\in\aad(\gamma_1)\times\cad(\gamma_2), \Intv(\widetilde{\gamma_1})=\Intv(\widetilde{\gamma_2}),\widetilde{\gamma_1}\sim\widetilde{\gamma_2}, (\widetilde{\gamma_1},\widetilde{\gamma_2})\neq {(\gamma_1,\gamma_2)}\}.$$
By Proposition \ref{p:dim of hom}, $\mathcal{B}$ is a basis of $\mathbf{rad}_{A_\PP}(M(\gamma_1),M(\gamma_2))$.
Denote by $$\mathcal{B}'=\{f_{\angle\alpha}\mid\angle\alpha\in \apfree(\gamma_1,\gamma_2)\}.$$ 
By Lemma~\ref{l:mor determined by angle}, we have $$\mathcal{B}'\subset\mathcal{B}.$$

Now let $f\in \mathcal{B}$. Then there exists a pair of homotopic segments $(\widetilde{\gamma}_1, \widetilde{\gamma}_2)$ such that 
 \[(\widetilde{\gamma_{1}},~\widetilde{\gamma_{2}})\in\aad(\gamma_1)\times\cad(\gamma_2), ~\Intv(\widetilde{\gamma_{1}})=\Intv(\widetilde{\gamma_{2}}), ~\widetilde{\gamma_{1}}\sim\widetilde{\gamma_{2}}, (\widetilde{\gamma_1},\widetilde{\gamma_2})\neq {(\gamma_1,\gamma_2)}.\]

Since $\widetilde{\gamma_{1}}\sim \widetilde{\gamma_{2}}$, let $\ba_i, \dots, \ba_j$ be the common sequence of arcs crossed by both $\widetilde{\gamma_{1}}$ and $\widetilde{\gamma_{2}}$. Write the complete ordered sequence of arcs in $\PP$ crossed by $\gamma_1$ as $$b_1, \dots, b_s, \ba_i, \dots, \ba_j, b_{s+1}, \dots, b_{s+k},$$ where $s \ge 0$ is the length of the head (initial part) and $k \ge 0$ is the length of the tail (terminal part) of $\gamma_1$. Similarly, write the complete ordered sequence of arcs in $\PP$ crossed by $\gamma_2$ as $$c_1, \dots, c_t, \ba_i, \dots, \ba_j, c_{t+1}, \dots, c_{t+l},$$ where $t \ge 0$ and $l \ge 0$ are the lengths of the head and tail of $\gamma_2$, respectively.

According to whether the head and tail are empty, we classify the configurations of $\widetilde{\gamma_{1}}$ into four types:
\begin{enumerate}
   \item $\widetilde{\gamma_{1}}^{(--)}$: both head and tail are empty (\ie $s = 0$ and $k = 0$);
     \item $\widetilde{\gamma_{1}}^{(-+)}$: the head is empty, and the tail is non-empty (\ie $s = 0$ and $k > 0$);
    \item $\widetilde{\gamma_{1}}^{(+-)}$: the head is non-empty, and the tail is empty (\ie $s > 0$ and $k = 0$);
   
  \item $\widetilde{\gamma_{1}}^{(++)}$: both head and tail are non-empty (\ie $s > 0$ and $k > 0$).
\end{enumerate}
Similarly,  we classify the configurations of $\widetilde{\gamma_{2}}$ into four types:
\begin{enumerate}
 \item $\widetilde{\gamma_{2}}^{(--)}$: both head and tail are empty (\ie $t = 0$ and $l = 0$);
 \item $\widetilde{\gamma_{2}}^{(-+)}$: the head is empty, and the tail is non-empty (\ie $t = 0$ and $l > 0$);
 \item $\widetilde{\gamma_{2}}^{(+-)}$: the head is non-empty, and the tail is empty (\ie $t > 0$ and $l = 0$);
   \item $\widetilde{\gamma_{2}}^{(++)}$: both head and tail are non-empty ( \ie $t > 0$ and $l > 0$). 
\end{enumerate}

Since $\widetilde{\gamma_{1}}$ is anticlockwise, 
$\widetilde{\gamma_{2}}$ is clockwise, 
all these possible configurations are depicted in Figure \ref{f:location for 1,2}.
 
   \begin{figure}[htpb]
   \centering
   \begin{minipage}[t]{0.48\linewidth}
    \centering
	\begin{tikzpicture}[scale=0.5]
	\draw[thick,fill=black!20] (1,0.17)arc (250:290:3);
    \node at (2,0.2){\tiny{$p_1$}};
    \fill(2,0) circle(1.5pt);
    \draw[thick,fill=black!20] (3,-4.17)arc (70:110:3);
    \node at (2,-4.2){\tiny{$p_2$}};
    \node [blue]at (0.5,-3.5)[below]{\tiny{$b_{s}$}};
    \fill(2,-4) circle(1.5pt);
    \draw[blue,thick](2,-4)--(2,0);
    \node [blue]at (2,-2)[left]{\tiny{$\af_i$}};
    \draw[blue,thick](2,-4)--(0,-3);
    \draw[blue,thick](2,0)--(0,-2);
    \node [blue]at (1,-1)[above]{\tiny{$c_t$}};
    \node [blue]at (8.5,-3.7){\tiny{$c_{t+1}$}};
    \fill(0,-3) circle(2pt);
  	\draw[thick,fill=black!20] (6,0.17)arc (250:290:3);
    \node at (7,0.2){\tiny{$p_3$}};
    \fill(7,0) circle(1.5pt);
    \fill(9,-2) circle(2pt);
    \draw[thick,fill=black!20] (8,-4.17)arc (70:110:3);
    \node at (7,-4.2){\tiny{$p_4$}};
    \fill(7,-4) circle(1.5pt);
    \draw[blue,thick](7,-4)--(7,0);
    \node [blue]at (7,-3)[right]{\tiny{$\af_j$}};
    \draw[blue,thick](7,-4)--(9,-3);
    \draw[blue,thick](7,0)--(9,-2);
    \fill(9,-3) circle(1.5pt);
    \draw[blue,thick,->](2,-3) to (1,-3.5);
     \draw[blue,thick,->] (7,-1) to (7.5,-.5);
    \node [green]at (10,-0.5){\tiny{$\widetilde{\gamma_1}^{(-+)}$}};
    \node [green]at (5,-3.5){\tiny{$\widetilde{\gamma_1}^{(+-)}$}};
    \node [green]at (10,-1){\tiny{$\widetilde{\gamma_1}^{(++)}$}};
     \node [green]at (6,-2){\tiny{$\widetilde{\gamma_1}^{(--)}$}};
    \node [blue] at (7.5,-0.5)[right]{\tiny{$b_{s+1}$}};
    \draw[green,thick] (0,-3) .. controls (5,-1.5) and (8,-1) .. (9,-0.5);
    \draw[green,thick] (0,-3.5) .. controls (6,-1.5) and (8,-1.5) .. (9,-1);
    \draw[green,thick] (0,-3.75) .. controls (7,-2.75) and (8,-2.25) .. (9,-2);
    \draw[green,thick] (0,-3)-- (9,-2);
    \node at (0,-3)[left]{\tiny{$q_1$}};
    \node at (9,-2)[right]{\tiny{$q_2$}};
	\end{tikzpicture}
    \end{minipage}
    \hfill
    \begin{minipage}[t]{0.48\linewidth}
    \centering
	\begin{tikzpicture}[scale=0.5]
	\draw[thick,fill=black!20] (1,0.17)arc (250:290:3);
    \node at (2,0.2){\tiny{$p_1$}};
    \fill(2,0) circle(1.5pt);
    \draw[thick,fill=black!20] (3,-4.17)arc (70:110:3);
    \node at (2,-4.2){\tiny{$p_2$}};
    \fill(2,-4) circle(1.5pt);
    \draw[blue,thick](2,-4)--(2,0);
    \node [blue]at (2,-2)[left]{\tiny{$\af_i$}};
    \draw[blue,thick](2,-4)--(0,-3);
    \node [blue]at (0.5,-3.5)[below]{\tiny{$b_{s}$}};
    \draw[blue,thick](2,0)--(0,-2);
    \node [blue]at (1,-1)[above]{\tiny{$c_t$}};
    \node [blue]at (8.5,-3.7){\tiny{$c_{t+1}$}};
    \fill(0,-3) circle(2pt);
  	\draw[thick,fill=black!20] (6,0.17)arc (250:290:3);
    \node at (7,0.2){\tiny{$p_3$}};
    \fill(7,0) circle(1.5pt);
    \draw[thick,fill=black!20] (8,-4.17)arc (70:110:3);
    \node at (7,-4.2){\tiny{$p_4$}};
    \fill(7,-4) circle(1.5pt);
    \draw[blue,thick](7,-4)--(7,0);
    \fill(9,-2) circle(2pt);
    \node [blue] at (7,-1.5)[right]{\tiny{$\af_j$}};
    \node [blue] at (8,-1)[right]{\tiny{$c_{s+1}$}};
     \draw[blue,thick](7,-4)--(9,-3);
     \draw[blue,thick](7,0)--(9,-2);
    \fill(9,-3) circle(1.5pt);
    \draw[red,thick] (0,-3) .. controls (5,-2.5) and (8,-3) .. (9,-3.75);
    \node [red]at (0,-1)[left]{\tiny{$\widetilde{\gamma_2}^{(+-)}$}};
    \node [red]at (0,-1.5)[left]{\tiny{$\widetilde{\gamma_2}^{(++)}$}};
    \node [red]at (5,-3.5)[left]{\tiny{$\widetilde{\gamma_2}^{(-+)}$}};
    \node [red]at (4,-2.3){\tiny{$\widetilde{\gamma_2}^{(--)}$}};
    
    \draw[red,thick] (0,-1.5) .. controls (6,-2) and (8,-2.5) .. (9,-3.5);
    \draw[red,thick] (0,-1) .. controls (6,-1.5) and (8,-2.25) .. (9,-2);
    \draw[red,thick] (0,-3)-- (9,-2);
    \node at (0,-3)[left]{\tiny{$q_1$}};
    \node at (9,-2)[right]{\tiny{$q_2$}};
    \draw[blue,thick,->] (1.5,-0.5) to(2,-1) ;
    \draw[blue,thick,->] (8,-3.5)to  (7,-3) ;
	\end{tikzpicture}
    \end{minipage}
	  \caption{The local configuration of  $\gamma_1$ (left) and $\gamma_2$ (right), such that $\widetilde{\gamma_1}$ is anticlockwise admissible and $\widetilde{\gamma_2}$ is clockwise admissible}
    \label{f:location for 1,2}
\end{figure}
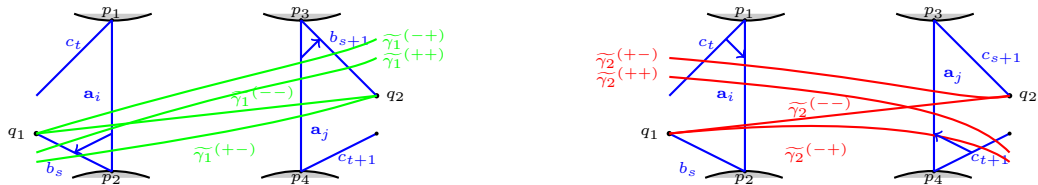

In the following, for convenience,  we use  $\widetilde{\gamma_a}\in\widetilde{\gamma_a}^{(\xi_1,\xi_2)}$ to denote that $\widetilde{\gamma_a}$ is of type $\widetilde{\gamma_a}^{(\xi_1,\xi_2)}$ for each $1\leq a\leq 2$, where $\xi_1\in\{-,+\},\xi_2\in\{-,+\}$. 

Now, using the hypothesis $\Int(\gamma_1,\gamma_2)=0$ and $(\widetilde{\gamma_1},\widetilde{\gamma_2})\neq {(\gamma_1,\gamma_2)}$, we analyze the possible locations of $\widetilde{\gamma_1},~\widetilde{\gamma_2}$. 
\begin{enumerate}
\item Suppose $\widetilde{\gamma_2}\in\widetilde{\gamma_2}^{(--)}$. Since 
$(\widetilde{\gamma_1},\widetilde{\gamma_2})\neq {(\gamma_1,\gamma_2)}$,
we have $\widetilde{\gamma_1}\notin\widetilde{\gamma_1}^{(--)}$. Also, since $\Int(\gamma_1,\gamma_2)=0$,  the only geometrically possible cases avoiding intersection are $\widetilde{\gamma_1}^{(+-)}$ and $\widetilde{\gamma_1}^{(-+)}$ (see Figure \ref{f:when i=s,j=t}).

When $\widetilde{\gamma_1}\in \widetilde{\gamma_1}^{(-+)}$, the arc $\gamma_1$ shares a common endpoint $q_1$ with $\gamma_2$. By the assumption that $\Int(\gamma_1, \gamma_2) = 0$ (implying minimal intersection), the arc $\gamma_2$ must follow $\gamma_1$ in the clockwise direction at $q_1$. Similarly, when $\widetilde{\gamma}_1 \in \widetilde{\gamma_1}^{(+-)}$, the arc $\gamma_1$ shares a common endpoint $q_2$ with $\gamma_2$, and again $\gamma_2$ follows $\gamma_1$ in the clockwise direction at $q_2$. In both cases, we obtain $f=f_{\angle\alpha}$ for some $\angle\alpha\in\apfree(\gamma_1,\gamma_2)$. 

 \begin{figure}[htpb]
   \centering
    \centering
	\begin{tikzpicture}[scale=0.5]
	\draw[thick,fill=black!20] (1,0.17)arc (250:290:3);
    \node at (2,0.2){\tiny{$p_1$}};
    \fill(2,0) circle(1.5pt);
    \draw[thick,fill=black!20] (3,-4.17)arc (70:110:3);
    \node at (2,-4.2){\tiny{$p_2$}};
    \fill(2,-4) circle(1.5pt);
    \draw[blue,thick](2,-4)--(2,0);
    \node [blue]at (2,-2)[left]{\tiny{$\af_i$}};
     \draw[blue,thick](2,-4)--(0,-3);
     \draw[blue,thick](2,0)--(0,-2);
      \fill(0,-3) circle(2pt);
  	 \draw[thick,fill=black!20] (6,0.17)arc (250:290:3);
     \node at (7,0.2){\tiny{$p_3$}};
     \fill(7,0) circle(1.5pt);
     \fill(9,-2) circle(2pt);
     \draw[thick,fill=black!20] (8,-4.17)arc (70:110:3);
     \node at (7,-4.2){\tiny{$p_4$}};
     \fill(7,-4) circle(1.5pt);
     \draw[blue,thick](7,-4)--(7,0);
     \node [blue]at (7,-2)[right]{\tiny{$\af_j$}};
     \draw[blue,thick](7,-4)--(9,-3);
     \draw[blue,thick](7,0)--(9,-2);
     \fill(9,-3) circle(1.5pt);
    \draw[blue,thick,->](2,-3) to (1,-3.5);
     \draw[blue,thick,->] (7,-1) to (7.5,-.5);
    \draw[blue,thick,->] (1.5,-0.5) to(2,-1) ;
    \draw[blue,thick,->] (8,-3.5)to  (7,-3) ;
     \node [green]at (10,-0.5){\tiny{$\widetilde{\gamma_1}^{(-+)}$}};
    \node [green]at (5,-3.5){\tiny{$\widetilde{\gamma_1}^{(+-)}$}};
    \node [green]at (10,-1){\tiny{$\widetilde{\gamma_1}^{(++)}$}};
     \node [red]at (6.2,-2){\tiny{$\widetilde{\gamma_2}^{(--)}$}};
      \draw[green,thick] (0,-3) .. controls (5,-1.5) and (8,-1) .. (9,-0.5);
      \draw[green,thick] (0,-3.5) .. controls (6,-1.5) and (8,-1.5) .. (9,-1);
      \draw[green,thick] (0,-3.75) .. controls (7,-2.75) and (8,-2.25) .. (9,-2);
     \draw[red,thick] (0,-3)-- (9,-2);
      \node at (0,-3)[left]{\tiny{$q_1$}};
      \node at (9,-2)[right]{\tiny{$q_2$}};
	\end{tikzpicture}
	  \caption{Case for $\widetilde{\gamma_2}\in\widetilde{\gamma_2}^{(--)}$}
    \label{f:when i=s,j=t}
\end{figure}
\item Suppose $\widetilde{\gamma_2}\in\widetilde{\gamma_2}^{(-+)}$. If $\widetilde{\gamma_1}\in\widetilde{\gamma_1}^{(+-)}$ or $\widetilde{\gamma_1}\in\widetilde{\gamma_1}^{(++)}$, then $\Int(\gamma_1,\gamma_2)\neq 0$ (see Figure \ref{f:when i=s,j<t}), contradicting our assumption.  If $\widetilde{\gamma_1}\in\widetilde{\gamma_1}^{(-+)}$ or $\widetilde{\gamma_1}\in \widetilde{\gamma_1}^{(--)}$, then $\gamma_1$ shares a common endpoint $q_1$ with $\gamma_2$, and at $q_1$, the arc$\gamma_2$ follows $\gamma_1$ in clockwise direction. Hence $f=f_{\angle\alpha}$ for some $\angle\alpha\in\apfree(\gamma_1,\gamma_2)$. 
\begin{figure}[htpb]
   \centering
    \centering
	\begin{tikzpicture}[scale=0.5]
	\draw[thick,fill=black!20] (1,0.17)arc (250:290:3);
    \node at (2,0.2){\tiny{$p_1$}};
    \fill(2,0) circle(1.5pt);
    \draw[thick,fill=black!20] (3,-4.17)arc (70:110:3);
    \node at (2,-4.2){\tiny{$p_2$}};
    \fill(2,-4) circle(1.5pt);
    \draw[blue,thick](2,-4)--(2,0);
    \node [blue]at (2,-2)[left]{\tiny{$\af_i$}};
     \draw[blue,thick](2,-4)--(0,-3);
     \draw[blue,thick](2,0)--(0,-2);
      \fill(0,-3) circle(2pt);
  	 \draw[thick,fill=black!20] (6,0.17)arc (250:290:3);
     \node at (7,0.2){\tiny{$p_3$}};
     \fill(7,0) circle(1.5pt);
     \fill(9,-2) circle(2pt);
     \draw[thick,fill=black!20] (8,-4.17)arc (70:110:3);
     \node at (7,-4.2){\tiny{$p_4$}};
     \fill(7,-4) circle(1.5pt);
     \draw[blue,thick](7,-4)--(7,0);
     \node [blue]at (7,-2)[right]{\tiny{$\af_j$}};
     \draw[blue,thick](7,-4)--(9,-3);
     \draw[blue,thick](7,0)--(9,-2);
     \fill(9,-3) circle(1.5pt);
     \draw[blue,thick,->](2,-3) to (1,-3.5);
     \draw[blue,thick,->] (7,-1) to (7.5,-.5);
    \draw[blue,thick,->] (1.5,-0.5) to(2,-1) ;
    \draw[blue,thick,->] (8,-3.5)to  (7,-3) ;
    \node [green]at (10,-0.5){\tiny{$\widetilde{\gamma_1}^{(-+)}$}};
    \node [green]at (5,-3.5){\tiny{$\widetilde{\gamma_1}^{(+-)}$}};
    \node [green]at (10,-1){\tiny{$\widetilde{\gamma_1}^{(++)}$}};
     \node [green]at (6,-2){\tiny{$\widetilde{\gamma_1}^{(--)}$}};
      \draw[green,thick] (0,-3) .. controls (5,-1.5) and (8,-1) .. (9,-0.5);
      \draw[green,thick] (0,-3.5) .. controls (6,-1.5) and (8,-1.5) .. (9,-1);
      \draw[green,thick] (0,-3.75) .. controls (7,-2.75) and (8,-2.25) .. (9,-2);
      \draw[red,thick] (0,-3) .. controls (5,-2.5) and (8,-3) .. (9,-3.75);
      \node at (0,-3)[left]{\tiny{$q_1$}};
      \node at (9,-2)[right]{\tiny{$q_2$}};
      \draw[green,thick] (0,-3)-- (9,-2);
      \node [red]at (10,-3.5){\tiny{$\widetilde{\gamma_2}^{(-+)}$}};
	\end{tikzpicture}
	  \caption{Case for $\widetilde{\gamma_2}\in\widetilde{\gamma_2}^{(-+)}$}
    \label{f:when i=s,j<t}
\end{figure}

\item Suppose $\widetilde{\gamma_2}\in \widetilde{\gamma_2}^{(+-)}$.  If $\widetilde{\gamma_1}\in\widetilde{\gamma_1}^{(-+)}$ or $\widetilde{\gamma_1}\in\widetilde{\gamma_1}^{(++)}$, then clearly $\Int(\gamma_1,\gamma_2) \neq 0$ (see Figure \ref{f:when i>s,j=t}), contradicting our assumption. If $\widetilde{\gamma}_1 \in \widetilde{\gamma}_1^{(--)}$ or $\widetilde{\gamma}_1 \in \widetilde{\gamma}_1^{(+-)}$, then $\gamma_2$ shares a common endpoint $q_2$ with $\gamma_1$. Similarly, at $q_2$, the arc $\gamma_2$ follows $\gamma_1$ in the clockwise direction. Hence $f=f_{\angle\alpha}$ for some $\angle\alpha\in\apfree(\gamma_1,\gamma_2)$. 
\begin{figure}[htpb]
   \centering
    \centering
	\begin{tikzpicture}[scale=0.5]
	\draw[thick,fill=black!20] (1,0.17)arc (250:290:3);
    \node at (2,0.2){\tiny{$p_1$}};
    \fill(2,0) circle(1.5pt);
    \draw[thick,fill=black!20] (3,-4.17)arc (70:110:3);
    \node at (2,-4.2){\tiny{$p_2$}};
    \fill(2,-4) circle(1.5pt);
    \draw[blue,thick](2,-4)--(2,0);
    \node [blue]at (2,-2)[left]{\tiny{$\af_i$}};
     \draw[blue,thick](2,-4)--(0,-3);
     \draw[blue,thick](2,0)--(0,-2);
      \fill(0,-3) circle(2pt);
  	 \draw[thick,fill=black!20] (6,0.17)arc (250:290:3);
     \node at (7,0.2){\tiny{$p_3$}};
     \fill(7,0) circle(1.5pt);
     \fill(9,-2) circle(2pt);
     \draw[thick,fill=black!20] (8,-4.17)arc (70:110:3);
     \node at (7,-4.2){\tiny{$p_4$}};
     \fill(7,-4) circle(1.5pt);
     \draw[blue,thick](7,-4)--(7,0);
     \node [blue]at (7,-2)[right]{\tiny{$\af_j$}};
     \draw[blue,thick](7,-4)--(9,-3);
     \draw[blue,thick](7,0)--(9,-2);
     \fill(9,-3) circle(1.5pt);
   \draw[blue,thick,->](2,-3) to (1,-3.5);
     \draw[blue,thick,->] (7,-1) to (7.5,-.5);
    \draw[blue,thick,->] (1.5,-0.5) to(2,-1) ;
    \draw[blue,thick,->] (8,-3.5)to  (7,-3) ;
     \node [green]at (10,-0.5){\tiny{$\widetilde{\gamma_1}^{(-+)}$}};
    \node [green]at (5,-3.5){\tiny{$\widetilde{\gamma_1}^{(+-)}$}};
    \node [green]at (10,-1){\tiny{$\widetilde{\gamma_1}^{(++)}$}};
    \node [green]at (6,-2){\tiny{$\widetilde{\gamma_1}^{(--)}$}};
     \node [red]at (0,-1)[left]{\tiny{$\widetilde{\gamma_2}^{(+-)}$}};
      \draw[green,thick] (0,-3) .. controls (5,-1.5) and (8,-1) .. (9,-0.5);
      \draw[green,thick] (0,-3.5) .. controls (6,-1.5) and (8,-1.5) .. (9,-1);
      \draw[green,thick] (0,-3.75) .. controls (7,-2.75) and (8,-2.25) .. (9,-2);
     \draw[red,thick] (0,-1) .. controls (6,-1.5) and (8,-2.25) .. (9,-2);
      \node at (0,-3)[left]{\tiny{$q_1$}};
      \node at (9,-2)[right]{\tiny{$q_2$}};
      \draw[green,thick] (0,-3)-- (9,-2);
     
	\end{tikzpicture}
	  \caption{Case for $\widetilde{\gamma_2}\in\widetilde{\gamma_2}^{(+-)}$}
    \label{f:when i>s,j=t}
\end{figure}
 \item Suppose $\widetilde{\gamma_2}\in \widetilde{\gamma_2}^{(++)}$. In this case, we always have $\Int(\gamma_1,\gamma_2)\ne 0$ (see Figure \ref{f:when i>s,j<t}), which contradicts our assumption.

   \begin{figure}[htpb]
   \centering
    \centering
	\begin{tikzpicture}[scale=0.5]
	\draw[thick,fill=black!20] (1,0.17)arc (250:290:3);
    \node at (2,0.2){\tiny{$p_1$}};
    \fill(2,0) circle(1.5pt);
    \draw[thick,fill=black!20] (3,-4.17)arc (70:110:3);
    \node at (2,-4.2){\tiny{$p_2$}};
    \fill(2,-4) circle(1.5pt);
    \draw[blue,thick](2,-4)--(2,0);
     \node [blue]at (2,-2)[left]{\tiny{$\af_i$}};
     \draw[blue,thick](2,-4)--(0,-3);
      \draw[blue,thick](2,0)--(0,-2);
      \fill(0,-3) circle(2pt);
  	 \draw[thick,fill=black!20] (6,0.17)arc (250:290:3);
     \node at (7,0.2){\tiny{$p_3$}};
     \fill(7,0) circle(1.5pt);
     \fill(9,-2) circle(2pt);
     \draw[thick,fill=black!20] (8,-4.17)arc (70:110:3);
     \node at (7,-4.2){\tiny{$p_4$}};
     \fill(7,-4) circle(1.5pt);
     \draw[blue,thick](7,-4)--(7,0);
     \node [blue]at (7,-2)[right]{\tiny{$\af_j$}};
     \draw[blue,thick](7,-4)--(9,-3);
     \draw[blue,thick](7,0)--(9,-2);
     \fill(9,-3) circle(1.5pt);
      \node [red]at (0,-1.5)[left]{\tiny{$\widetilde{\gamma_2}^{(++)}$}};
       \draw[blue,thick,->](2,-3) to (1,-3.5);
     \draw[blue,thick,->] (7,-1) to (7.5,-.5);
    \draw[blue,thick,->] (1.5,-0.5) to(2,-1) ;
    \draw[blue,thick,->] (8,-3.5)to  (7,-3) ;
      \node [green]at (10,-0.5){\tiny{$\widetilde{\gamma_1}^{(-+)}$}};
     \node [green]at (5,-3.5){\tiny{$\widetilde{\gamma_1}^{(+-)}$}};
     \node [green]at (10,-1){\tiny{$\widetilde{\gamma_1}^{(++)}$}};
     \node [green]at (6,-2){\tiny{$\widetilde{\gamma_1}^{(--)}$}};
      \draw[green,thick] (0,-3) .. controls (5,-1.5) and (8,-1) .. (9,-0.5);
      \draw[green,thick] (0,-3.5) .. controls (6,-1.5) and (8,-1.5) .. (9,-1);
      \draw[green,thick] (0,-3.75) .. controls (7,-2.75) and (8,-2.25) .. (9,-2);
       \draw[red,thick] (0,-1.5) .. controls (6,-2) and (8,-2.5) .. (9,-3.5);
       \draw[green,thick] (0,-3)-- (9,-2);
      \node at (0,-3)[left]{\tiny{$q_1$}};
      \node at (9,-2)[right]{\tiny{$q_2$}};
	\end{tikzpicture}
	  \caption{Case for $\widetilde{\gamma_2}\in\widetilde{\gamma_2}^{(++)}$}
    \label{f:when i>s,j<t}
\end{figure}
\end{enumerate}
Therefore, in all possible cases, we have $f\in \mathcal{B}'$. Hence $\mathcal{B}=\mathcal{B}'$, and the theorem follows.
\end{proof}

\subsection{Morphism relations via 
$\PP$-free negatively oriented angles}

Let $\mathbf{R}$  be a partial dissection on $(\SS,\MM,\PP)$.

\begin{lemma}\label{l:f_alpha_3 can be factored}
Let $\gamma_a,\gamma_b,\gamma_c\in\mathbf{R}$ share a common endpoint $p$.
    Let $\angle\alpha_1\in\appfree(\gamma_a,\gamma_b)$ and let $\angle\alpha_2\in\appfree(\gamma_b,\gamma_c)$ be adjacent to $\angle\alpha_1$ such that  $\angle\alpha_2$ is on the clockwise side of $\angle\alpha_1$ (cf. Figure \ref{f:the location for k,j,i-}).
     Set $\angle \alpha_3=\angle\alpha_1+\angle\alpha_2$. Then $\angle\alpha_3\in \appfree(\gamma_a,\gamma_c)$ and $$f_{\angle\alpha_3}=f_{\angle\alpha_2}f_{\angle\alpha_1}.$$
\end{lemma}     
  \begin{proof} It is clear $\angle\alpha_3\in \appfree(\gamma_a,\gamma_c)$.
 \begin{figure}[h]
        \centering
       \begin{tikzpicture}[scale=0.5]
				\draw[thick,blue] (5,3) .. controls (6,2.5) and (8,1) .. (9,-1)\nn;
                 \draw[blue] (6.6,2.6)node[left]{$\af_r$}; 
                \draw[thick,blue] (5,3) .. controls (6,3.5) and (7,4) .. (8,5.5)\nn;
                   \draw[blue] (6.6,4)node[left]{$\af_{r+1}$};
                \draw[thick,blue] (9,-1) .. controls (9,0.5) and (9.5,3) .. (10,4)\nn;
                \draw[blue] (9.3,3.4)node{$\af_h$};
                \draw[thick,blue] (10,4) .. controls (11,3.8) and (11.5,4.5) .. (13,5.5)\nn;
                 \draw[blue] (12,4.2)node{$\af_{h+1}$};
                 \draw[blue] (11.5,-0.5)node[above]{$\af_n$};
			    \draw[thick,blue] (9,-1) .. controls (8,-0.5) and (7,-0.3) .. (4,-0.7)\nn;
                \draw[blue] (6.6,-0.5)node[below]{$\af_1$};
                \draw[thick,blue] (9,-1) .. controls (10,-0.5) and (10.5,0.2) .. (13,-1.3)\nn;
                \draw[thick,blue] (9,-1) .. controls (8,0.3) and (7,0.5) .. (4,1)\nn;
				\draw  (5,3)\nn (8,-1.5)\nn;
                \draw [red,thick] (8,-1.5)..controls (9,0.5) and (10,0)..(12,-1.5);
                \draw [red,thick] (8,-1.5)..controls (9,1) and (10,2.5)..(12,5.8);
                \draw [red,thick] (8,-1.5)..controls (7,0) and (6,0.5)..(7,5.8);
                \draw[red] (12,-1.7)node[right]{$\gamma_c$};
                \draw[red] (12.2,6)node[right]{$\gamma_b$};
                \draw (9,-1)node[below]{$q$};
                \draw (8,-1.5) node [below]{$p$};
                \draw[red] (6.8,6)node[left]{$\gamma_a$}; 
                \draw[blue] (6.4,1.8)node[left]{$\vdots$}; 
                \draw[thick,blue,dotted] (8,2) .. controls (8.5,2) and (8.8,2) .. (9,2);
                \draw[thick,blue,dotted] (10,2) .. controls (11,1.5) and (11,1) .. (11,0);
                \draw[red,thick,->](6.8,.5)to [bend left=30](8.9,.5);
                \draw[red,thick,->](8.9,.5) to[bend left=30](10,-.3);
                \draw[red](8.5,1.2)node{\tiny${\angle\alpha_1}$};
                \draw[red](10,0.5)node{\tiny${\angle\alpha_2}$};
       \end{tikzpicture}
        \caption{The location for $\gamma_a,~\gamma_b,~\gamma_c$}\label{f:the location for k,j,i-}
    \end{figure}
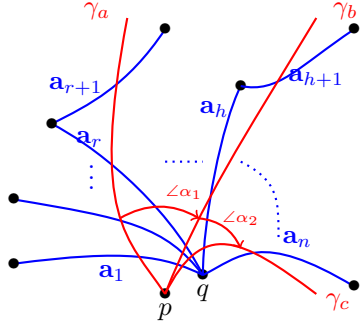
Let $\widetilde{\gamma_{a,1}}$ be the arc segments of $\gamma_a$, let $\widetilde{\gamma_{b,1}}$ and $\widetilde{\gamma_{b,2}}$ be the arc segments of $\gamma_b$, and let $\widetilde{\gamma_{c,1}}$ be the arc segments of $\gamma_c$, chosen so that $f_{\angle\alpha_1}$ is induced by the pair $(\widetilde{\gamma_{a,1}},\widetilde{\gamma_{b,1}})$ and $f_{\angle\alpha_2}$ is induced by the pair $(\widetilde{\gamma_{b,2}},\widetilde{\gamma_{c,1}})$. Choose suitable orientations of these segments. Then we have 
 \[(\widetilde{\gamma_{a,1}},~\widetilde{\gamma_{b,1}})\in\aad(\gamma_a)\times\cad(\gamma_b), \Intv(\widetilde{\gamma_{a,1}})=\Intv(\widetilde{\gamma_{b,1}}), \widetilde{\gamma_{a,1}}\sim\widetilde{\gamma_{b,1}},~\widetilde{\gamma_{a,1}}(0)=\widetilde{\gamma_{b,1}}(0)=p,\] and
 \[(\widetilde{\gamma_{b,2}},~\widetilde{\gamma_{c,1}})\in\aad(\gamma_b)\times\cad(\gamma_c),\Intv(\widetilde{\gamma_{b,2}})=\Intv(\widetilde{\gamma_{c,1}}), \widetilde{\gamma_{b,2}}\sim\widetilde{\gamma_{c,1}},~~\widetilde{\gamma_{b,2}}(0)=\widetilde{\gamma_{c,1}}(0)=p.\]
  Let $$x=|\Intv(\widetilde{\gamma_{a,1}})|=|\Intv(\widetilde{\gamma_{b,1}})|, \quad y=|\Intv(\widetilde{\gamma_{b,2}})|=|\Intv(\widetilde{\gamma_{c,1}})|.$$ 
Note that $\widetilde{\gamma_{b,1}}(0)=~\widetilde{\gamma_{c,2}}(0)=p.$ Since $\widetilde{\gamma_{b,1}}\in\cad(\gamma_b)$ and $\widetilde{\gamma_{b,2}}\in\aad(\gamma_b)$, we must have $x\neq y$.

\noindent\textbf{Case 1}:  $x<y$. Then $\widetilde{\gamma_{b,1}}$ is a sub-arc segment of  $\widetilde{\gamma_{b,2}}$. Since $\widetilde{\gamma_{b,2}}\sim\widetilde{\gamma_{c,1}}$, one can find a sub-arc segment $\widetilde{\gamma_{c,2}}$ of $\widetilde{\gamma_{c,1}}$ such that $ \widetilde{\gamma_{c,2}}\sim \widetilde{\gamma_{b,1}}$. Hence $$\widetilde{\gamma_{a,1}}\sim \widetilde{\gamma_{b,1}}\sim \widetilde{\gamma_{c,2}}.$$
Since $|\widetilde{\gamma_{c,2}}|=x<y=|\widetilde{\gamma_{c,1}}|$, together with $\widetilde{\gamma_{b,2}}\sim\widetilde{\gamma_{c,1}}$ and
$\widetilde{\gamma_{b,1}}\in\cad(\gamma_b)$, we obtain
 $\widetilde{\gamma_{c,2}}\in \cad(\gamma_c)$. Thus we get a pair $(\widetilde{\gamma_{a,1}},~\widetilde{\gamma_{c,2}})$ such that 
\[(\widetilde{\gamma_{a,1}},~\widetilde{\gamma_{c,2}})\in\aad(\gamma_a)\times\cad(\gamma_c), \Intv(\widetilde{\gamma_{a,1}})=\Intv(\widetilde{\gamma_{c,2}}), \widetilde{\gamma_{a,1}}\sim\widetilde{\gamma_{c,2}},\widetilde{\gamma_{a,1}}(0)=\widetilde{\gamma_{c,2}}(0)=p.\]
This pair induces  $f_{\angle\alpha_3}$. By construction, we have $$f_{\angle\alpha_3}=f_{\angle\alpha_2}f_{\angle\alpha_1}.$$

\noindent\textbf{Case 2}:  $x>y$. Then $\widetilde{\gamma_{b,2}}$ is a sub-arc segment of  $\widetilde{\gamma_{b,1}}$. Since $\widetilde{\gamma_{b,1}}\sim\widetilde{\gamma_{a,1}}$, one can find a sub-arc segment $\widetilde{\gamma_{a,2}}$ of $\widetilde{\gamma_{a,1}}$ such that $ \widetilde{\gamma_{a,2}}\sim \widetilde{\gamma_{b,2}}$. Hence $$\widetilde{\gamma_{a,2}}\sim \widetilde{\gamma_{b,2}}\sim \widetilde{\gamma_{c,2}}.$$
Since $|\widetilde{\gamma_{a,1}}|=x>y=|\widetilde{\gamma_{a,2}}|$, together with $\widetilde{\gamma_{b,1}}\sim\widetilde{\gamma_{a,1}}$ and 
 $\widetilde{\gamma_{b,2}}\in\aad(\gamma_b)$, we obtain  
 $\widetilde{\gamma_{a,2}}\in \aad(\gamma_a)$. Thus we get a pair $(\widetilde{\gamma_{a,2}},~\widetilde{\gamma_{c,1}})$ such that 
\[(\widetilde{\gamma_{a,2}},~\widetilde{\gamma_{c,1}})\in\aad(\gamma_a)\times\cad(\gamma_c), \Intv(\widetilde{\gamma_{a,2}})=\Intv(\widetilde{\gamma_{c,1}}), \widetilde{\gamma_{a,2}}\sim\widetilde{\gamma_{c,1}},\widetilde{\gamma_{a,2}}(0)=\widetilde{\gamma_{c,1}}(0)=p.\]
This pair induces $f_{\angle\alpha_3}$. By construction, we have $$f_{\angle\alpha_3}=f_{\angle\alpha_2}f_{\angle\alpha_1}.$$ This completes the proof.
 \end{proof}   

\begin{lemma}\label{l:relation for gamma_a is a loop}
    Let $\gamma_a\in\mathbf{R}$ be a loop and let $\angle\alpha\in\appfree(\gamma_a,\gamma_a)$.
    Then $$f^2_{\angle\alpha}=f_{\angle\alpha} f_{\angle\alpha}=0.$$
\end{lemma}
\begin{proof}
     Let $\widetilde{\gamma_{a,1}}$ and $ \widetilde{\gamma_{a,2}}$ be the arc segments of $\gamma_a$ such that $f_{\angle\alpha}$ is induced by the pair $(\widetilde{\gamma_{a,1}}, \widetilde{\gamma_{a,2}})$. Thus there are orientations of $\widetilde{\gamma_{a,1}}$ and $\widetilde{\gamma_{a,2}}$ satisfying
 \[(\widetilde{\gamma_{a,1}},~\widetilde{\gamma_{a,2}})\in\aad(\gamma_a)\times\cad(\gamma_a), \Intv(\widetilde{\gamma_{a,1}})=\Intv(\widetilde{\gamma_{a,2}}), \widetilde{\gamma_{a,1}}\sim\widetilde{\gamma_{a,2}},~\widetilde{\gamma_{a,1}}(0)=~\widetilde{\gamma_{a,2}}(0)=p.\]
 Suppose $f_{\angle\alpha} f_{\angle\alpha}\neq 0$. Then there exist three arc segments 
 $\widetilde{\gamma_{a,3}},\widetilde{\gamma_{a,4}}, \widetilde{\gamma_{a,5}}$ such that 
\begin{itemize}
    \item  $\widetilde{\gamma_{a,3}}\in\aad(\gamma_{a})$ is a sub-arc segment of $\widetilde{\gamma_{a,1}}$;
    \item  $\widetilde{\gamma_{a,4}}$ is a sub-arc segment of $\widetilde{\gamma_{a,1}}$ and a sub-arc segment of $\widetilde{\gamma_{a,2}}$;
    \item $\widetilde{\gamma_{a,5}}\in\cad(\gamma_{a})$ is a sub-arc segment of $\widetilde{\gamma_{a,2}}$;
\end{itemize}
 satisfying $$ \Intv(\widetilde{\gamma_{a,3}})=\Intv(\widetilde{\gamma_{a,4}})=\Intv(\widetilde{\gamma_{a,5}}), \quad \widetilde{\gamma_{a,3}}\sim\widetilde{\gamma_{a,4}}\sim\widetilde{\gamma_{a,5}}.$$
and $f_{\angle\alpha} f_{\angle\alpha}$ is induced by $(\widetilde{\gamma_{a,3}},~\widetilde{\gamma_{a,5}}).$ Note that by Theorem~\ref{t:basis of rad fromed by angles}, we must have either $\widetilde{\gamma_{a,3}}(0)\in\MM$ or   $\widetilde{\gamma_{a,3}}(1)\in\MM$.
However,
 because $\gamma_a$ is a loop, the first arc segment of $\widetilde{\gamma_{a,1}}$ (divided by $\PP$) and the first arc segment of $\widetilde{\gamma_{a,2}}$ (divided by $\PP$) are exactly the two end arc segments of $\gamma_a$. Then by 
 $(\widetilde{\gamma_{a,1}}, \widetilde{\gamma_{a,2}})\neq (\gamma_a,\gamma_a)$, $\widetilde{\gamma_{a,4}}$ is a sub-arc segment of $\widetilde{\gamma_{a,1}}$ and a sub-arc segment of $\widetilde{\gamma_{a,2}}$, we must have  $ \widetilde{\gamma_{a,4}}$ is a sub-arc segment of $\gamma_a$ with no endpoint in $\MM$. Since $\widetilde{\gamma_{a,3}} \sim \widetilde{\gamma_{a,4}}$, neither endpoint of $\widetilde{\gamma_{a,3}}$ is a marked point, which contradicts the requirement from Theorem~\ref{t:basis of rad fromed by angles}.  Therefore, we must have  $f_{\angle\alpha} f_{\angle\alpha}=0$. 
\end{proof}
 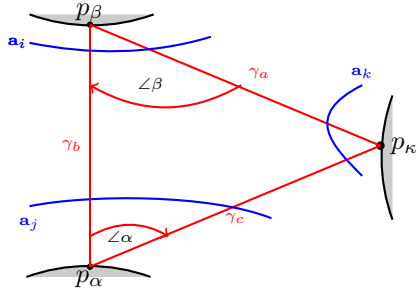
\begin{figure}[htpb]\centering
	\begin{tikzpicture}[scale=0.8]
    \node [blue]at (4.8,-0.3){\tiny{$\af_i$}};
	\draw[thick,fill=black!20] (5,0.17)arc (250:290:3);
    \node at (6,0.2){$p_\beta$};
    \fill(6,0) circle(1.5pt);
    \draw[thick,fill=black!20] (7,-4.17)arc (70:110:3);
    \node at (6,-4.2){$p_\alpha$};
    \fill(6,-4) circle(1.5pt);
    \draw[thick,fill=black!20] (11,-1.17)arc (160:200:3);
    \node at (11.2,-2){$p_\kappa$};
    \fill(10.8,-2) circle(2pt);
    \draw[red,thick] (6,0)--(6,-4);
     \draw[red,thick,->] (8.5,-1) to [bend left=30] (6,-1);
    \node at(7,-1){\tiny{$\angle\beta$}};
     \draw[red,thick,->] (6,-3.5) to [bend left=30](7.3,-3.5);
    \node at(6.5,-3.5){\tiny{$\angle\alpha$}};
    \draw[red,thick] (6,0)--(10.8,-2);
    \draw[red,thick] (6,-4)--(10.8,-2);
    \node [red]at (5.7,-2){\tiny{$\gamma_b$}};
    \node [red]at (8.4,-3.2){\tiny{$\gamma_c$}};
    \node [red]at (8.8,-0.8){\tiny{$\gamma_a$}};
    \draw [blue,thick](5,-0.3)..controls (6,-.5) and (7,-.5)..(8,-0.2);
    \node [blue]at (4.8,-0.3){\tiny{$\af_i$}};
    \draw [blue,thick](5,-3)..controls (6,-2.8) and (8,-2.8)..(9,-3.2);
    \node [blue]at (5,-3.3){\tiny{$\af_j$}};
    \draw [blue,thick](10.5,-1)..controls (9.5,-1.5) and (10,-2)..(10.5,-2.5);
    \node [blue]at (10.5,-0.8){\tiny{$\af_k$}};
	\end{tikzpicture}
	\caption{Case  when $p_\alpha\ne p_\beta$}
	\label{f:relation in A_R, endpoint are disctinct}
\end{figure}

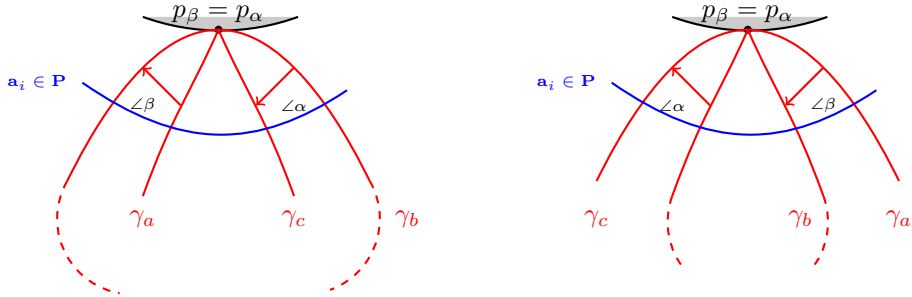
\begin{figure}[h]
\begin{tikzpicture}

\draw[thick,fill=black!20] (2,0.17)arc (250:290:3);
\node at (3,0.2){$p_\beta=p_\alpha$};
\fill(3,0) circle(1.5pt);
\draw[thick,color=red] (3,0) parabola (1,-2);
\draw[thick,color=red,dashed] (1,-2) arc (150:260:1);
\draw[thick,color=red,->] (2.5,-1)--(2,-0.5);
\node at (2,-1) {{\tiny$\angle\beta$}};
\draw[thick,color=red] (3,0) .. controls (2.6,-0.9) and (2.4,-1.1) .. (2,-2.2);
\draw[thick,color=red] (3,0) .. controls (3.4,-0.9) and (3.6,-1.1) .. (4,-2.2);
\draw[thick,color=red] (3,0) parabola (5,-2);
\draw[thick,color=red,dashed] (5,-2) arc (30:-70:1);
\node[red] at (5.5,-2.5){$\gamma_b$};
\node[red] at (4,-2.5){$\gamma_c$};
\node[red] at (2,-2.5){$\gamma_a$};
\draw[thick,color=red,->](4,-0.5)--(3.5,-1);
\draw [blue,thick](1.2,-0.7)..controls (2.5,-1.6) and (3.5,-1.6)..(4.7,-0.8);
\node at (4,-1) {\tiny$\angle\alpha$};
\node [blue] at (0.6,-0.7){\tiny{$\af_i\in\PP$}};

\draw[thick,fill=black!20] (9,0.17)arc (250:290:3);
\node at (10,0.2){$p_\beta=p_\alpha$};
\fill(10,0) circle(1.5pt);
\draw[thick,color=red] (10,0) parabola (8,-2);
\draw[thick,color=red,dashed] (9,-2.2) arc (160:220:1);
\draw[thick,color=red,->](9.5,-1)--(9,-0.5);
\node at (9,-1) {\tiny$\angle\alpha$};
\draw[thick,color=red] (10,0) .. controls (9.6,-0.9) and (9.4,-1.1) .. (9,-2.2);
\draw[thick,color=red] (10,0) .. controls (10.4,-0.9) and (10.6,-1.1) .. (11,-2.2);
\draw[thick,color=red] (10,0) parabola (12,-2);
\draw[thick,color=red,dashed] (11,-2.2) arc (20:-40:1);
\draw[thick,color=red,->](11,-0.5)--(10.5,-1);
\node at (11,-1) {\tiny$\angle\beta$};
\node[red] at (10.7,-2.5){$\gamma_b$};
\node[red] at (12,-2.5){$\gamma_a$};
\node[red] at (8,-2.5){$\gamma_c$};
\draw [blue,thick](8.2,-0.7)..controls (9.5,-1.6) and (10.5,-1.6)..(11.7,-0.8);
\node [blue] at (7.6,-0.7){\tiny{$\af_i\in\PP$}};
\end{tikzpicture}
\caption{Case  when $p_\beta=p_\alpha$}\label{f:gamma_b is a loop}
\end{figure}
\begin{lemma}\label{l:p_alpha ne p_beta}
    Let $\gamma_a,\gamma_b,\gamma_c\in\mathbf{R}$ with $$\angle\alpha\in\ha_{p_\alpha,\mathbf{P-free}}(\gamma_b,\gamma_c),\quad \angle\beta\in\ha_{p_\beta, \mathbf{P}-\mathbf{free}}(\gamma_a,\gamma_b).$$
    Then $f_{\angle\alpha} f_{\angle\beta}=0$ in each of the following cases:
    \begin{itemize}
        \item [(1)] $p_\beta\ne p_\alpha$ ({\em cf.} Figure~\ref{f:relation in A_R, endpoint are disctinct});
        \item[(2)]$\gamma_b$ is a loop and $\angle\beta$, $\angle\alpha$ are positioned as Figure~\ref{f:gamma_b is a loop}.
    \end{itemize}
   
\end{lemma}
\begin{proof}
     Let
$\widetilde{\gamma_{a}}$ be the arc segments of $\gamma_a$, let $\widetilde{\gamma_{b,1}},~\widetilde{\gamma_{b,2}}$ be the arc segments of $\gamma_b$ and let $\widetilde{\gamma_{c}}$ the arc segments of $\gamma_c$ such that  $f_{\angle\beta}$ is induced by the pair $(\widetilde{\gamma_{a}},\widetilde{\gamma_{b,1}})$ and
 $f_{\angle\alpha}$ is induced by the pair $(\widetilde{\gamma_{b,2}},\widetilde{\gamma_{c}})$. Thus there are orientations of $\widetilde{\gamma_{a}}$, $\widetilde{\gamma_{b,1}},~\widetilde{\gamma_{b,2}}$ and $\widetilde{\gamma_{c}}$  such that
 \[(\widetilde{\gamma_{a}},~\widetilde{\gamma_{b,1}})\in\aad(\gamma_a)\times\cad(\gamma_b), \Intv(\widetilde{\gamma_{a}})=\Intv(\widetilde{\gamma_{b,1}}), \widetilde{\gamma_{a}}\sim\widetilde{\gamma_{b,1}},~\widetilde{\gamma_{a}}(0)=\widetilde{\gamma_{b,1}}(0)=p_{\beta},\]and
 \[(\widetilde{\gamma_{b,2}},~\widetilde{\gamma_{c}})\in\aad(\gamma_b)\times\cad(\gamma_c),\Intv(\widetilde{\gamma_{b,2}})=\Intv(\widetilde{\gamma_{c}}), \widetilde{\gamma_{b,2}}\sim\widetilde{\gamma_{c}},~~\widetilde{\gamma_{b,2}}(0)=\widetilde{\gamma_{c}}(0)=p_{\alpha}.\]
Suppose $f_{\angle\alpha} f_{\angle\beta}\neq 0$. Then there exist three arc segments 
 $\widetilde{\gamma_{a,1}},\widetilde{\gamma_{b,3}}, \widetilde{\gamma_{c,1}}$ such that 
 \begin{itemize}
     \item $\widetilde{\gamma_{a,1}}\in\aad(\gamma_{a})$ is a sub-arc segment of $\widetilde{\gamma_{a}}$;
     \item  $\widetilde{\gamma_{b,3}}$ is a sub-arc segment of $\widetilde{\gamma_{b,1}}$ and a sub-arc segment of $\widetilde{\gamma_{b,2}}$;
     \item  $\widetilde{\gamma_{c,1}}\in\cad(\gamma_{c})$ is a sub-arc segment of $\widetilde{\gamma_{c}}$;
 \end{itemize}
   satisfying $$ \Intv(\widetilde{\gamma_{a,1}})=\Intv(\widetilde{\gamma_{b,3}})=\Intv(\widetilde{\gamma_{c,1}}), \quad \widetilde{\gamma_{a,1}}\sim\widetilde{\gamma_{b,3}}\sim\widetilde{\gamma_{c,1}}$$
 and $f_{\angle\alpha} f_{\angle\beta}$ is induced by $(\widetilde{\gamma_{a,1}},~\widetilde{\gamma_{c,1}}).$
Note that by Theorem~\ref{t:basis of rad fromed by angles}, we must have either $\widetilde{\gamma_{a,1}}(0)\in\MM$ or   $\widetilde{\gamma_{a,1}}(1)\in\MM$.
However, 
in case $(1)$, both $\widetilde{\gamma_{b,3}}(0)$ and $\widetilde{\gamma_{b,3}}(1)$ are not marked points, so
 $\widetilde{\gamma_{a,1}}(0)\notin\MM,  \widetilde{\gamma_{a,1}}(1)\notin\MM$,  which contradicts the requirement from Theorem~\ref{t:basis of rad fromed by angles}. Hence we must hve  $f_{\angle\alpha} f_{\angle\beta}=0$ in this case.
In case $(2)$, note that $\gamma_b$ is a loop,  by the proof of  Lemma~\ref{l:relation for gamma_a is a loop},
one can also get that   $\widetilde{\gamma_{b,3}}(0)\notin \MM$ and $\widetilde{\gamma_{b,3}}(1)\notin\MM$. So we have $\widetilde{\gamma_{a,1}}(0)\notin \MM$ and $\widetilde{\gamma_{a,1}}(1)\notin \MM$,  which contradicts the requirement from Theorem~\ref{t:basis of rad fromed by angles}. Thus we must have $f_{\angle\alpha} f_{\angle\beta}=0$.
\end{proof}

\begin{lemma}\label{l:conditions for if and only if p_alphap_beta neq 0}
    Let $\gamma_a,\gamma_b,\gamma_c\in\mathbf{R}$ with $$\angle\alpha_1\in\ha_{p_1, \mathbf{P-free}}(\gamma_a,\gamma_b),\quad \angle\alpha_2\in\ha_{p_2,\mathbf{P-free}}(\gamma_b,\gamma_c).$$
Then $f_{\angle\alpha_2} f_{\angle\alpha_1}\neq 0$ if and only if  $p_1= p_2$ and  $\angle\alpha_2$ is adjacent to $\angle\alpha_1$ with $\angle\alpha_2$ is on the clockwise side of $\angle\alpha_1$ (cf.  Figure \ref{f:the location for k,j,i-}).
   \end{lemma}
\begin{proof} If  $p_1=p_2$ and $\angle\alpha_2$ is adjacent to $\angle\alpha_1$ with $\angle\alpha_2$ is on the clockwise side of $\angle\alpha_1$, then Lemma \ref{l:f_alpha_3 can be factored} yields $$f_{\angle\alpha_2} f_{\angle\alpha_1}=f_{\angle\alpha_2+\angle\alpha_1}\neq 0.$$
Conversely, suppose $f_{\angle\alpha_2} f_{\angle\alpha_1}\neq 0$. By Lemma~\ref{l:p_alpha ne p_beta}, we must have $p_1=p_2$. Let $p_1=p_2=p$. 

If $\gamma_b$ is not a loop, then the definition of the two angles imply that $\angle\alpha_2$ must be  adjacent to $\angle\alpha_1$ with $\angle\alpha_2$ is on the clockwise side of $\angle\alpha_1$.

It remains to consider the case where $\gamma_b$ is a loop. 
   \begin{enumerate}
       \item If $\gamma_a=\gamma_b=\gamma_c$, then all three arcs coincide. Because  $\gamma_b$ is a loop, so the set $\ha_{p, \mathbf{P}-\textbf{free}}(\gamma_b,\gamma_b)$ contains exactly one element. Then we must have   $\angle\alpha_2=\angle\alpha_1$. Lemma~\ref{l:relation for gamma_a is a loop} gives $$f_{\angle\alpha_2} f_{\angle\alpha_1}=f_{\angle\alpha_1} f_{\angle\alpha_1}=0,$$ a contradiction.
  \item If $\gamma_a=\gamma_b\neq \gamma_c$, then 
        \[
        \angle\alpha_1\in\ha_{p, \mathbf{P}-\mathbf{free}}(\gamma_b,\gamma_b),\quad 
        \angle\alpha_2\in\ha_{p,\mathbf{P}-\mathbf{free}}(\gamma_b,\gamma_c).
        \]
        We consider two subcases.
        
        \textbf{Case 2a:} $\angle\alpha_1$ is a sub-angle of $\angle\alpha_2$ (\ie the fan of  $\angle\alpha_1$ is a sub-fan  of the fan of $\angle\alpha_2$). Then $\angle\alpha_2=\angle\alpha_1+\angle\alpha_3$ for some $\angle\alpha_3\in\ha_{p,\mathbf{P}-\mathbf{free}}(\gamma_b,\gamma_c)$, where $\angle\alpha_3$ is adjacent to $\angle\alpha_1$ with $\angle\alpha_3$ on the clockwise side of $\angle\alpha_1$. By Lemmas~\ref{l:f_alpha_3 can be factored} and \ref{l:relation for gamma_a is a loop},
        \[
        f_{\angle\alpha_2}f_{\angle\alpha_1}=f_{\angle\alpha_3}f_{\angle\alpha_1}f_{\angle\alpha_1}=f_{\angle\alpha_3}f_{\angle\alpha_1}^2=0,
        \]
        again a contradiction.
        
        \textbf{Case 2b:} $\angle\alpha_1$ is not a sub-angle of $\angle\alpha_2$. Then $\angle\alpha_2$ must be adjacent to $\angle\alpha_1$ with $\angle\alpha_2$ on the clockwise side of $\angle\alpha_1$, as desired.

\item If $\gamma_a\neq \gamma_b= \gamma_c$, the argument is symmetric to the case above.

 \item Suppose $\gamma_a\neq \gamma_b$ and $\gamma_b\neq \gamma_c$. Let $\angle\alpha_3\in \ha_{p, \mathbf{P}-\mathbf{free}}(\gamma_b,\gamma_b)$. 
        
        If exactly one of $\angle\alpha_1$ and $\angle\alpha_2$ is a sub-angle of $\angle\alpha_3$, then $\angle\alpha_2$ must be adjacent to $\angle\alpha_1$ with $\angle\alpha_2$ on the clockwise side of $\angle\alpha_1$.
        
        Otherwise, the configuration of $\angle\alpha_1$ and $\angle\alpha_2$ can only be as shown in Figure~\ref{f:gamma_b is a loop}, where $\angle\beta=\angle\alpha_1$ and $\angle\alpha=\angle\alpha_2$. By Lemma~\ref{l:p_alpha ne p_beta}, we have $f_{\angle\alpha_2} f_{\angle\alpha_1}=0$, contradicting our assumption.
\end{enumerate}
   Therefore the lemma holds.
\end{proof}

\begin{lemma}\label{l:fp=0=>fafb=0}
Let $\gamma_1,\cdots,\gamma_{s+1}\in\mathbf{R}$ with $s\ge 2$.
  Suppose that for each $1\leq i\leq s$, there exists an angle $$\angle\alpha_i\in\apfree(\gamma_i,\gamma_{i+1}),$$ and suppose further that $$f_{\angle\alpha_s}\cdots f_{\angle\alpha_1}=0.$$ Then there exists some $1\leq t\leq s-1$ such that $$f_{\angle\alpha_{t+1}}f_{\angle\alpha_t}=0.$$
\end{lemma}
\begin{proof} Suppose, for contradiction, that $f_{\angle\alpha_{t+1}}f_{\angle\alpha_t}\ne0$ for every $1\le t\le s-1$. Then by Lemma~\ref{l:conditions for if and only if p_alphap_beta neq 0} and Lemma \ref{l:f_alpha_3 can be factored}, all angles  share a common vertex $p\in\MM$, and the sum $$\angle\alpha_1+\dots+\angle\alpha_s\in\angle^{-}_p(\gamma_1,\gamma_{s+1})$$ is well-defined. Moreover, since each $\angle\alpha_i\in\appfree(\gamma_i,\gamma_{i+1})$, we have $$\angle\alpha_1+\dots+\angle\alpha_s\in\appfree(\gamma_1,\gamma_{s+1}).$$ Combining  Lemma \ref{l:f_alpha_3 can be factored} and Theorem~\ref{t:basis of rad fromed by angles}, we obtain \[f_{\angle\alpha_s}\cdots f_{\angle\alpha_1}=f_{\angle\alpha_1+\dots+\angle\alpha_s}\ne0,\]
which contradicts the hypothesis.
\end{proof}

\subsection{$\textbf{PR-free}$ negatively oriented  angles}
Let $\mathbf{R}$ be a partial dissection on $(\SS,\MM,\PP)$.
\begin{definition}($\mathbf{PR}$-$\mathbf{free}$ angle)
Let $\gamma_1, \gamma_2\in\mathbf{R}$ share a common endpoint $p$, 
and let $\angle\alpha$ be an oriented angle from $\gamma_1$ to $\gamma_2$ at $p$. We say that
 $\angle\alpha$ is \emph{$\mathbf{PR}$-$\mathbf{free}$} if the fan of $\angle\alpha$ contains no interior arcs from $\PP\cup \mathbf{R}$. 
\end{definition}

It is clear that  $\angle\alpha$ is {$\mathbf{PR}$-$\mathbf{free}$} means $\angle \alpha$ is indecomposable with respect to $\PP\cup\RR$.  Let $\gamma_i, \gamma_j\in\mathbf{R}$ share a common endpoint $p$. We denote by
\begin{itemize}
       \item $\ha_{p,\textbf{PR-free}}(\gamma_i,\gamma_j)$ the set of all $\mathbf{PR}$-$\textbf{free}$ negatively oriented angles from $\gamma_i$ to $\gamma_j$ at $p$;
    \item $\ha_{\textbf{PR-free}}(\gamma_i,\gamma_j)$ the union of all $\ha_{p,\textbf{PR-free}}(\gamma_i,\gamma_j)$ over common endpoints $p$ of $\gamma_i$ and $\gamma_j$.
\end{itemize}

\subsection{Irreducible morphisms and $\textbf{PR-free}$ negatively oriented angles}

Let $\mathbf{R}=\{\gamma_1,\dots,\gamma_n\}$ be a faithful dissection on $(\SS,\MM,\PP)$. 
So the $A_{\PP}$-module 
\[
M(\mathbf{R})=\bigoplus_{i=1}^{n} M(\gamma_i)
\]
is a  tilting $A_{\PP}$-module.
Denote by $\End_{A_{\PP}} M(\mathbf{R})$ the endomorphism algebra of $M(\mathbf{R})$, and let $\xi_i$ denote the identity map of $M(\gamma_i)$. 
Set 
\[
\mathscr{B}=\{\xi_i \mid 1 \le i \le n\}\cup\{f_{\angle\alpha} \mid \angle\alpha\in\ha_{p, \mathbf{P}-\mathbf{free}}(\gamma_i,\gamma_j),\ 1\le i,j\le n\}.
\]
By  Theorem~\ref{t:basis of rad fromed by angles}, $\mathscr{B}$ is a basis of $\End_{A_{\PP}} M(\mathbf{R})$. 
Let 
\[
\mathscr{B}'=\{f_{\angle\alpha} \mid \angle\alpha\in\ha_{p, \mathbf{PR}-\mathbf{free}}(\gamma_i,\gamma_j),\ 1\le i,j\le n\}.
\]
By Lemma~\ref{l:f_alpha_3 can be factored} and  Theorem~\ref{t:basis of rad fromed by angles}, the algebra $\End_{A_{\PP}} M(\mathbf{R})$ is generated by $\{\xi_i \mid 1\le i\le n\}\cup \mathscr{B}'$, where $\{\xi_i \mid 1\le i\le n\}$ forms a complete set of primitive orthogonal idempotents. 

Since $f_{\angle\alpha}$ is uniquely determined by $\angle\alpha$, we identify $f_{\angle\alpha}$ with $\angle\alpha$. 
 Consequently, $$\End_{A_{\PP}} M(\mathbf{R}) \cong K\widetilde{Q}_{\mathbf{R}}/\langle \widetilde{I}_{\mathbf{R}}\rangle,$$ 
where the quiver $\widetilde{Q}_{\mathbf{R}}$ is defined as follows. The vertex set is $(\widetilde{Q}_{\mathbf{R}})_0=\{1,2,\dots,n\}$, corresponding to the arcs in $\mathbf{R}$.
The arrow set $(\widetilde{Q}_{\mathbf{R}})_1$ is given by $\mathscr{B}'$: there is an arrow $\alpha$ from $j\to i$ if and only if there exists an angle $\angle\alpha\in\angle_{p,\textbf{PR-free}}^{-}(\gamma_i,\gamma_j)$ for some common endpoint $p$.

Let $w = \alpha_1 \cdots \alpha_s$ be a path in $\widetilde{Q}_{\mathbf{R}}$ of length $s \ge 1$. 
We define the corresponding endomorphism $f_w \in \End_{A_{\PP}} M(\mathbf{R})$ by 
\[
f_w := f_{\alpha_1} \cdots f_{\alpha_s} = f_{\angle\alpha_1} \cdots f_{\angle\alpha_s}.
\]
(For a path $w$ of length $0$ at vertex $i$, we set $f_w := \xi_i$.) 
This definition naturally extends the assignment on arrows to paths. 
Consequently, the $K$-linear map sending each path $w$ in the path algebra $K\widetilde{Q}_{\mathbf{R}}$ to $f_w$ induces the algebra isomorphism $\End_{A_{\PP}} M(\mathbf{R}) \cong K\widetilde{Q}_{\mathbf{R}}/\langle \widetilde{I}_{\mathbf{R}}\rangle$.

We now prove that $(\widetilde{Q}_{\mathbf{R}},\widetilde{I}_{\mathbf{R}})$ is a gentle pair. For each loop $\epsilon\in (\widetilde{Q}_{\mathbf{R}})_1$, Lemma~\ref{l:relation for gamma_a is a loop} gives $\epsilon^2\in \widetilde{I}_{\mathbf{R}}$.
 For arrow $\alpha,\beta\in (\widetilde{Q}_{\mathbf{R}})_1$ satisfying 
any of conditions of  Lemma \ref{l:p_alpha ne p_beta}, we have $\alpha\beta\in  \widetilde{I}_{\mathbf{R}}$. Analogously to the definition of a tiling algebra, and using  Lemma \ref{l:f_alpha_3 can be factored}, \ref{l:relation for gamma_a is a loop}, \ref{l:p_alpha ne p_beta} and \ref{l:conditions for if and only if p_alphap_beta neq 0}, we obtain the following result.
\begin{lem}\label{l: Q^R satisfy g1-g3}
   Let $v\in (\widetilde{Q}_{\mathbf{R}})_0$.
   \begin{itemize}
       \item [(1)]  There are at most two incoming arrows and two outgoing arrows at $v$.
       \item [(2)]Let $\beta_1,\beta_2$ be two incoming arrows at $v$, and let $\alpha$ be an outgoing arrow at $v$ in $(\widetilde{Q}_{\mathbf{R}})_1$. Then either $\beta_1\alpha\in \widetilde{I}_{\mathbf{R}}$ and $\beta_2\alpha\notin \widetilde{I}_{\mathbf{R}}$ or $\beta_1\alpha\notin \widetilde{I}_{\mathbf{R}}$ and $\beta_2\alpha\in \widetilde{I}_{\mathbf{R}}$.
       \item [(3)]Let $\beta$ be an incoming arrow at $v$, and let $\alpha_1,\alpha_2$ be two outgoing arrows at $v$  in $(\widetilde{Q}_{\mathbf{R}})_1$. Then either $\beta\alpha_1\in \widetilde{I}_{\mathbf{R}}$ and $\beta\alpha_2\notin \widetilde{I}_{\mathbf{R}}$ or $\beta\alpha_1\notin \widetilde{I}_{\mathbf{R}}$ and $\beta\alpha_2\in \widetilde{I}_{\mathbf{R}}.$
   \end{itemize}
   \end{lem}
   In the following, we want to prove that  $$\widetilde{I}_{\mathbf{R}}=\{\epsilon^2\mid \epsilon \text{ is a loop in } (\widetilde{Q}_{\mathbf{R}})_1\}\cup\{\alpha\beta\mid \alpha,\beta\in (\widetilde{Q}_{\mathbf{R}})_1 \text{ satisfying 
any of conditions of  Lemma } \ref{l:p_alpha ne p_beta}\}.$$
By Lemma~\ref{l:fp=0=>fafb=0}, we obtain the following.  
\begin{lem}\label{l:p=0=> ab=0}
Let $p=\alpha_1\cdots\alpha_s$ be a path in $\widetilde{Q}_{\mathbf{R}}$ such that $p\in \widetilde{I}_{\mathbf{R}}$. Then there exists $1\leq i\leq s-1$ such that $\alpha_i\alpha_{i+1}\in \widetilde{I}_{\mathbf{R}}$.
\end{lem}

Moreover, we have the following stronger statement.
\begin{lem}\label{l: only zero relation}
Let $c,d\in (\widetilde{Q}_{\mathbf{R}})_0$, and let $p_1,\cdots,p_l$ ($l\ge 2$) be distinct paths in $\widetilde{Q}_{\mathbf{R}}$ from $d$ to $c$ such that $$\lambda_1{p_1}+\cdots+\lambda_l{p_l}\in \widetilde{I}_{\mathbf{R}},$$where $0\neq \lambda_i\in K$ for each $1\leq i\leq l$. Then $p_i\in \widetilde{I}_{\mathbf{R}}$ for every $1\leq i\leq l$.
\end{lem}
\begin{proof}



Let $\gamma_{1,1}, \dots, \gamma_{1,s_1+1}, \dots, \gamma_{l,1}, \dots, \gamma_{l,s_l+1} \in \mathbf{R}$ be arcs such that $\gamma_{1,1}=\gamma_{2,1}=\dots= \gamma_{l,1}=\gamma_c$ and $\gamma_{1,s_1+1}=\gamma_{2,s_2+1}=\dots=\gamma_{l,s_l+1}=\gamma_d$. Since each $p_i$ is a path from $d$ to $c$ in $\widetilde{Q}_{\mathbf{R}}$, we can then express the path $p_i$ as the composition $$p_i = \alpha_{i,s_i} \alpha_{i,s_i-1} \cdots \alpha_{i,1},$$ for $\alpha_{i,j}\in (\widetilde{Q}_{\RR})_1$ with $1\le j\le s_i$.
In particular, there exists a fan  whose start arc is $\gamma_d$ and end arc is $\gamma_c$. For each $1 \le i \le l$ and $1 \le j \le s_i$, let $\angle\alpha_{i,j} \in \ha_{\textbf{PR-free}}(\gamma_{i,j}, \gamma_{i,j+1})$.

If $p_i\notin \widetilde{I}_{\mathbf{R}}$, 
$$f_{p_i}=f_{\alpha_{i,s_i}}\cdots f_{\alpha_{i1}}=f_{\angle\alpha_{i,s_i}}\cdots f_{\angle\alpha_{i1}}\neq 0.$$ Then $f_{p_i}\in\Hom_{A_{\PP}}(M(\gamma_c),M(\gamma_d))$.
By Lemma~\ref{l:conditions for if and only if p_alphap_beta neq 0},  $\angle\alpha_{ij}$ and $\angle\alpha_{i,j+1}$ are adjacent angles with $\angle\alpha_{i,j+1}$ is on the clockwise side of $\angle\alpha_{ij}$ for each $1\leq j\leq s_i-1$ ({\em cf.} Figure~\ref{f:puncture}). Denote their common endpoint by $v_{i}$, denote $\angle\alpha_i=\angle\alpha_{i,1}+\cdots+\angle\alpha_{i,s_i}$.
By Lemma~\ref{l:f_alpha_3 can be factored}, $$f_{p_i}=f_{\angle\alpha_{i,s_i}}\cdots f_{\angle\alpha_{i1}}=f_{\angle\alpha_{i,1}+\cdots+\angle\alpha_{i,s_i}}=f_{\angle\alpha_i}.$$  Denote by $\mathcal{B}'$    the  basis of $\mathbf{rad}(M(\gamma_c),M(\gamma_d))$. Then by Theorem~\ref{t:basis of rad fromed by angles}, $f_{p_i}=f_{\angle \alpha_i}\in \mathcal{B}'$. 

 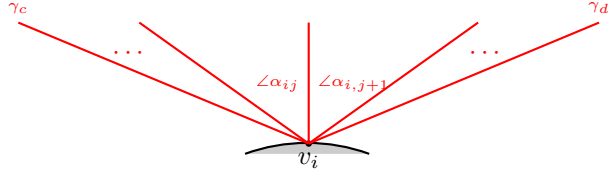
\begin{figure}[htpb]\centering
	\begin{tikzpicture}[scale=0.8]
    \draw[thick,fill=black!20] (7,-4.17)arc (70:110:3);
    \fill(6,-4) circle(1.5pt);
   \draw[red,thick] (6,-2)--(6,-4);
   \node at(6,-4)[below]{$v_i$};
    \draw[red,thick] (6,-4)--(10.8,-2);
    \draw[red,thick] (6,-4)--(8.8,-2);
      \draw[red,thick] (6,-4)--(1.2,-2);
    \draw[red,thick] (6,-4)--(3.2,-2);
    \node [red]at (8.5,-2.5)[right]{$\cdots$};
    \node [red]at (3.5,-2.5)[left]{$\cdots$};
    \node [red]at (1.2,-2)[above]{\tiny{$\gamma_{c}$}};
    \node [red]at (10.8,-2)[above]{\tiny{$\gamma_{d}$}};
  \node [red]at (6,-3)[left]{\tiny{$\angle\alpha_{ij}$}};
 \node [red]at (6,-3)[right]{\tiny{$\angle\alpha_{i,j+1}$}};
	\end{tikzpicture}
	\caption{$f_{p_i}=f_{\angle\alpha_i}$ if $p_i\notin \widetilde{I}_{\mathbf{R}}$}
	\label{f:puncture}
\end{figure}
Let $I = \{i \in \{1, \dots, l\} \mid p_i \notin \widetilde{I}_{\mathbf{R}}\}$.
Because  $\lambda_1{p_1}+\cdots+\lambda_l{p_l}\in \widetilde{I}_{\mathbf{R}}$, then we have 
\[\sum_{i \in I} \lambda_i f_{p_i} = \sum\limits_{i=1}^l\lambda_i f_{\angle\alpha_i}=0.\] 
By  Theorem~\ref{t:basis of rad fromed by angles}, we know that $$\{f_{p_i}|i\in I\}=\{f_{\angle \alpha_i}|i\in I\}\subset \mathcal{B}'.$$
Therefore, we have $\lambda_i=0$ for each $i\in I$. By assumption, $\lambda_i\neq 0$ for each $1\leq i\leq l$, then we must have $I=\emptyset$, which means 
$p_i\in \widetilde{I}_{\mathbf{R}}$ for each $1\leq i\leq l$. 
\end{proof}
Therefore, using  Lemma~\ref{l:p=0=> ab=0} and Lemma ~\ref{l: only zero relation}, 
we have $$\widetilde{I}_{\mathbf{R}}=\{\epsilon^2\mid \epsilon \text{ is a loop in } (\widetilde{Q}_{\mathbf{R}})_1\}\cup\{\alpha\beta\mid \alpha,\beta\in (\widetilde{Q}_{\mathbf{R}})_1 \text{ satisfying 
any of conditions of  Lemma } \ref{l:p_alpha ne p_beta}\}.$$
Then associated with Lemma~\ref{l: Q^R satisfy g1-g3},
$(\widetilde{Q}_{\mathbf{R}},\widetilde{I}_{\mathbf{R}})$ is a gentle pair.

\subsection{Tilting endomorphism algebra arising from a faithful dissection}

\begin{definition}\label{d:definition of A_R}
    Let $(\SS,\MM,\PP)$ be a tiling and let $\mathbf{R}$ be a faithful dissection of $(\SS,\MM,\PP)$. Then we can define a new finite-dimensional algebra $B_{\mathbf{R}}=KQ_{\mathbf{R}}/\langle I_{\mathbf{R}}\rangle $,
where:
\begin{itemize}
    \item $(Q_{\mathbf{R}})_0$: The vertex set $({Q}_{\mathbf{R}})_0$ corresponds bijectively to the permissible arcs in  $\mathbf{R}$;
     \item $(Q_{\mathbf{R}})_1$: An arrow $\alpha: j \to i$ belongs to $(Q_{\mathbf{R}})_1$ if and only if there exists an angle $\angle\alpha\in\ha_{\textbf{PR-free}}(\gamma_i,\gamma_j)$;

    \item the set $I_{\mathbf{R}}$ is defined analogously to $I_{\PP}$ (see Section \ref{s:def of tiling}): \begin{itemize}
        \item $\epsilon^2$ for every loop $\epsilon$;
        \item $\alpha\beta$ whenever $p_\alpha\neq p_\beta$, or the endpoints of the arc segment corresponding to $s(\beta)=t(\alpha)$ coincide ({\em cf.} Figure \ref{f:A_R algebra}).
    \end{itemize}
    
  \begin{figure}[h]
\begin{tikzpicture}

\draw[thick,fill=black!20] (2,0.17)arc (250:290:3);
\node at (3,0.2){$p_\alpha=p_\beta$};
\fill(3,0) circle(1.5pt);
\draw[thick,color=red] (3,0) parabola (1,-2);
\draw[thick,color=red,dashed] (1,-2) arc (150:260:1);
\draw[thick,color=cyan,->] (2,-0.5)--(2.5,-1);
\node at (2.7,-1.2){\tiny{$t(\beta)$}};
\node at (1.9,-0.2){\tiny{$s(\beta)$}};
\node at (2,-1) {$\beta$};
\draw[thick,color=red] (3,0) .. controls (2.6,-0.9) and (2.4,-1.1) .. (2,-2.2);
\draw[thick,color=red] (3,0) .. controls (3.4,-0.9) and (3.6,-1.1) .. (4,-2.2);
\draw[thick,color=red] (3,0) parabola (5,-2);
\draw[thick,color=red,dashed] (5,-2) arc (30:-70:1);
\draw[thick,color=cyan,->](3.5,-1)--(4,-0.5);
\draw [blue,thick](1.2,-0.7)..controls (2.5,-1.6) and (3.5,-1.6)..(4.7,-0.8);
\node at (4,-1) {$\alpha$};
\node at (3.3,-1.2){\tiny{$s(\alpha)$}};
\node at (4.1,-0.3){\tiny{$t(\alpha)$}};
\node [blue] at (0.6,-0.7){\tiny{$\af_i\in\PP$}};

\draw[thick,fill=black!20] (9,0.17)arc (250:290:3);
\node at (10,0.2){$p_\alpha=p_\beta$};
\fill(10,0) circle(1.5pt);
\draw[thick,color=red] (10,0) parabola (8,-2);
\draw[thick,color=red,dashed] (9,-2.2) arc (160:220:1);
\draw[thick,color=cyan,->](9,-0.5)--(9.5,-1);
\node at (9,-1) {$\alpha$};
\node at (9.7,-1.2){\tiny{$t(\alpha)$}};
\node at (8.9,-0.2){\tiny{$s(\alpha)$}};
\draw[thick,color=red] (10,0) .. controls (9.6,-0.9) and (9.4,-1.1) .. (9,-2.2);
\draw[thick,color=red] (10,0) .. controls (10.4,-0.9) and (10.6,-1.1) .. (11,-2.2);
\draw[thick,color=red] (10,0) parabola (12,-2);
\draw[thick,color=red,dashed] (11,-2.2) arc (20:-40:1);
\draw[thick,color=cyan,->](10.5,-1)--(11,-0.5);
\node at (11,-1) {$\beta$};
\node at (10.3,-1.2){\tiny{$s(\beta)$}};
\node at (11.1,-0.3){\tiny{$t(\beta)$}};
\draw [blue,thick](8.2,-0.7)..controls (9.5,-1.6) and (10.5,-1.6)..(11.7,-0.8);
\node [blue] at (7.6,-0.7){\tiny{$\af_i\in\PP$}};
\end{tikzpicture}
\caption{Case $\alpha\beta=0$ when $p_\alpha=p_\beta$}\label{f:A_R algebra}
\end{figure}
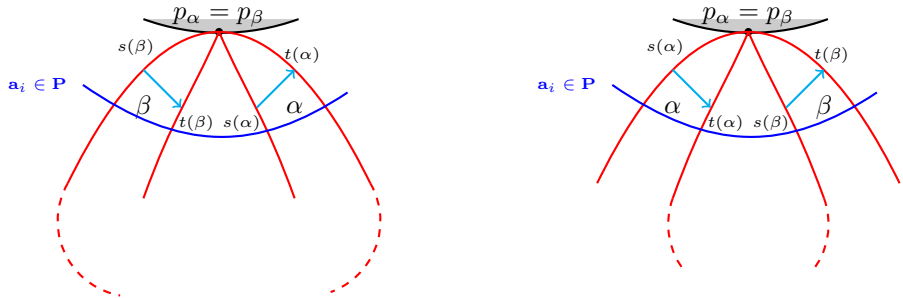
\end{itemize}
\end{definition}

\begin{theorem}\label{t:iso for end}
    Let $(\SS,\MM,\PP)$ be a tiling, $\mathbf{R}$ be a faithful dissection of $(\SS,\MM,\PP)$. Then $$B_{\RR}\cong \End_{A_{\PP}}M(\mathbf{R}).$$
\end{theorem}
\begin{proof}
Suppose that $\End_{A_{\PP}} M(\mathbf{R}) \cong K\widetilde{Q}_{\mathbf{R}}/\langle \widetilde{I}_{\mathbf{R}}\rangle$, by the constructions of $\widetilde{Q}_{\mathbf{R}}$ and $Q_{\RR}$, it is clear that $\widetilde{Q}_{\mathbf{R}}=Q_{\RR}$. By Lemma \ref{l:p_alpha ne p_beta} and the constructions of $\widetilde{I}_{\mathbf{R}}$ and $I_{\RR}$, we have $\widetilde{I}_{\mathbf{R}}={I}_{\mathbf{R}}$. Thus we have
$(\widetilde{Q}_{\mathbf{R}},\widetilde{I}_{\mathbf{R}})=(Q_{\RR},I_{\RR})$, which means $B_{\RR}\cong \End_{A_{\PP}}M(\mathbf{R}).$
\end{proof}

\begin{ex}\label{e:example for A_R}
   Let $(\SS,\MM,\PP)$ be a tiling and $\mathbf{R}$ a faithful dissection as shown in the Figure \ref{f:example for A_R}.
    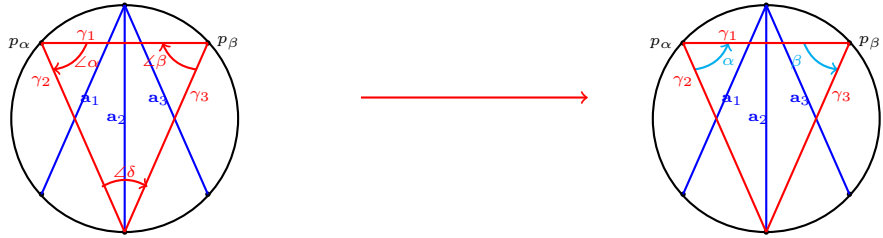
\begin{figure}[h]
       \begin{minipage}[t]{0.4\textwidth}
       \centering
      \begin{tikzpicture}[scale=0.5]
          	\draw[ thick] (0,0) circle (3);
            \fill(0,-3) circle(2pt);
            \fill(0,3) circle(2pt);
            \fill(-2.2,-2) circle(2pt);
             \fill(-2.2,2) circle(2pt);
            \fill(2.2,-2) circle(2pt);
             \fill(2.2,2) circle(2pt);
            \draw[blue,thick](0,3)to(0,-3);
             \draw[blue,thick](0,3)to(-2.2,-2);
             \draw[blue,thick](0,3)to(2.2,-2);
             \node [blue]at (-0.9,0.5){\tiny{$\af_1$}};
             \node at (-2.2,2) [left] {\tiny{$p_\alpha$}};
             \node at (2.2,2) [right] {\tiny{$p_\beta$}};
             \node [blue]at (-0.2,0){\tiny{$\af_2$}};
             \draw[red,thick, bend left=30,->](-1,2) to(-1.9,1.3);
             \node[red] at(-1,1.5){\tiny{$\angle\alpha$}};
             \node[red] at(0.8,1.5){\tiny{$\angle\beta$}};
              \draw[red,thick, bend left=30,->] (1.9,1.3) to(1,2);
            \draw[red,thick, bend right=-30,->] (-0.6,-1.8)to (0.6,-1.8);
            \node[red] at(0,-1.8)[above]{\tiny{$\angle\delta$}};
             \node [blue]at (0.9,0.5){\tiny{$\af_3$}};
             \draw[red,thick](-2.2,2)to(2.2,2);
             \draw[red,thick](-2.2,2)to(0,-3);
             \draw[red,thick](0,-3)to(2.2,2);
             \node[red]at (-1,2.2){\tiny{$\gamma_1$}};
             \node[red]at (-2.2,1){\tiny{$\gamma_2$}};
             \node[red]at (2,0.6){\tiny{$\gamma_3$}};
      \end{tikzpicture}
  \end{minipage}
  \begin{minipage}[t]{0.15\linewidth}
    \begin{tikzpicture}[xscale=0.6,yscale=0.6]
\draw[red,thick,->] (0,2) to (5,2);
\draw[white,thick,->] (0,-1) to (2,-1);
	\end{tikzpicture}
    \end{minipage}
\begin{minipage}[t]{0.4\textwidth}
\centering
      \begin{tikzpicture}[scale=0.5]
          	\draw[ thick] (0,0) circle (3);
            \fill(0,-3) circle(2pt);
            \fill(0,3) circle(2pt);
            \fill(-2.2,-2) circle(2pt);
             \fill(-2.2,2) circle(2pt);
            \fill(2.2,-2) circle(2pt);
             \fill(2.2,2) circle(2pt);
            \draw[blue,thick](0,3)to(0,-3);
             \draw[blue,thick](0,3)to(-2.2,-2);
             \draw[blue,thick](0,3)to(2.2,-2);
             \node [blue]at (-0.9,0.5){\tiny{$\af_1$}};
             \node at (-2.2,2) [left] {\tiny{$p_\alpha$}};
             \node at (2.2,2) [right] {\tiny{$p_\beta$}};
             \node [blue]at (-0.2,0){\tiny{$\af_2$}};
             \draw[cyan,thick,->, bend right=30](-1.9,1.3) to(-1,2);
             \node[cyan] at(-1,1.5){\tiny{$\alpha$}};
             \node[cyan] at(0.8,1.5){\tiny{$\beta$}};
              \draw[cyan,thick,->, bend right=30](1,2) to(1.9,1.3);
             \node [blue]at (0.9,0.5){\tiny{$\af_3$}};
             \draw[red,thick](-2.2,2)to(2.2,2);
             \draw[red,thick](-2.2,2)to(0,-3);
             \draw[red,thick](0,-3)to(2.2,2);
             \node[red]at (-1,2.2){\tiny{$\gamma_1$}};
             \node[red]at (-2.2,1){\tiny{$\gamma_2$}};
             \node[red]at (2,0.6){\tiny{$\gamma_3$}};
      \end{tikzpicture}
      \end{minipage}
       \caption{Example for $B_{\RR}$}
       \label{f:example for A_R}
   \end{figure}
   
   It is clear that there exist angles $\angle\alpha\in\ha_{p_\alpha,\textbf{PR-free}}(\gamma_1,\gamma_2)$ and $\angle\beta\in\ha_{p_\beta,\textbf{PR-free}}(\gamma_3,\gamma_1)$, which give rise to two arrows $\alpha:2\to1$ and $\beta:1\to3$. Additionally, since $a_2\in\PP$ lies between $\gamma_2$ and $\gamma_3$, there are no $\textbf{PR-free}$ angles between $\gamma_2$ and $\gamma_3$; consequently, there is no arrow between $2$ and $3$. Furthermore, since $p_\alpha\ne p_\beta$, by definition, we have $B_{\RR}=KQ_{\RR}/\langle I_{\RR}\rangle$, where
  \[Q_{\RR}: \xymatrix{2\ar[r]^{\alpha}&1\ar[r]^{\beta}&3}\quad \text{and}\quad I_{\RR}=\{\alpha\beta\}.\]
\end{ex}
\begin{rk}
   If $\RR$ is a partial dissection, \ie $M(\RR)$ is a $\tau$-rigid $A_\PP$-module, the above results still hold.
\end{rk}

\section{More derived equivalent algebras}
Let $(\SS,\MM,\PP)$ be a tiling, $\RR\in \FD(\SS)$ be a faithful dissection on  $(\SS,\MM,\PP)$. Now we know that $B_{\RR}$ is derived equivalent to $A_{\PP}$. In this section, we provide two ways to get more algebras which are derived equivalent to $A_{\PP}$ from $\RR$.
\subsection{Tilting flip}\label{s:tilting flip}
In this subsection, we introduce a particular flip that yields another faithful dissection from a given one.
\begin{definition}[{\cite[Definition 3.7]{HZZ}}]
Let $(\SS,\MM,\PP)$ be a tiling. An arc $\gamma$ is called a {\it generalized permissible arc} if either $\gamma$ is a permissible arc or $\gamma \in \PP$. A  {\it generalized dissection} is a  set $\mathbf{R}$ of maximal  generalized permissible arcs  such that no arcs in  $\mathbf{R}$ have self-intersections and any two arcs in  $\mathbf{R}$ have no intersections with each other.
\end{definition}
\begin{definition}
   Let $\mathbf{R}$ be a generalized dissection. For any $\gamma \in\mathbf{R}$, if there exists a generalized permissible arc $\gamma' \neq \gamma$ such that $\mu_\gamma(\RR) := (\RR \setminus \{\gamma\}) \cup \{\gamma'\}$ is again a generalized dissection, we call $\mu_\gamma(\RR)$ the {\it flip} of $\mathbf{R}$ at $\gamma$. In particular, $\mu_\gamma(\mathbf{R})$ is a {\it tilting flip} if both $\mathbf{R}$ and $\mu_\gamma(\RR)$ are faithful dissections.
\end{definition}

\begin{ex}
Let $(\SS,\MM,\PP)$ be the tiling and $\mathbf{R}$ be the faithful dissection in Example \ref{e:example for A_R}, Figure \ref{f:ex for flip at 2} shows the tilting flip at $\gamma_2$. 
     \begin{figure}[h]
     \begin{minipage}[t]{0.4\linewidth}
\centering
	\begin{tikzpicture}[scale=0.5]
     	\draw[ thick] (0,0) circle (3);
            \fill(0,-3) circle(2pt);
            \fill(0,3) circle(2pt);
            \fill(-2.2,-2) circle(2pt);
             \fill(-2.2,2) circle(2pt);
            \fill(2.2,-2) circle(2pt);
             \fill(2.2,2) circle(2pt);
            \draw[blue,thick](0,3)to(0,-3);
             \draw[blue,thick](0,3)to(-2.2,-2);
             \draw[blue,thick](0,3)to(2.2,-2);
             \node [blue]at (-0.9,0.5){\tiny{$\af_1$}};
             \node [blue]at (-0.2,0){\tiny{$\af_2$}};
             \node [blue]at (0.9,0.5){\tiny{$\af_3$}};
             \draw[red,thick](-2.2,2)to(2.2,2);
             \draw[red,thick](-2.2,2)to(0,-3);
             \draw[red,thick](0,-3)to(2.2,2);
             \node[red]at (-1,2.2){\tiny{$\gamma_1$}};
             \node[red]at (-2.2,1){\tiny{$\gamma_2$}};
             \node[red]at (1.8,0.6){\tiny{$\gamma_3$}};
      \end{tikzpicture}
      \end{minipage}
      \begin{minipage}[t]{0.15\linewidth}
    \begin{tikzpicture}[xscale=0.6,yscale=0.6]
\draw[red,thick,->] (0,2) to (5,2);
\draw[white,thick,->] (0,-1) to (2,-1);
\draw [red,thick] (2.5,2)node[above]{\tiny{$\mu_{\gamma_2}(\RR)$}};
	\end{tikzpicture}
    \end{minipage}
        \begin{minipage}[t]{0.4\linewidth}
    \centering
	\begin{tikzpicture}[scale=0.5]
          	\draw[ thick] (0,0) circle (3);
            \fill(0,-3) circle(2pt);
            \fill(0,3) circle(2pt);
            \fill(-2.2,-2) circle(2pt);
             \fill(-2.2,2) circle(2pt);
             \fill(2.2,-2) circle(2pt);
             \fill(2.2,2) circle(2pt);
             \draw[blue,thick](0,3)to(0,-3);
             \draw[blue,thick](0,3)to(-2.2,-2);
             \draw[blue,thick](0,3)to(2.2,-2);
             \node [blue]at (-0.9,0.5){\tiny{$\af_1$}};
             \node [blue]at (-0.2,0){\tiny{$\af_2$}};
             \node [blue]at (0.9,0.5){\tiny{$\af_3$}};
             \draw[red,thick](-2.2,2)to(2.2,2);
             \draw[red,thick](2.2,2)to(-2.2,-2);
             \draw[red,thick](0,-3)to(2.2,2);
             \node[red]at (-1,2.2){\tiny{$\gamma_1$}};
             \node[red]at (-1.5,-2){\tiny{$\gamma_2$}};
             \node[red]at (1.8,0.6){\tiny{$\gamma_3$}};
      \end{tikzpicture}
      \end{minipage}
       \caption{Example for tilting flip at $\gamma_2$}
      \label{f:ex for flip at 2}
   \end{figure}
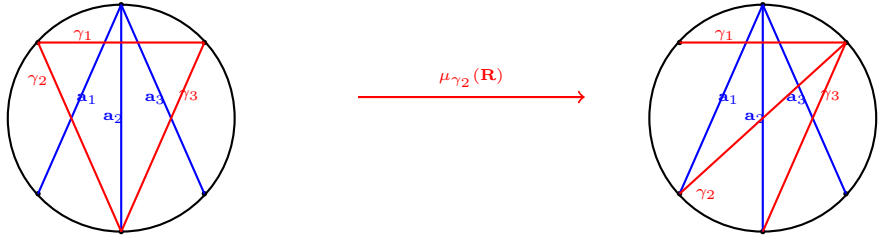
\end{ex}
The following result was obtained in \cite{Nie} (in Chinese) via the geometric model of tilting modules. Here we provide a simple alternative proof.
 \begin{theorem}\label{t:tilting flip}
     Let $(\SS, \MM, \PP)$ be a tiling and $\mathbf{R}$ be a faithful dissection. For an arc $\gamma \in\mathbf{R}$,  $M(\mu_\gamma(\RR))$ is a tilting module if and only if $\mathbf{R} \setminus \{\gamma\}$  is a partial faithful dissection.
 \end{theorem}
 \begin{proof}
 By Theorem~\ref{t: tilting modules and faithful dissection},  $M(\RR)$ is a tilting $A_\PP$-module. Denote by $\overline{\RR}=\RR\setminus\{\gamma\}$, so $M(\overline{\RR})$ is an almost tilting $A_{\PP}$-module. By \cite[Proposition 1.2 ]{HU05}, $M(\overline{\RR})$ has two complements if and only if $M(\overline{\RR})$ is faithful. By Lemma~\ref{l:PFD is faithful}, $M(\overline{\RR})$ is faithful if and only if $\overline{\RR}$ is a faithful partial dissection on $(\SS, \MM, \PP)$. Then we get the result.
 \end{proof}

\subsection{Tiling for a tilting endmorphism algebra}\label{s:geo model of endo algebra}

Let $\mathbf{R}$ be a faithful dissection of the tiling $(\SS,\MM,\PP)$. In this subsection, we construct a geometric model $(\SS_{\mathbf{R}}, \MM_{\mathbf{R}}, \PP_{\mathbf{R}})$ for $\End_{A_\PP}(M(\mathbf{R}))$. By considering faithful dissections of $(\SS_{\mathbf{R}}, \MM_{\mathbf{R}}, \PP_{\mathbf{R}})$, we can obtain further algebras that are derived equivalent to $A_\PP$.

\subsubsection{Separating arcs at non-$\PP$-$\mathbf{free}$ angles} 
Let $p_\alpha \in \MM$ be a marked point such that there exists an angle $\angle\alpha$ at $p_\alpha$ with both sides in $\mathbf{R}$ and $\angle\alpha$ is not $\PP$-free. Let $\af_1, \dots, \af_s \in \PP$ be the interior arcs of the fan of $\alpha$, ordered so that $\af_{i+1}$ lies on the clockwise side of $\af_i$ for each $1 \leq i \leq s-1$.

Let $\omega_l, \dots, \omega_1, \af_s, \gamma_1, \gamma_2, \dots, \gamma_r \in \PP \cup \mathbf{R}$ be the complete fan at $p_\alpha$, thus $\omega_l, \dots, \omega_1$ lie on the anticlockwise side of $\af_s$ and $\gamma_1, \gamma_2, \dots, \gamma_r$ lie on the clockwise side of $\af_s$. To separate $\{\gamma_1, \dots, \gamma_r\}$ from $\{\omega_1, \dots, \omega_l\}$, we perform the following surgery:

\begin{itemize}
    \item[\textbf{Step 1.}] Let $p \in \MM$ be the marked point immediately following $p_\alpha$ in the clockwise order along $\partial \SS$. Insert a new marked point $v$ on $\partial \SS$ between $p_\alpha$ and $p$, so that the clockwise order becomes $p_\alpha, v, p$. Let $b$ denote the boundary segment between $p_\alpha$ and $v$.
    
    \item[\textbf{Step 2.}] Slide the endpoints of $\gamma_1, \dots, \gamma_r$ from $p_\alpha$ to $v$ along the boundary $\partial \SS$ (\confer Figure~\ref{f:split}). We keep the same notation $\gamma_1, \dots, \gamma_r$ for the resulting arcs. Note that the new $\gamma_1, \dots, \gamma_r$ still lie in $\PP \cup \mathbf{R}$, and all other structures remain unchanged. Consequently, these arcs detach from $p_\alpha$ and meet at the newly created vertex $v$ on the boundary.
\end{itemize}

We iterate this procedure over all marked points (including the newly created ones) and all angles with sides in $\mathbf{R}$ that are not $\PP$-free, until every angle with sides in $\mathbf{R}$ becomes $\PP$-free. Let $\MM_{\mathrm{add}}$ be the set of all newly created marked points obtained through this process, and set $\MM_{\mathrm{new}} = \MM \cup \MM_{\mathrm{add}}$. Correspondingly, we denote the updated collection of arcs by $\RR_{\mathrm{new}}$. 

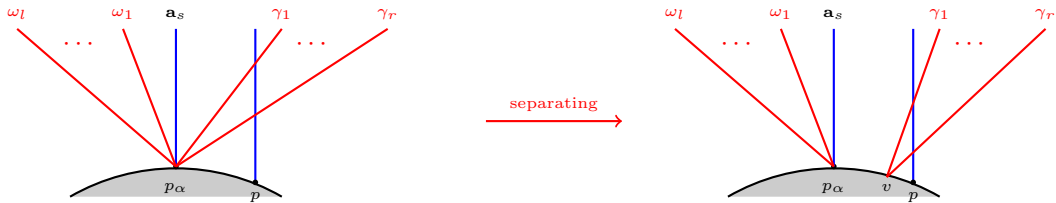
\begin{figure}[h]
     \centering
     \begin{minipage}{0.4\linewidth}
         \begin{tikzpicture}[scale=0.7)]
    \draw[thick,fill=black!20] (8,-4.17)arc (60:120:4);
    \node at (6,-4){\tiny{$p_{\alpha}$}};
    \fill(6,-3.6) circle(1.5pt);
    \draw[blue,thick] (6,-1)--(6,-3.6);
     \draw[blue,thick] (7.5,-3.9)--(7.5,-1);
    \node [red]at (5,-1)[above]{\tiny{$\omega_1$}};
       \node [red]at (3,-1)[above]{\tiny{$\omega_l$}};
    \node [red]at (8,-1)[above]{\tiny{$\gamma_1$}};
      \node [red]at (10,-1)[above]{\tiny{$\gamma_r$}};
    \node[red] at (4.2,-1) [below]{$\cdots$};
     \node[red] at (8.6,-1) [below]{$\cdots$};
     \node at (7.5,-3.9) [below]{\tiny{$p$}};
      \node at (6,-1) [above]{\tiny{$\af_s$}};
  
    \fill(7.5,-3.9)circle(1.5pt);
     \draw[red,thick] (6,-3.6)--(8,-1);
      \draw[red,thick] (6,-3.6)--(10,-1);
     \draw[red,thick] (6,-3.6)--(3,-1);
      \draw[red,thick] (6,-3.6)--(5,-1);
      
    \end{tikzpicture}
    \end{minipage}%
    \begin{minipage}{0.18\linewidth}
    \centering
    \begin{tikzpicture}
    \draw[red,thick,->] (0,0) -- (1.8,0) node[midway, above] {\tiny{$\text{separating}$}};
        \end{tikzpicture}
    \end{minipage}%
     \begin{minipage}{0.4\linewidth}
    \begin{tikzpicture}[scale=0.7]
    \draw[thick,fill=black!20] (8,-4.17)arc (60:120:4);
    \node at (6,-4){\tiny{$p_{\alpha}$}};
    \fill(6,-3.6) circle(1.5pt);
    \draw[blue,thick] (6,-1)--(6,-3.6);
     \draw[blue,thick] (7.5,-3.9)--(7.5,-1);
    \node [red]at (5,-1)[above]{\tiny{$\omega_1$}};
       \node [red]at (3,-1)[above]{\tiny{$\omega_l$}};
    \node [red]at (8,-1)[above]{\tiny{$\gamma_1$}};
      \node [red]at (10,-1)[above]{\tiny{$\gamma_r$}};
    \node[red] at (4.2,-1) [below]{$\cdots$};
     \node[red] at (8.6,-1) [below]{$\cdots$};
     \node at (7.5,-3.9) [below]{\tiny{$p$}};
      \node at (6,-1) [above]{\tiny{$\af_s$}};
      
    \fill(7.5,-3.9)circle(1.5pt);
   
     \draw[red,thick] (6,-3.6)--(3,-1);
      \draw[red,thick] (6,-3.6)--(5,-1);
 
      \node at(7,-3.75) [below]{\tiny{$v$}};
       \draw[red,thick] (7,-3.8)--(8,-1);
      \draw[red,thick] (7,-3.8)--(10,-1);
    \end{tikzpicture}
    \end{minipage}
     \caption{Separating}
     \label{f:split}
 \end{figure} 
\subsubsection{Merging consecutive boundary segments}

Consider the marked surface $(\SS,\MM_{\mathrm{new}},\RR_{\mathrm{new}})$.
For each basic tile  which has a sequence of consecutive boundary segments
$b_1,b_2,\dots,b_k$ along its boundary for some $k\ge 2$, merge these segments
into a single boundary segment by removing all points in the interior of the
union $b_1\cup b_2\cup\cdots\cup b_k$ (\confer Figure~\ref{f:merging}).
Denote by $\MM_{\mathrm{rem}}$ the set of marked points which have been removed.

\begin{figure}[h]
     \centering
     \begin{minipage}{0.25\linewidth}
         \begin{tikzpicture}[scale=0.7]
    \draw[thick,fill=black!20] (8,-4.17)arc (60:120:4);
   
    \node at (7,-3.6) [below]{\tiny{$b_k$}};
    \node at (5,-3.6) [below]{\tiny{$b_1$}};
     \node at (7.5,-3.9) [below]{\tiny{$p_2$}};
    \node at (4.5,-3.9) [below]{\tiny{$p_1$}};
    \fill(7.5,-3.9)circle(1.5pt);
    \fill(4.5,-3.9)circle(1.5pt);

    \fill(6,-3.6) circle(1.5pt);
    \fill(5.25,-3.7) circle(1.5pt);
 \fill(6.75,-3.7) circle(1.5pt);

    \draw[red,thick] (7.5,-3.9)--(8,-3);
    \node [red] at(8,-3)[above]{$\vdots$};
    \fill (8,-3)circle(1.5pt);
    \fill(4,-3)circle(1.5pt);
    \fill(4,-2)circle(1.5pt);
    \fill(8,-2)circle(1.5pt);
    \fill(5,-1)circle(1.5pt);
    \fill(7,-1)circle(1.5pt);
    \node [red] at(4,-3)[above]{$\vdots$};
    \draw[red,thick] (4.5,-3.9)--(4,-3);
    \draw[red,thick] (4,-2)--(5,-1);
    \draw[red,thick] (8,-2)--(7,-1);
    \draw[red,thick] (5,-1)--(7,-1);
    \end{tikzpicture}
\end{minipage}
     \begin{minipage}{0.25\linewidth}
    \centering
    \begin{tikzpicture}
    \draw[red,thick,->] (-0.8,0) -- (0.8,0) node[midway, above] {\tiny{\text{Merging}}};
        \end{tikzpicture}
    \end{minipage}%
      \begin{minipage}{0.4\linewidth}
         \begin{tikzpicture}[scale=0.7]
    \draw[thick,fill=black!20] (8,-4.17)arc (60:120:4);
  
    \fill(7.5,-3.9)circle(1.5pt);
    \fill(4.5,-3.9)circle(1.5pt);
    \draw[red,thick] (7.5,-3.9)--(8,-3);
    \node [red] at(8,-3)[above]{$\vdots$};
  
     \node at (7.5,-3.9) [below]{\tiny{$p_2$}};
    \node at (4.5,-3.9) [below]{\tiny{$p_1$}};
    \fill (8,-3)circle(1.5pt);
    \fill(4,-3)circle(1.5pt);
    \fill(4,-2)circle(1.5pt);
    \fill(8,-2)circle(1.5pt);
    \fill(5,-1)circle(1.5pt);
    \fill(7,-1)circle(1.5pt);
    \node [red] at(4,-3)[above]{$\vdots$};
    \draw[red,thick] (4.5,-3.9)--(4,-3);
    \draw[red,thick] (4,-2)--(5,-1);
    \draw[red,thick] (8,-2)--(7,-1);
    \draw[red,thick] (5,-1)--(7,-1);
    \end{tikzpicture}
     \end{minipage}
    \caption{Merge consecutive boundary segments}
    \label{f:merging}
 \end{figure}
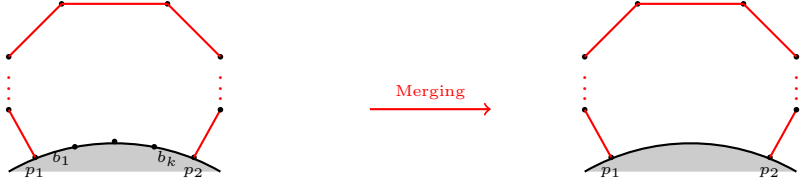


\subsubsection{Construction new tiling for  $\End_{A_\PP}(M(\mathbf{R}))$}
\begin{construction}
Let $(\SS,\MM,\PP)$ be a tiling and $\mathbf{R}$ be a faithful dissection on $(\SS,\MM,\PP)$. 
We associate to the endomorphism algebra $\End_{A_\PP} M(\mathbf{R})$ a marked surface $(\SS_{\mathbf{R}},\MM_{\mathbf{R}},\PP_{\mathbf{R}})$ defined as follows:
\begin{itemize}
    \item $\SS_{\mathbf{R}} = \SS$;
    \item $\MM_{\mathbf{R}} = \MM_{\text{new}} \setminus \MM_{\text{rem}}$;
    \item $\PP_{\mathbf{R}} = \mathbf{R}_{\text{new}}$.
\end{itemize}
\end{construction}

\begin{proposition}\label{p:A_R isomorphic to End}
The triple $(\SS_{\mathbf{R}},\MM_{\mathbf{R}},\PP_{\mathbf{R}})$ is a tiling. Moreover, there is an algebra isomorphism
\[
A_{\PP_{\mathbf{R}}} \cong \End_{A_\PP}(M(\mathbf{R}))\cong B_{\RR}.
\]
\end{proposition}

\begin{proof}
By Definition~\ref{d:definition of A_R} and the definition of tiling algebras, the result follows directly from the construction of $(\SS_{\mathbf{R}},\MM_{\mathbf{R}},\PP_{\mathbf{R}})$.
\end{proof}

\begin{ex}
We employ the example from Section 6 of \cite{HZZ}, modified by replacing the puncture with an unmarked boundary component.
 Let $(\SS,\MM,\PP)$ be a tiling as shown in Figure \ref{f:ex for tiling}.  Then $A_\PP=KQ_\PP/\langle I_\PP\rangle$, for \[Q_\PP=\xymatrix{
				\ar@(lu,ld)[]_{\epsilon_1}1\ar@/^.5pc/[r]^{a}&2\ar@/^.5pc/[l]^{b}\ar[rd]_h&3\ar[l]_c\ar[r]^d&4\ar[r]^e&5\ar@(ru,rd)[]^{\epsilon_5}\\
				&&7&6\ar[u]_f\ar[l]^g
			}\]
         $I_{\PP}=\{\epsilon_1^2,ab,ba,ch,de,\epsilon_5^2\}.$
        
    \begin{figure}[htpb]
	\centering
	\begin{subfigure}[b]{0.48\textwidth}
		\centering
		\begin{tikzpicture}[scale=0.8] 

 \fill (0,-0.5)circle(2pt);
  \fill (0,0.5)circle(2pt);
   \fill (3,0)circle(2pt);
     \fill (0,3)circle(2pt); 
      \draw(3,0)node[right]{\tiny{$p_1$}};
   \draw(0,3)node[above]{\tiny{$p_2$}};
   \draw(-0.3,-0.5)node[below]{\tiny{$p_4$}};
   \draw(0,0.5)node[above]{\tiny{$p_3$}};
   
			\draw[thick] (0,0) circle (3);
			\draw[thick,fill=gray!20] (0,0) circle (.5);
		\draw[thick,fill=gray!20](90:1.5) circle (.2);
\draw[thick,fill=gray!20](-70:1.6) circle (.1);
	\draw[thick,fill=gray!20](-70:2.5) circle (.15);
           	\draw[blue, thick] (0,-0.5) .. controls (-0.2,-1) and (0.2,-2.3)..(0.8,-2)..
            controls (1.1,-1.9) and (1.4,-1)..(0,-0.5);
         	\draw[blue, thick] (0,-0.5) .. controls (-0.3,-1) and (-0.2,-2.7)..(0.7,-2.7).. controls (1.9,-2.7) and (1.9,-1)..(0,-0.5);   
           
			\draw[blue,thick] (0,-.5) ..controls (-2,0) and (-1.5,2.5)..(0,3);
			\draw[blue,thick] (0,-.5)to[bend right=10](3,0);
			\draw[blue, thick] (0,.5) .. controls (0.8,1) and (0.8,2)..(0,2);
			\draw[blue, thick] (0,.5) .. controls (-0.8,1) and (-0.8,2)..(0,2);
			\draw[blue,thick] (0,.5)to[out=20,in=-30,looseness=1.5](0,3);
			\draw[blue,thick] (3,0)to[out=160,in=-10](0,.5);
			\draw[blue] (0,-2.4)node[below]{$2$} (0.5,-2.2)node{$1$} (-1.4,0.5)node{$3$} (1.2,2)node{$4$} (0,2.2)node{$5$} (2,.5)node{$6$}(2,-.6)node{$7$};
			
		\end{tikzpicture}
		\caption{$(\SS,\MM,\PP)$}
		\label{f:ex for tiling}
	\end{subfigure}
	\hfill 
	\begin{subfigure}[b]{0.48\textwidth}
		\centering
		\begin{tikzpicture}[scale=0.8] 
				
 \fill (0,-0.5)circle(2pt);
  \fill (0,0.5)circle(2pt);
   \fill (3,0)circle(2pt);
     \fill (0,3)circle(2pt); 
      \draw(3,0)node[right]{\tiny{$p_1$}};
   \draw(0,3)node[above]{\tiny{$p_2$}};
   \draw(-0.3,-0.5)node[below]{\tiny{$p_4$}};
   \draw(0,0.5)node[above]{\tiny{$p_3$}};
   
			\draw[thick] (0,0) circle (3);
			\draw[thick,fill=gray!20] (0,0) circle (.5);
		\draw[thick,fill=gray!20](90:1.5) circle (.2);
\draw[thick,fill=gray!20](-70:1.6) circle (.1);
	\draw[thick,fill=gray!20](-70:2.5) circle (.15);
           	\draw[blue, thick] (0,-0.5) .. controls (-0.2,-1) and (0.2,-2.3)..(0.8,-2)..
            controls (1.1,-1.9) and (1.4,-1)..(0,-0.5);
         	\draw[blue, thick] (0,-0.5) .. controls (-0.3,-1) and (-0.2,-2.7)..(0.7,-2.7).. controls (1.9,-2.7) and (1.9,-1)..(0,-0.5);   
           
			\draw[blue,thick] (0,-.5) ..controls (-2,0) and (-1.5,2.5)..(0,3);
			\draw[blue,thick] (0,-.5)to[bend right=10](3,0);
			\draw[blue, thick] (0,.5) .. controls (0.8,1) and (0.8,2)..(0,2);
			\draw[blue, thick] (0,.5) .. controls (-0.8,1) and (-0.8,2)..(0,2);
			\draw[blue,thick] (0,.5)to[out=20,in=-30,looseness=1.5](0,3);
			\draw[blue,thick] (3,0)to[out=160,in=-10](0,.5);
 \draw[orange,thick] (3,0).. controls (1.5,-1.4) and (1,-0.9) .. (0.5,-1.2) .. controls (0.1,-1.3) and (0.3,-2) .. (0.9,-1.8)  ..controls(1.2,-1.6)and (1.5,-1.6)..(3,0);

\draw[orange,thick] (3,0).. controls (1.5,-1.2) and (1,-0.9) .. (0.4,-1.05) .. controls (-0.4,-1.3) and (-0.5,-2.9) .. (0.6,-2.8)  ..controls(1.0,-2.8)and (1.67,-2.8)..(3,0);
 \draw[orange,thick] (3,0).. controls (2.2,1.6) and (0,2.9) .. (-0.8,1.8)
                   .. controls (-1.4,0.6) and (-0.9,-0.7) .. (0,-0.7)
                   .. controls (0.7,-0.7) and (1.2,0.4) .. (0,0.5);
                     \draw[orange,thick] (3,0) .. controls(2.2,0.8) and (1.5,1)..(0,1)
                   ..controls(-1,1)and (-1,-0.5)..(0,-0.5);
 \draw [orange,thick](3,0)..controls (2.2,1.4)and (1,1.9)..(-0.1,1.8)
                   ..controls (-0.6,1.8) and (-0.5,1)..(-0,1.2)
                   ..controls (0.2,1.3) and (2,1.6)..(3,0);
    \draw[orange,thick] (3,0).. controls (2,2.0) and (0,2.95) .. (-0.8,2)
          .. controls (-1.9,0.5) and (-0.8,-0.8) .. (-0.2,-0.8)
                   .. controls (0.5,-1) and (0.8,-0.5) .. (0.9,-0.1)
                    .. controls(1,0.3) and (0.6,0.8)..(0,0.8)
                   ..controls(-0.9,0.8)and (-0.9,-0.5)..(0,-0.5);
                 
   \draw[orange,thick] (3,0).. controls (2,2.9) and (0,3) .. (-1.5,2) .. controls (-3,1) and (-1.5,-0.9) .. (-0.2,-1)  ..controls(0,-1)and (1.5,-1)..(3,0);

		\end{tikzpicture}
		\caption{$(\SS,\MM,\PP,\RR)$}
		\label{f:ex for tiling and faithful dissection}
	\end{subfigure}
	\caption{The example for the tiling $(\SS,\MM,\PP)$ with faithful dissection $\RR$.}
   \end{figure}
   \begin{figure}[htpb]
	\centering
	\begin{subfigure}[b]{0.48\textwidth}
		\centering
	\begin{tikzpicture}[scale=0.8]
		\fill (0,-0.5)circle(2pt);
  \fill (0,0.5)circle(2pt);
   \fill (3,0)circle(2pt);
     \fill (0,3)circle(2pt); 
      \draw(3,0)node[right]{\tiny{$p_1$}};
   \draw(0,3)node[above]{\tiny{$p_2$}};
   \draw(-0.3,-0.5)node[below]{\tiny{$p_4$}};
   \draw(0,0.5)node[above]{\tiny{$p_3$}};
   
			\draw[thick] (0,0) circle (3);
			\draw[thick,fill=gray!20] (0,0) circle (.5);
		\draw[thick,fill=gray!20](90:1.5) circle (.2);
\draw[thick,fill=gray!20](-70:1.6) circle (.1);
	\draw[thick,fill=gray!20](-70:2.5) circle (.15);
  
 \draw[orange,thick] (3,0).. controls (1.5,-1.4) and (1,-0.9) .. (0.5,-1.2) .. controls (0.1,-1.3) and (0.3,-2) .. (0.9,-1.8)  ..controls(1.2,-1.6)and (1.5,-1.6)..(3,0);

\draw[orange,thick] (3,0).. controls (1.5,-1.2) and (1,-0.9) .. (0.4,-1.05) .. controls (-0.4,-1.3) and (-0.5,-2.9) .. (0.6,-2.8)  ..controls(1.0,-2.8)and (1.67,-2.8)..(3,0);
 \draw[orange,thick] (3,0).. controls (2.2,1.6) and (0,2.9) .. (-0.8,1.8)
                   .. controls (-1.4,0.6) and (-0.9,-0.7) .. (0,-0.7)
                   .. controls (0.7,-0.7) and (1.2,0.4) .. (0,0.5);
                     \draw[orange,thick] (3,0) .. controls(2.2,0.8) and (1.5,1)..(0,1)
                   ..controls(-1,1)and (-1,-0.5)..(0,-0.5);
 \draw [orange,thick](3,0)..controls (2.2,1.4)and (1,1.9)..(-0.1,1.8)
                   ..controls (-0.6,1.8) and (-0.5,1)..(-0,1.2)
                   ..controls (0.2,1.3) and (2,1.6)..(3,0);
    \draw[orange,thick] (3,0).. controls (2,2.0) and (0,2.95) .. (-0.8,2)
          .. controls (-1.9,0.5) and (-0.8,-0.8) .. (-0.2,-0.8)
                   .. controls (0.5,-1) and (0.8,-0.5) .. (0.9,-0.1)
                    .. controls(1,0.3) and (0.6,0.8)..(0,0.8)
                   ..controls(-0.9,0.8)and (-0.9,-0.5)..(0,-0.5);
                 
   \draw[orange,thick] (3,0).. controls (2,2.9) and (0,3) .. (-1.5,2) .. controls (-3,1) and (-1.5,-0.9) .. (-0.2,-1)  ..controls(0,-1)and (1.5,-1)..(3,0);

\draw[orange](-2,1)node[left]{\tiny{$\gamma_7$}};
 \draw[orange](-1.2,1)node[left]{\tiny{$\gamma_6$}};   
 \draw[orange](-1.2,1)node[right]{\tiny{$\gamma_3$}};
 \draw[orange](-0.3,1.5)node[left]{\tiny{$\gamma_5$}};
\draw[orange](1.5,1)node[below]{\tiny{$\gamma_4$}};
\draw[orange](0.5,-1.7)node[below]{\tiny{$\gamma_1$}};
\draw[orange](0,-1.5)node[left]{\tiny{$\gamma_2$}};
 
	\end{tikzpicture}
	\caption{$(\SS,\MM,\RR)$}\label{f:ex for faithful dissection}
	\end{subfigure}
	\hfill 
	\begin{subfigure}[b]{0.48\textwidth}
		\centering
	\begin{tikzpicture}[scale=0.8]
		\draw[thick] (0,0) circle (3);
		\draw[thick,fill=gray!20] (0,0) circle (.5);
		\draw[thick,fill=gray!20](-1,-1.8) circle (.2);
		\draw[thick,fill=gray!20](90:1.5) circle (.2);
	\fill (0,-0.5)circle(2pt);
  \fill (0,0.5)circle(2pt);
     \fill (0,3)circle(2pt); 
      \draw(0,0.5)node[right]{\tiny{$p_3$}};
   \draw(0,3)node[above]{\tiny{$v$}};
   \draw(-0.3,-0.5)node[below]{\tiny{$p_4$}};
   \draw(-60:3)node[below]{\tiny{$p_1$}};

        \draw[thick,fill=gray!20] (0,-2) circle (.1);
        \draw[thick,orange] (0,1) .. controls (.5,1.05) and (.5,2) .. (0,3);
		\draw[thick,orange] (0,1) .. controls (-.5,1.05) and (-.5,2) .. (0,3);
        \draw[thick,orange] (0,3)..controls(-.8,2) and (-.8,1)..(0,0.5);
        \draw[thick,orange] (0,3)..controls(-1.5,1.5) and (-1.5,-1)..(0,-0.5);
        \draw[thick,orange] (0,3)..controls(1.5,1.5) and (1.5,-1)..(0,-0.5);
        \draw[thick,orange] (0,3)..controls(-2,1.8) and (-2,-.6)..(-1,-.6)
                              ..controls (0,-.8) and (1,-1)..(-60:3);
        \draw [thick,orange](-.5,-1.8)..controls(-0.45,-1.3) and (0.95,-1.7)..(-60:3);
        \draw [thick,orange](-.5,-1.8)..controls(-0.45,-2.3) and (0.95,-2.3)..(-60:3);
        \draw [orange,thick](-1.6,-1.8)..controls(-1.55,-1.5)and(-0.5,.2)..(-60:3);
        \draw [orange,thick](-1.6,-1.8)..controls(-1.55,-2.8)and(-.5,-2.5)..(-60:3);
         \draw[orange] (0,-2.5)node[below]{$2$} (0,-1.5)node{$1$};
			\draw [orange] (-1.5,1.5)node{$7$}; 
			\draw[orange](0.5,1.8)node{$5$};
			\draw [orange] (-.8,1)node{$3$};
			\draw [orange] (-1,.25)node{$6$};
			\draw [orange](1.2,1)node{$4$};
       
      \end{tikzpicture}
		\caption{$(\SS_\RR,\MM_\RR,\PP_\RR)$}\label{f:ex}
	\end{subfigure}
	
	\caption{ From $(\SS,\MM,\PP,\RR)$ to $(\SS_\RR,\MM_\RR,\PP_\RR)$}
\end{figure}
 Let $\mathbf{R}=\{\gamma_1,\gamma_2,\gamma_3,\gamma_4,\gamma_5,\gamma_6,\gamma_7\}$ be a faithful dissection (\confer Figure \ref{f:ex for tiling and faithful dissection}).
By Theorem~\ref{t: tilting modules and faithful dissection},  $M(\RR)$ is a tilting module. Consequently, we obtain the triple $(\SS,\MM,\RR)$ (see Figure~\ref{f:ex for faithful dissection}). Note that $(\SS,\MM,\RR)$ is not a tiling and $A_{\RR}$ is not the endmorphism algebra of $M(\RR)$. As the arcs $6,7 \in \PP$ lie between the permissible arcs $\gamma_4$ and $\gamma_7$, by separating the permissible arcs in $\RR$, merging consecutive boundary segments, and adjusting the positions of certain permissible arcs of $\RR$, we obtain a new tiling $(\SS_\RR,\MM_\RR,\PP_\RR)$ (\confer Figure~\ref{f:ex}). Furthermore, this construction yields a gentle algebra $B_{\RR}=A_{\PP_{\RR}} = KQ_{\PP_{\RR}}/\langle I_{\PP_{\RR}}\rangle$ that is derived equivalent to $A_\PP$, where \[Q_{\PP_{\RR}}=
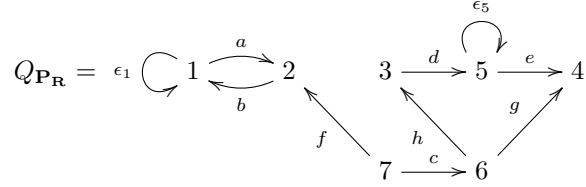
\xymatrix{\ar@(lu,ld)[]_{\epsilon_1}1\ar@/^.5pc/[r]^{a}&2\ar@/^.5pc/[l]^{b}&3\ar[r]^d&5\ar@(ul,ur)[]^{\epsilon_5}\ar[r]^e&4\\ 
				&&7\ar[lu]^{f}\ar[r]^{c}&6\ar[lu]^{h}\ar[ru]^g}\]
            with $I_{\PP_{\RR}}=\{ab,ba,\epsilon_1^2,cg,de,\epsilon_5^2\}.$


\end{ex}

    \bibliographystyle{acm}
    \bibliography{skewref}

\end{document}